\documentclass[notitlepage,reqno,11pt]{amsart}
\usepackage{latexsym,amssymb, epsfig, amsmath,amsfonts, subfigure,amsthm,dsfont}

\usepackage{rotating}
\usepackage[toc,page]{appendix}
\usepackage{color}
\usepackage{multirow}
\usepackage{relsize}
\usepackage{microtype}
\usepackage[foot]{amsaddr}

\usepackage[colorlinks, allcolors=blue]{hyperref}

\usepackage[margin=1in]{geometry}

\numberwithin{equation}{section}
 \newtheorem{assumption}{Assumption}[section]
\newtheorem{lemma}{Lemma}[section]
\newtheorem{theorem}{Theorem}[section]

\newtheorem{coro}{Corollary}[section]
\newtheorem{prop}{Proposition}[section]
\newtheorem{remark}{Remark}[section]

\newlength{\defbaselineskip}
\newcommand{\setlinespacing}[1]%
           {\setlength{\baselineskip}{#1 \defbaselineskip}}

\newcommand{\RR}{{\mathbb R}}
\newcommand{\ZZ}{{\mathbb Z}}
\newcommand{\NN}{{\mathbb N}}

\newcommand{\mfa}{\mathfrak{a}} 
\newcommand{\mfi}{\mathfrak{i}} 
\def\E{\mathbb{E}}
\def\P{\mathbb{P}}

\newcommand{\cT}{{\mathcal{T}}}
\newcommand{\sT}{{\mathfrak{T}}}

\newcommand{\sF}{{\mathcal{F}}}
\newcommand{\cI}{{\mathcal{I}}}
\newcommand{\sI}{{\mathfrak{I}}}

\newcommand{\sS}{{\mathcal{S}}}

\newcommand{\beql}[1]{\begin{equation}\label{#1}}
\newcommand{\eeq}{\end{equation}}

\newcommand{\beqal}[1]{\begin{eqnarray}\label{#1}}
\newcommand{\eeqa}{\end{eqnarray}}
\newcommand{\beq}{\begin{displaymath}}
\newcommand{\eeqno}{\end{displaymath}}
\newcommand{\bali}[1]{\begin{align}\label{#1}}
\newcommand{\eali}{\begin{align}}
\newcommand{\balino}{\begin{align*}}
\newcommand{\ealino}{\begin{align*}}

\newcommand{\ep}{\epsilon}

\newcommand{\Var}{\text{\rm Var}}

\newcommand{\R}  {\mathbb{R}}

\newcommand{\bD}{{\mathbf D}}

\newcommand{\bC}{{\mathbf C}}

\newcommand{\bone}{{\mathbf 1}}

\newcommand{\qandq}{\quad\mbox{and}\quad}

\newcommand{\qforq}{\quad\mbox{for}\quad}

\newcommand{\qasq}{\quad\mbox{as}\quad}

\newcommand{\non}{\nonumber}

\newcommand{\baa}{\begin{eqnarray*}}
\newcommand{\eaa}{\end{eqnarray*}}

\newcommand{\ttl}{\Large  
Spatially dense stochastic epidemic models  \\ [5pt] with infection-age dependent infectivity
}

\begin{document}

\title[FLLN for spatially dense  non-Markovian epidemic models ]{\ttl}

\author[Guodong \ Pang]{Guodong Pang$^*$}
\address{$^*$Department of Computational Applied Mathematics and Operations Research,
George R. Brown School of Engineering,
Rice University,
Houston, TX 77005}
\email{gdpang@rice.edu}

\author[{\'E}tienne \ Pardoux]{{\'E}tienne Pardoux$^\dag$}
\address{$^\dag$Aix Marseille Univ, CNRS,  I2M, Marseille, France}
\email{etienne.pardoux@univ.amu.fr}

\date{\today}

\begin{abstract} 
We study an  individual-based stochastic spatial epidemic model where the number of locations and the number of individuals at each location both grow to infinity.
Each individual is associated with a random  infection-age dependent infectivity function.
Individuals are infected through interactions across the locations  with heterogeneous effects. 
The epidemic dynamics can be described using a time-space representation for the the total force of infection, the number of susceptible individuals, the number of  infected individuals that are infected at each time and have been infected for a certain amount of time, as well as the number of recovered individuals.
We prove a functional law of large numbers for these time-space processes, and in the limit, we obtain a set of time-space integral equations. 
We then derive the PDE models from the limiting  time-space integral equations, in particular, the density (with respect to the infection age) of the time-age-space integral equation for the number of infected individuals tracking the age of infection satisfies a linear PDE in time and age with an integral boundary condition. These integral equation and PDE limits can be regarded as dynamics on graphon under certain conditions.

\end{abstract}


\maketitle



\allowdisplaybreaks

 \section{Introduction}

In order to capture the geographic heterogeneity, spatial epidemic models have been well developed,  both in  discrete and continuous spaces.
In discrete space, multi-patch epidemic models have been studied in \cite{sattenspiel1995structured,arrigoni2002limits,allen2007asymptotic,xiao2014transmission,bichara2018multi,nzipardouxyeo} and recently by the authors \cite{PP-2020b},  where each patch represents a geographic location, and infection may occur within each patch and from the distance (for example, due to short travels). See also the multi-patch multi-type epidemic models in \cite{bichara2018multi,FPP2021-MPMG}, as well as relevant models in  \cite{ball2008network,magal2016final,magal2018final}. 
Some of these studies assume migration of individuals among different patches \cite{sattenspiel1995structured,allen2007asymptotic,nzipardouxyeo,PP-2020b,FPP2021-MPMG}, while others do not but assume interactions between patches to induce infection \cite{arrigoni2002limits,bichara2018multi,xiao2014transmission,magal2016final,magal2018final}.
In continuous space,  various PDE models have been developed (see the monographs \cite{RR2003book,martcheva2015,BCF-2019} and a survey \cite{ruan2007spatial}). 
There are two well--known models without spatial movement: Kendall's spatial model  \cite{kendall1957, kendall1965} and Diekmann-Thieme's PDE model   \cite{diekmann1978thresholds,diekmann1979run,thieme1977asymptotic,thieme1977model}. 
Kendall's spatial model is a system of ODEs with a spatial parameter (without spatial partial derivative).
 It was proved to be the {functional law of large numbers (FLLN) limit of the multitype Markovian SIR model by Andersson and Djehiche  \cite{andersson1995limit}, where both the number of types and the population size go to infinity, and being ``Markovian" refers to the case of exponentially distributed infection durations.
Diekmann-Thieme's spatial PDE model (with partial derivatives with respect to time and infection-age) has the infection rate depending on the age of infection, as in the PDE model first proposed by Kermack and McKendrick in their 1932 paper \cite{KM32}. 
Similar to Kendall's spatial model, there is no partial derivative with respect to the spatial parameter, since there is no movement in space.
The Diekmann-Thieme PDE model 
was not yet proved  to be the FLLN limit of a non-Markovian stochastic epidemic model (in which the infectious durations have a general distribution), and can be seen as 
is a special case of our FLLN limit, which is new. 
We should also mention the spatial models in continuous space in  \cite{bowongemakouapardoux} and  \cite{vuong2021conditional}, where the stochastic model starts with a continuous process for the movement of individuals, in particular, it is assumed that individual movements follow an It{\^o} diffusion process, and the epidemic models are Markovian.

In this paper, we start with an individual-based stochastic SIR epidemic model at a finite number of locations. 
The individuals in the each location are grouped into ``Susceptible", ``Infected" and ``Recovered" compartments.
Each susceptible individual at every location may be infected from his or her own location or from other locations (see the infection rate function in equation \eqref{eqn-upsilon}).
Note that individuals do not migrate from one location to another in our model. 
Each individual is associated with a random infectivity function/process, independent from any other individual but having the same law as all the other individuals.
This random infectivity function also determines the law of the infectious duration of each individual.
Those random functions are i.i.d. for all individuals. 
For each individual, we track the age of infection, that is, the elapsed time since the individual was infected (for the initially infected individuals, this means we also know their infection times before time zero). 
To describe the epidemic dynamics at each location, we use the aggregate infectivity process of the population and a two-parameter (equivalently, measure-valued) process tracking the number of individuals that have been infected for less than or equal to a certain amount of time as well as the numbers of susceptible and recovered individuals. Such an individual-based stochastic model with only one location  has been studied by the authors in \cite{PP-2021}, where an FLLN is established and the associated PDE model for the limit is derived.  In our previous works of large population scaling limits for stochastic epidemic models (see the survey \cite{FPP-survey}), most models consider a homogeneous population with the two exceptions of a multi-patch (discrete space) model \cite{PP-2020b,FPP2021-MPMG}. Our model in this paper starts from a dense discrete space model, while the limit as both the size of the population and the number of patches/locations tend to infinity simultaneously is a deterministic spatial model in continuous space. In particular, the PDE model includes the Diekmann-Thieme spatial model as a special case (see Remarks \ref{rem-special} and \ref{rem-Diekman-Kendall}). 

We consider this stochastic epidemic model in a spatially dense setting, where the number of locations increases to infinity while the number of individuals in each location (and the total population) also goes to infinity. This has the same flavor as  the asymptotic regime in   \cite{andersson1995limit} for the 
multitype Markovian SIR model where the number of types goes to infinity while the population in each type also go to infinity.
It is worth mentioning the paper \cite{arrigoni2002limits} in which a measure-valued limit is proved for a multi-patch Markovian SIS epidemic model without migration in the asymptotic regime with both the number of patches and the number of individuals in each patch going to infinity. 
This is also in a similar fashion as the asymptotic regime of the Markovian SIR epidemic model with migration on a refining spatial grid in $\R^d$ ($d=1,2,3$), recently studied in \cite{nzipardouxyeo}, where the mesh of the grid goes to zero and the population size at each site also goes to infinity. In the limit of that model, a Laplace operator describes the spatial movement in the  time-space dynamics. 
 Unlike these works under Markovian assumptions, our model is non-Markovian and has an infection process with the infection-age dependent infectivity, which brings new mathematical challenges. 


For this model, it is convenient to describe the epidemic dynamics at all locations using a time-space representation of the vector-valued processes (for the number of infected individuals tracking the age of infection, this in fact becomes a time-age-space process). 
We prove an FLLN (Theorem \ref{thm-FLLN}) for the scaled time-space processes under a set of regularity conditions on the initial conditions, infection contact rates and random infectivity functions (Assumptions \ref{AS-LLN-1}, \ref{AS-LLN-2} and \ref{AS-lambda}). 
The limits in the FLLN are described by a set of time-space integral equations. 
It is worth highlighting that the heterogeneity of interaction effects between different locations is represented by a  function $\beta(x,y)$ for $x,y\in [0,1]$ (which resembles the kernel function of a graphon, see further discussions below).

For the weak convergence of the time-space processes, we introduce new weak convergence criteria for these time-space processes (Theorems \ref{thm-D-conv-x} and \ref{thm-DD-conv-x}), which involves the $L_1$ norm for the spatial component. 
To verify these criteria, we establish moment estimates for the increments of these processes, which is challenging due to the interactions among the individuals at the different locations. In particular, the interactions introduce nontrivial dependence in various components of the time-space processes. 
We first study the joint time-space dynamics of the susceptible population and total force of infection (Section \ref{sec-proof-conv-S-F}). 
This involves the existence and uniqueness of solution to a set of time-space Volterra-type integral equations (see equations \eqref{eq:SF}-\eqref{eqn-bar-mfF-0}), and the moment estimates associated with the increments involving the varying infectivity functions together with their interactions (in order to use Theorem \ref{thm-D-conv-x}). Given this convergence, we then establish the convergence of the time-age-space process tracking the infection ages of individuals (Section \ref{sec-proof-conv-I}). In order to employ Theorem \ref{thm-DD-conv-x}, we need to establish the moment estimates for the increments with respect to both time and infection-age parameters, for which the dependence due to interactions also brings additional challenges.

From the limit tracking the  rescaled number of infected individuals with a  given age of infection, we derive a PDE model with partial derivatives with respect to time and the age of infection (not with respect to the  spatial variable, since there is no migration among locations). It is a linear PDE model with an integral boundary condition. 
It may be seen as an extension of the PDE models in  \cite{PP-2021}, with the addition of a spatial component. 
We then discuss how the PDE model is related to the well-known Diekmann-Thieme PDE model and how it reduces to Kendall's PDE model in the Markovian case (see Remarks \ref{rem-special} and \ref{rem-Diekman-Kendall}). Note that our PDE model is more general since we do not require any condition on the distribution function of the infectious periods.


Our work also contributes to the recent studies of stochastic dynamics on graphon. 
Keliger et al. \cite{keliger2022local} consider a finite-state Markov chain with local density-dependence on a discretized graph of a graphon, and then prove an FLLN for the Markovian time-space dynamics. Their model includes a Markovian SIS model on graphon, 
and since each individual is a node on the sampled graph and naturally there is no spatial movement, the limit is in fact a system of ODEs without spatial partial derivative.   There is some resemblance between that limit and our PDE model for the Markovian SIS model, see further discussions in Remark \ref{rem-SIS-PDE}, although it is important to note that in our stochastic multi-patch model, the number of individuals in each patch also goes to infinity while in the stochastic model on the sampled graph from a graphon in \cite{keliger2022local}, there is only one individual in each node of the graph.
Petit et al. \cite{petit2021random} consider a random walk on graphon and prove an LLN limit for the Markovian time-space dynamics, which is again a system of ODEs  without spatial partial derivative. 
However, we start with a non-Markovian multi-patch epidemic dynamics, and the limiting integral equations in Theorem \ref{thm-FLLN} and the PDE models in Proposition \ref{prop-PDE-g}
and Corollaries \ref{coro-PDE-ac} and \ref{coro-PDE-det} can be regarded as dynamics on graphon, when the kernel function $\beta(x,y)$ is symmetric and takes values in $[0,1]$ (see further discussions in Remark \ref{rem-graphon}).


\subsection{Organization of the paper}
The paper is organized as follows. In Section \ref{sec-model}, we provide the detailed model description. We then present the scaled processes and assumptions and state the FLLN result in Section \ref{sec-FLLN}. 
We derive the PDE models from the FLLN limits and discuss how they are related to the already known spatial PDE models in Section \ref{sec-PDE}. 
The proofs of the FLLN are given in Sections 
\ref{sec-proof-conv-S-F} and \ref{sec-proof-conv-I} after some technical preliminaries in Section \ref{sec-technical}.

\subsection{Notation}
 All random variables and processes are defined on a common complete probability space $(\Omega, \sF, \P)$. Throughout the paper, $\NN$ denotes the set of natural numbers, and $\RR^k (\RR^k_+)$ denotes the space of $k$-dimensional vectors
with  real (nonnegative) coordinates, with $\RR (\RR_+)$ for $k=1$.  
Let $\bD=\bD(\RR_+;\RR)$ denote the space of $\RR$--valued c{\`a}dl{\`a}g functions defined on $\RR_+$. Here, convergence in $\bD$ means convergence in the  Skorohod $J_1$ topology, see Chapter 3 of \cite{billingsley1999convergence}. 
 Let $\bC$ be the subset of $\bD$ consisting of continuous functions.   
  Let $\bD_\bD= \bD(\RR_+; \bD(\RR_+;\RR))$ be the $\bD$-valued $\bD$ space, and the convergence in the space $\bD_\bD$ means that both $\bD$ spaces are endowed with the Skorohod $J_1$ topology. 
For any nondecreasing and bounded  c{\`a}dl{\`a}g function $G(\cdot): \R_+\to \R_+$, abusing notation, we write $G(dx)$ by treating $G(\cdot)$ as the positive (finite) measure on 
$\R_+$ whose distribution function is $G$. 
For any $\RR$--valued c{\`a}dl{\`a}g function $\phi(\cdot)$ on $\R_+$, the integral $\int_{a}^b \phi(x)G(dx)$ represents $\int_{(a,b]} \phi(x) G(dx)$ for $a<b$.  
We use $\bone_{\{\cdot\}}$ for the indicator function.
For $x,y \in\RR$, we denote $x\wedge y = \min\{x,y\}$ and $x\vee y = \max\{x,y\}$.

We use $\|\cdot\|_1$ to denote the $L^1([0,1])$ norm. 
For time-space processes $Z(t,x)$ and $Z(t,s,x)$,  for each $x$, we regard them in the spaces $\bD$ and $\bD_\bD$, respectively.  
For the weak convergence of the time-space processes $Z^N(t,x)$ to $Z(t,x)$ as $N\to\infty$, we use the Skorohod topology for the processes in $\bD$ with  the  $L^1([0,1])$ norm with respect to $x$.   Similarly, for the weak convergence of the time-space processes $Z^N(t,s,x)$ to $Z(t,s,x)$ as $N\to\infty$, we use the Skorohod topology for the processes in $\bD_\bD$ with  the  $L^1([0,1])$ norm with respect to $x$.
We write these spaces as $\bD(\R_+, L^1([0,1]))$ and  $\bD(\R_+, \bD(\R_+, L^1([0,1]))$, or  $\bD(\R_+, L^1)$ and  $\bD(\R_+, \bD(\R_+, L^1))$ for short. 
See the weak convergence criteria in Theorems \ref{thm-D-conv-x} and \ref{thm-DD-conv-x}.

\bigskip

\section{Model and FLLN} \label{sec-model-FLLN}

\subsection{Model Description} \label{sec-model}

We consider a population of fixed size $N$ distributed in $K$ locations 
in some bounded domain $\sS$ in $\RR^d$ ($d\ge1$). To be specific, we choose $\sS=[0,1]$. 
The arguments in the paper would remain the same for any such a domain $\sS$ since we do not consider migration among locations. 
 Also let $K$ depend on $N$, denoted as $K^N$. 
Let the $K^N$ locations  be at $x^N_k,\, k=1,\dots,K^N$ in $[0,1]$ such that $0 \le x^N_1< x^N_2<\cdots<x^N_{K^N} \le 1$. For notational convenience, let $\mathtt{I}^N_k, k=1,\dots,K^N$
be a partition of $[0,1]$ such that $x^N_k\in \mathtt{I}^N_k$ and   $|\mathtt{I}^N_k|=(K^N)^{-1}$ for all $1\le k\le K^N$. 
In each location, individuals are categorized into three groups: susceptible, infected (possibly including both exposed and infectious) and recovered. 
We assume that individuals do not move among the different locations, and 
susceptible individuals in each location can be infected from their own location as well as from other locations (as explained below). 
Suppose that there are $B^N_k$ individuals at location $x_k^N$, such that $B^N_1+\cdots+B^N_{K^N}=N$. 
(For example, there is an equal number of individuals in each path, that is, $B^N_k = N/K^N$ for all $k$.)
We assume that
\begin{align} \label{eqn-KB-condition} 
\text{both } K^N\to\infty& \qandq \frac{N}{K^N}\to\infty, \qasq N\to\infty\,. 
\end{align}

Notation: Whenever not causing any confusion, we drop the superscript $N$ in $x^N_k$, $\mathtt{I}^N_k$, $K^N$ and $B^N_k$.  
For any vector ${\bf z}=(z_1,\dots,z_K)$, we write $z(x)=\sum_{k=1}^K z_k {\bf 1}_{\mathtt{I}^N_k}(x)$ where ${\bf 1}_{\mathtt{I}^N_k}(\cdot)$ denotes the indicator function of the set $\mathtt{I}^N_k$. For a process ${\bf Z}(t)=(Z_1(t),\dots,Z_K(t))$, we write $Z(t,x) = \sum_{k=1}^K Z_k(t) {\bf 1}_{\mathtt{I}^N_k}(x)$ for $t \ge 0, x \in [0,1]$. 

Let $S^N_k(t)$, $I^N_k(t)$ and $R^N_k(t)$ be the numbers of susceptible, infected and recovered individuals in location $x_k$ at time $t\ge 0$. We clearly have $B^N_k = S^N_k(t)+I^N_k(t)+R^N_k(t)$ for each $t\ge 0$. 
We can also write the vectors ${\bf S}^N(t)=(S^N_1(t),\dots,S^N_K(t))$, ${\bf I}^N(t)=(I^N_1(t),\dots,I^N_K(t))$ and ${\bf R}^N(t)=(R^N_1(t),\dots,R^N_K(t))$, as the following time-space processes $S^N(t,x) = \sum_{k=1}^K S^N_k(t){\bf 1}_{\mathtt{I}_k}(x)$, 
$I^N(t,x) = \sum_{k=1}^K I^N_k(t){\bf 1}_{\mathtt{I}_k}(x)$ and $R^N(t,x) = \sum_{k=1}^K R^N_k(t){\bf 1}_{\mathtt{I}_k}(x)$, respectively. Note that $S^N_k(t) = S^N(t,x_k)$, and so on.

To each infected individual is attached  a random infectivity function. 
Individual $j$ in location $x_k$ has a random infectivity function $\lambda_{j,k}(\cdot)$.
We  assume that 
 $ \lambda_{j, k}(t) =0$ a.s. for $t<0$, 
 for all $j\in\ZZ \backslash\{0\}$, $ k=1,\dots, K$,  and that each $\lambda_{j,k}$ has paths in $\bD$. 
 We assume that the sequence $\{\lambda_{j,k}: j \in \ZZ \backslash \{0\},k =1,\dots,K \}$ is i.i.d.,
with $j \ge 1$ indexing newly infected individuals after time 0 and $j \le -1$ indexing the initially infected ones at time 0 (that is, those infected before time 0). 
 We use $\lambda(\cdot)$ as a generic function to denote them. 
 Let $\bar{\lambda}(t) = \E[\lambda(t)]$ and $v(t) =\Var(\lambda(t)) = \E\big[\big(\lambda(t) - \bar\lambda(t)\big)^2\big]$  for $t\ge 0$. 

Define $\eta_{j,k} = \sup\{t>0: \lambda_{j,k}(t)>0\}$, which represents the duration of the infected period for individual $j$. 
Note that this may include both the exposed and infectious periods for which the function $\lambda_{j,k}(t)$ start with being zero in the exposed period. Under the above assumption on $\{\lambda_{j,k}\}$, the variables $\{\eta_{j,k} \}$ are also i.i.d. 
Let $F(t) = \P(\eta_{j,k} \le t)$ for $j \in \ZZ \backslash \{0\}$ and $k =1,\dots, K$, representing the cumulative distribution function (c.d.f.) for the duration of the infected period. Define $F^c :=1-F$. 

Let $\{\tau^N_{j,k}: j \in \ZZ\setminus\{0\}\}$ be the associated infection times for each individual. 
Evidently, $\tau^N_{j,k}\ge 0$ for $j \ge 1$ and $\tau^N_{j,k}<0$ for $j \le -1$. Let $\tilde{\tau}^N_{j,k}=-\tau^N_{j,k}>0$ for $j \le -1$,  which denotes the time elapsed since infection, i.e. the infection age at time 0 for the initially infected individuals. 
The counting process $A^N_k(t) =\max\{ j\ge 1:  \tau^N_{j,k} \le t \}$ represents the number of newly infected individuals in location $k$ over $(0,t]$.
Some initially infected individuals may have recovered by time 0, that is, if $\eta_{j,k} \le \tilde{\tau}^N_{j,k}$ for $j \le -1$, then the individual $j$ is in the compartment of the removed, $\mathcal{R}^N_k(0)$}. However, if  $\eta_{j,k} > \tilde{\tau}^N_{j,k}$ for $j \le -1$, then the individual $j$ remains infected at time 0 and belongs to the compartment of the infected, $\mathcal{I}^N_k(0)$.  Then we have $R^N_k(0) = |\mathcal{R}^N_k(0)|$, $I^N_k(0) = |\mathcal{I}^N_k(0)|$ and $S^N_k(0) = B^N_k - I^N_k(0)  -  R^N_k(0)$ for each $k$.  Let $\cT^N_k(0) = \mathcal{I}^N_k(0) \cup \mathcal{R}^N_k(0)$.
The  process \begin{equation} \label{eqn: sT-Nk}
\sT^N_k(0,\mfa) = \sum_{j:-j \in  \cT^N_k(0) } {\bf1}_{ \tilde{\tau}^N_{-j,k} \le \mfa } \end{equation} 
represents the number of initially infected individuals 
with an infection age  less than or equal to $\mfa$ at time 0,  which include those that remain infected and those that are recovered. We assume that for each $k=1,\dots,K$, the sequence $\{\tau^N_{-j,k}: - j\in \cT^N_k\}$ is independent of the sequence 
$\{\lambda_{-j,k}: - j\in \cT^N_k \}$. We remark that this independence assumption may not be natural, since the future event times may depend on the value of $\lambda_{-j,k}$; however, this assumption is essential for the proofs, and the sources of initial infections may differ from the new infections (such as migration). 
Let 
$\eta^0_{j,k}=\sup\{t>0: \, \lambda_{j,k}(\tilde{\tau}_{j,k}+t)>0\}$
 be the associated remaining infected period.  It is clear that $\eta^0_{j,k} = \eta_{j,k} - \tilde{\tau}_{j,k}^N$. Then, for $j \le -1$,
 the conditional distribution of $\eta^0_{j,k}$ given that $ \eta_{j,k} > \tilde{\tau}_{j,k}^N=s>0$ is 
\begin{align} \label{enq-eta0-age}
\P\big(\eta^0_{j,k}> t \big| \eta_{j,k}>\tilde{\tau}_{j,k}^N=s\big) =\P\big(\eta_{j,k} - \tilde{\tau}_{j,k}^N > t \big| \eta_{j,k}>\tilde{\tau}_{j,k}^N=s\big) = \frac{F^c(t+s)}{F^c(s)}, \qforq t, s >0. 
\end{align}
Note that conditional on $\{\tilde{\tau}_{j,k}\}$, 
 the $\eta^0_{j,k}$'s are independent but not identically distributed.

The total force of infection of the infected individuals in location $k$ is given by 
\begin{equation} \label{eqn-mfk}
\mathfrak{F}_k^N(t) = \sum_{j: -j \in \cT^N_k(0)} \lambda_{-j,k}(\tilde{\tau}^N_{-j,k}+ t) + \sum_{j=1}^{S^N_k(0)} \lambda_{j,k} (t-\tau^N_{j,k}), \quad t \ge 0. 
\end{equation}
Note that since $\lambda_{j,k} (t) =0$ for $t> \eta_{j,k}$,   those that are initially infected but recovered do not contribute to the total force of infection. Then the first summation is equal to the summation over  $j$ such that $-j \in  \mathcal{I}^N_k(0)$. 
We similarly write the time-space process for the total force of infection in the population: 
$$\mathfrak{F}^N(t,x) = \sum_{k=1}^{K^N} \mathfrak{F}_k^N(t){\bf 1}_{\mathtt{I}_k}(x).$$

The rate of infection for individuals in location $k$ is given by
\begin{equation}  \label{eqn-upsilon}
\Upsilon^N_k(t) = \frac{S^N_k(t)}{B^N_k} \frac{1}{K^N} \sum_{k'=1}^{K^N}\beta^N_{k,k'} \mathfrak{F}_{k'}^N(t), \quad t \ge 0. 
\end{equation}
Here the factor $\beta^N_{k,k'}$ reflects the effect of infection of individuals from location $k'$ upon those in location $k$.
It also represents the heterogeneity of the effects of the interactions among different locations. 

The number of newly infected individuals in location $k$ by time $t$, $A^N_k(t)$, can be expressed as 
\begin{equation}\label{eqn-An-k-rep}
A^N_k(t) = \int_0^t \int_0^\infty {\bf 1}_{u \le \Upsilon^N_k(s) } Q_{k}(ds, d u),
\end{equation} 
where  $\{Q_k(ds,du),\ 1\le k\le K\}$ are mutually independent  standard  (i.e., with mean measure the Lebesgue measure) Poisson random measures (PRMs) on $\RR^2_+$.
Recall that $\{A^N_k(t): t\ge 0\}$ has the event times $\{\tau^N_{j,k}, j \ge 1\}$.

Let $\sI^N_k(t,\mfa)$ be the number of infected individuals in location $k$ that are infected at time $t$ and have been infected for less than or equal to $\mfa$. Then we can write
\begin{align} \label{eqn-In-k-rep}
\sI^N_k(t,\mfa) = \sum_{j: -j \in \cI^N_k(0)} {\bf1}_{\eta^0_{-j,k} >t} {\bf 1}_{ \tilde{\tau}^N_{-j,k} \le (\mfa-t)^+}  + \sum_{j=A^N_k((t-\mfa)^+)+1}^{A^N_k(t)} {\bf1}_{\tau^N_{j,k} + \eta_{j,k} >t}\,\,. 
\end{align} 
Note that $\sI^N_k(0,\mfa)$ (the first term on the right hand side) differs from $\sT^N_k(0,\mfa)$ since it only counts the initially infected individuals that remain infected at time 0. 
It is also clear that  for all $t\ge0$, 
$$
I^N_k(t)  =  \sI^N_k(t, \infty). 
$$

To account for the location, we also write the time-age-space process
$$
\sI^N(t,\mfa,x) =  \sum_{k=1}^{K^N} \sI^N_k(t,\mfa) {\bf 1}_{\mathtt{I}_k}(x). 
$$
Note that for each $x$, the process $\sI^N(t,\mfa,x) $ has paths in $\bD_\bD$. 
The quantity $\sI^N(0,\mfa,x) $,  corresponding to $\sI^N_k(0,\mfa)$, represents the infection-age distribution of the initially infected individuals at the different locations. It is given as input data for our model, satisfying the condition in Assumption \ref{AS-LLN-1} below.

The dynamics of $S^N_k(t)$, $I^N_k(t)$ and $R^N_k(t)$ can be expressed as 
\begin{align*}
S^N_k(t) &= S^N_k(0) -  A^N_k(t), \\
I^N_k(t) &= \sum_{j: -j \in \cI^N_k(0)} {\bf 1}_{\eta^0_{-j,k} >t}  + \sum_{j=1}^{A^N_k(t)} {\bf 1}_{\tau^N_{j,k} + \eta_{j,k} >t}\,, \\
R^N_k(t) &= R^N_k(0)+  \sum_{j: -j \in \cI^N_k(0)} {\bf 1}_{\eta^0_{-j,k} \le t}  + \sum_{j=1}^{A^N_k(t)} {\bf 1}_{\tau^N_{j,k} + \eta_{j,k} \le t}\,. 
\end{align*}

\subsection{FLLN} \label{sec-FLLN}
We recall that 
\begin{equation} \label{eqn-N-B-1}
N = \sum_{k=1}^{K^N} (S^N_k(t)+ I^N_k(t) + R^N_k(t) ) = \sum_{k=1}^{K^N} B^N_k \,,
\end{equation}
and observe that 
\begin{equation} \label{eqn-N-B-2}
\int_0^1 (S^N(t,x)+ I^N(t,x) + R^N(t,x) ) dx =  \frac{1}{K^N}  \sum_{k=1}^{K^N} (S^N_k(t)+ I^N_k(t) + R^N_k(t) )= \frac{N}{K^N} \,. 
\end{equation}
It is then reasonable to introduce the scaling of the processes by $N/K^N$, that is, for any process $Z^N_k = \mathfrak{F}_k^N, \sI^N_k, \Upsilon^N_k,  A^N_k,  S^N_k,  I^N_k,  R^N_k$, we define $\bar{Z}^N_k= (N/K^N)^{-1} Z^N_k$.
We then define the scaled time-space processes  
\[
 \bar{Z}^N(t,x)=  \sum_{k=1}^{K^N} \bar{Z}^N_k(t) {\bf 1}_{\mathtt{I}_k}(x), \quad Z^N_k = \mathfrak{F}_k^N, \Upsilon^N_k,  A^N_k,  S^N_k,  I^N_k,  R^N_k
\]
and
\[
 \bar{\sI}^N(t,\mfa,x)= \sum_{k=1}^{K^N}\bar{\sI}_k^N (t,\mfa){\bf 1}_{\mathtt{I}_k}(x)\,. 
\]
In addition, define the scaled population size at each location \[\bar{B}^N(x) = \sum_{k=1}^{K^N}\bar{B}_k^N {\bf 1}_{\mathtt{I}_k}(x), \quad \text{with} \quad \bar{B}_k^N =  (N/K^N)^{-1} B^N_k.\]
Hence, from \eqref{eqn-N-B-2} and the scaling, we obtain
\[
\int_0^1 (\bar{S}^N(t,x)+ \bar{I}^N(t,x) + \bar{R}^N(t,x) ) dx= \int_0^1 \bar{B}^N(x) dx=1\,.
\]



We make the following assumption on the initial condition.


\begin{assumption} \label{AS-LLN-1}
There exist nonnegative deterministic functions $(\bar{S}(0,x), \bar{\sT}(0,\mfa, x), \bar{R}(0,x))$ such that 
for each $x$, $\bar{\sT}(0,\cdot, x)$ is in $\bC$, and
 for each $\mfa \in [0, \infty]$, 
\begin{align}\label{eqn-initial-L1conv}
\|\bar{S}^N(0,\cdot) - \bar{S}(0,\cdot)\|_{1} \to 0, \quad \ \|\bar{\sT}^N(0,\mfa, \cdot) - \bar{\sT}(0,\mfa, \cdot)\|_{1} \to 0, \quad \|\bar{R}^N(0,\cdot) - \bar{R}(0,\cdot)\|_{1} \to 0 
\end{align}
in probability as $N\to\infty$.
This implies that $\|\bar{\sI}^N(0,\mfa, \cdot) - \bar{\sI}(0,\mfa, \cdot)\|_{1} \to 0$ in probability as $N\to\infty$, where $ \bar{\sI}(0, d \mfa, x) =  \bar{\sT}(0,d \mfa, x) F^c(\mfa)$; see Lemma \ref{lem:sIsT} below. 
In addition, letting $\bar{I}(0,x) = \bar{\sI}(0,\infty, x)$, we have
\begin{equation} \label{eqn-initial-integral}
\int_0^1 (\bar{S}(0,x)+\bar{I}(0,x)+\bar{R}(0,x)) dx =1. 
\end{equation}
There exists $\bar{B}(x)$ such that 
\begin{equation}\label{cvinfty}
\|\bar{B}^N(\cdot) - \bar{B}(\cdot)\|_{\infty} = \sup_{x\in [0,1]} |\bar{B}^N(x) - \bar{B}(x) | \to 0,
\end{equation}
where for some constants $0<c_B<C_B<\infty$, 
\begin{equation} \label{eqn-barB-condition} 
\bar{B}(x) \in  [c_B, C_B] \quad \forall x\in[0,1], \end{equation}
and 
\[
 \int_0^1 \bar{B}(x) dx =1\,. \]
Moreover, the following identity holds:
 \begin{equation} \label{eqn-barB-identity} 
 \bar{B}(x) = \bar{S}(0,x)+\bar{I}(0,x)+\bar{R}(0,x), \quad \forall x\in[0,1]. 
 \end{equation}
 
\end{assumption}
Note that, thanks to \eqref{cvinfty} and \eqref{eqn-barB-condition},  we may and do assume that $c_B$ and $C_B$ have been chosen in such a way that for some $N_0$,
\begin{equation}\label{eqn-CB}
\bar{B}^N(x) \in  [c_B, C_B] \quad \forall N\ge N_0,  \, x\in[0,1]. 
\end{equation} 


Under the assumption in \eqref{eqn-initial-L1conv}, it follows that 
\begin{align*}
\|\bar{I}^N (0,\cdot) - \bar{I}(0,\cdot) \|_{1} \to 0 
\end{align*}
in probability as $N\to \infty$. 

\begin{lemma} \label{lem:sIsT}
Under the assumption in \eqref{eqn-initial-L1conv}, 
 $\|\bar{\sI}^N(0,\mfa, \cdot) - \bar{\sI}(0,\mfa, \cdot)\|_{1} \to 0$ in probability as $N\to\infty$, where $ \bar{\sI}(0, d \mfa, x) =  \bar{\sT}(0,d \mfa, x) F^c(\mfa)$.
\end{lemma}
\begin{proof}
By \eqref{eqn: sT-Nk}, 
\[
\bar{\sT}^N(0,d\mfa,x) = \frac{K^N}{N}\sum_{k=1}^{K^N}\sum_{j:-j \in  \cT^N_k(0) } \delta_{ \tilde{\tau}^N_{-j,k}}  (d \mfa ) {\bf 1}_{\mathtt{I}_k}(x) \,.
\]
By \eqref{eqn-In-k-rep}, 
\begin{align*}
 \bar{\sI}^N(0,d \mfa, x) & =  \frac{K^N}{N}\sum_{k=1}^{K^N} \sum_{j: -j \in \cT^N_k(0)} {\bf1}_{\eta_{-j,k} >  \tilde{\tau}^N_{-j,k}} \delta_{ \tilde{\tau}^N_{-j,k}}(d \mfa){\bf 1}_{\mathtt{I}_k}(x) \\
 &=  \frac{K^N}{N}\sum_{k=1}^{K^N} \sum_{j: -j \in \cT^N_k(0)} \big({\bf1}_{\eta_{-j,k} > \mfa} - F^c( \mfa) \big)\delta_{ \tilde{\tau}^N_{-j,k}}(d \mfa){\bf 1}_{\mathtt{I}_k}(x) \\
 & \quad + \frac{K^N}{N}\sum_{k=1}^{K^N} \sum_{j: -j \in \cT^N_k(0)}  F^c( \mfa) \delta_{ \tilde{\tau}^N_{-j,k}}(d \mfa){\bf 1}_{\mathtt{I}_k}(x) \\
 & :=  \bar{\sI}^N_0(0,d \mfa, x) +  \bar{\sI}^N_1(0,d \mfa, x)\,. 
 \end{align*}
 By the convergence $ \|\bar{\sT}^N(0,\mfa, \cdot) - \bar{\sT}(0,\mfa, \cdot)\|_{1} \to 0$ in probability, and for each $x$, $\bar{\sT}(0,\cdot, x)$ is in $\bC$, we immediately have the convergence  $\|\bar{\sI}^N_1(0,\mfa, \cdot) - \bar{\sI}(0,\mfa, \cdot)\|_{1} \to 0$ in probability as $N\to\infty$. 
It is easy to show that  $\|\bar{\sI}^N_0(0, \mfa, \cdot) \|_1\to 0 $ in probability by the independence of the sequences $\{ \tilde{\tau}^N_{-j,k}\}_{j}$ and $\{\lambda_{-j,k}\}_{j}$ for each $k$ (see a similar argument in the proof of Lemma \ref{lem-mfN0-conv} below). 
\end{proof}



%


We introduce for each $x,x'\in [0,1]$,
\begin{equation} \label{eqn-betaN-xx'} 
\beta^N(x,x') = \sum_{k,k'} \beta^N_{k,k'} {\bf 1}_{\mathtt{I}_k}(x)  {\bf 1}_{\mathtt{I}_{k'}}(x')\,. 
\end{equation}

%

\begin{assumption} \label{AS-LLN-2}
There exists a constant $C_\beta>0$ such that for all $N\ge1$, $x \in[0,1]$,
\begin{equation}\label{eqn-C-beta}
\int_0^1\beta^N(x,y)dy \vee\int_0^1\beta^N(y,x)dy\le C_\beta\, .
\end{equation} 
There exists a function $\beta: [0,1]\times [0,1]\mapsto \R_+$ such that
for any bounded measurable function $\phi: [0,1]\mapsto \RR$, 
\begin{align}\label{conv-beta}
\left\| \int_0^1[\beta^N(\cdot,y)-\beta(\cdot,y)]\phi(y)dy\right\|_1\to0\,.
\end{align}
\end{assumption}

\begin{remark}
Concerning condition \eqref{eqn-C-beta}, let us first note that,  if $\beta^N_{k,k'}=\beta^N_{k',k}$ (symmetric) 
for all $N\ge1,\ 1\le k,k'\le K$, the boundedness of $\int_0^1\beta^N(x,y)dy$ is equivalent to that of
$\int_0^1\beta^N(y,x)dy$. Clearly \eqref{conv-beta} implies that \eqref{eqn-C-beta} is satisfied with $\beta^N$ replaced by $\beta$. We note that this assumption allows in particular $\beta(x,y)$ to explode on the diagonal $x=y$,  for example, $\beta(x,y) = \frac{c}{\sqrt{|x-y|}}$ for some $c>0$, meaning that infectious interactions between ``close by" individuals are much more frequent than between 
distant ones. See further discussions in Remark \ref{rem-graphon}. 
\end{remark}

We make the following assumption on the random function $\lambda$.

\begin{assumption} \label{AS-lambda}
  Let $\lambda(\cdot)$ be a process having the same law of $\{\lambda_j(\cdot)\}_j$. 
Assume that there exists a constant $\lambda^*$ such that $\sup_{t\in \R_+} \lambda(t) \le \lambda^*$ almost surely.
Assume that there exist an integer $\kappa$, a random sequence  $0=\zeta^0 \le \zeta^1 \le \cdots \le \zeta^\kappa <\infty$ and associated random functions $\lambda^\ell \in \bC(\RR_+;[0,\lambda^\ast])$, $1\le\ell \le \kappa$, such that 
\begin{align} \label{eqn-lambda-assump}
\lambda(t) = \sum_{\ell=1}^\kappa \lambda^\ell(t) \bone_{[\zeta^{\ell-1},\zeta^\ell)}(t).
\end{align}
We write $F_\ell$ for the c.d.f. of $\zeta^\ell$, $\ell =1, \dots,\kappa$. 
In addition, 
we assume that there exists a deterministic nondecreasing function $\varphi \in \bC(\RR_+;\RR_+)$ with $\varphi(0)=0$ such that $|\lambda^\ell(t) - \lambda^\ell(s)| \le \varphi(t-s)$ almost surely for all $t,s \ge 0$ and for all $\ell\ge 1$. 

\end{assumption}

We next state the main result of the paper. In several formulas, whenever $F^c(\mfa)$ appears in the denominator, 
 it is only possibly equal to zero if the corresponding numerator is also equal to zero, and in that case, we use the convention that $\frac{0}{0} =0$.

\begin{theorem} \label{thm-FLLN} 
Under Assumptions \ref{AS-LLN-1}, \ref{AS-LLN-2} and \ref{AS-lambda},  
\begin{align} \label{eqn-LLN-conv}
&\|\bar{\mathfrak{F}}^N(t,\cdot) - \bar{\mathfrak{F}}(t,\cdot)\|_{1} \to 0, \quad \| \bar{S}^N(t,\cdot) -\bar{S}(t,\cdot)\|_{1} \to 0, \quad \quad \| \bar{R}^N(t,\cdot) -\bar{R}(t,\cdot)\|_{1} \to 0,  \non \\
&  \|  \bar{\sI}^N(t,\mfa, \cdot) -  \bar{\sI}(t,\mfa, \cdot) \|_{1} \to 0
\end{align}
 in probability as $N\to \infty$, locally uniformly in $t$ and $\mfa$, 
where the limits are given by the unique solution to the following set of integral equations. The limit  $(\bar{S}(t,x),\bar{\mathfrak{F}}(t,x))$ is a unique solution to the system of integral equations:  for $t\ge 0$ and $x\in [0,1]$, 
\begin{align}
\bar{S}(t,x) &= \bar{S}(0,x) - \int_0^t \bar\Upsilon(s,x) ds\,, \label{eqn-barS-tx} \\
\bar{\mathfrak{F}}(t,x) &=\int_0^\infty \frac{\bar{\lambda}(\mfa+t)}{F^c(\mfa)}\bar{\sI}(0,d \mfa, x) + \int_0^t  \bar{\lambda}(t-s) \bar\Upsilon(s,x) ds \,, \label{eqn-barmfF-tx}
\end{align}
where
\begin{align} \label{eqn-barUpsilon-tx}
\bar\Upsilon(t,x) =\frac{ \bar{S}(t,x)}{\bar{B}(x)} \int_0^1 \beta(x,x') \bar{\mathfrak{F}}(t,x')dx' =  \bar{\sI}_\mfa(t,0, x) \,,
\end{align}
with $\bar\sI_\mfa(t,\mfa,x)$ being the derivative of $\bar{\sI}(t,\mfa, x)$ with respect to the infection age $\mfa$.
Given $\bar{S}(t,x)$ and $\bar{\mathfrak{F}}(t,x)$, the limits $\bar{\sI}(t,\mfa, x) $ and $\bar{R}(t,x) $ are given by
\begin{align}
\bar{\sI}(t,\mfa, x) &=\int_0^{(\mfa-t)^+} \frac{F^c(\mfa'+t)}{F^c(\mfa')} \bar{\sI}(0,d \mfa', x)  + \int_{(t-\mfa)^+}^t F^c(t-s) \bar\Upsilon(s,x)  ds \,, \label{eqn-barsI-tax}\\
\bar{R}(t,x) &=  \bar{R}(0,x) + \int_0^{\infty}\Big(1- \frac{F^c(\mfa'+t)}{F^c(\mfa')} \Big) \bar{\sI}(0,d \mfa', x)  + \int_0^t F(t-s) \bar\Upsilon(s,x)  ds \,. \label{eqn-barR-tx}
\end{align}
In addition, 
\begin{align*}
 \| \bar{I}^N(t,\cdot) -\bar{I}(t,\cdot)\|_{1} \to 0
\end{align*}
locally uniformly in $t$ in probability as $N\to \infty$,  
where 
\begin{align} \label{eqn-barI-tx}
\bar{I}(t,x) =\int_0^{\infty} \frac{F^c(\mfa'+t)}{F^c(\mfa')} \bar{\sI}(0,d \mfa', x)  + \int_{0}^t F^c(t-s) \bar\Upsilon(s,x)  ds\,. 
\end{align}
For each $x$, the limits $\bar{S}(t,x)$, $\bar{\mathfrak{F}}(t,x)$,  $\bar{\sI}(t,\mfa, x)$, $\bar{I}(t,x)$ and $\bar{R}(t,x) $ are continuous in $t$ and $\mfa$. 


\end{theorem}

\begin{remark} \label{rem-graphon}
Our model can be regarded in some sense as non-Markovian epidemics dynamics on graphon. In particular, the 
function $\beta(x,x')$ can be regarded as the graphon kernel function,  representing the inhomogeneity in the connectivity.  However, the kernel function is often assumed to take values in $[0,1]$ and to be symmetric in the graphon literature. In our model, $\beta(x,x')$ does not necessarily take values in $[0,1]$ although it can be rescaled to $[0,1]$ in case it is bounded, and the function $\beta(x,x')$ may not be necessarily symmetric.  In the prelimit (the $N^{\rm th}$ system), the locations $\{\mathtt{I}^N_k\}_k$ can be regarded as a discretization of the unit interval $[0,1]$ and the infection rate functions between different locations $\beta^N_{k,k'}$ in \eqref{eqn-betaN-xx'} can then be regarded as the corresponding discretization of the function $\beta(x,x')$. We refer the readers to  \cite{keliger2022local} and \cite{petit2021random} for Markov dynamics on graphon  and the corresponding ODE approximations with a spatial parameter (no spatial partial derivative). See also Remark \ref{rem-SIS-PDE} for further discussions on how our PDE model relates to the ODE limit with a spatial parameter  for the Markovian SIS model on graphon in \cite{keliger2022local}. 
\end{remark}

\begin{remark} \label{rem-SIS}
One could adapt the methods we use for a spatial SIS model, in which infected individuals becomes susceptible right after recovery. 
For the spatial SIS model,  we have the identity
 $B^N_k=S^N_k(t)+ I^N_k(t)$ and 
 $\sum_{k=1}^{K^N} B^N_k =\sum_{k=1}^{K^N} (S^N_k(t)+ I^N_k(t))=N$ for each $t\ge0$.
  In the limit,  $\bar{B}(x) = \bar{S}(t,x) + \bar{I}(t,x)$ for each $t\ge 0, x\in [0,1]$ and
$\int_0^1 \bar{B}(x) dx = \int_0^1 (\bar{S}(t,x) + \bar{I}(t,x)) dx=1$ for each $t\ge 0$. 
We use two processes $\bar{\mathfrak{F}}^N(t,x)$ and $\bar{\sI}^N(t,\mfa, x) $ to describe the epidemic dynamics, and can show that $\|\bar{\mathfrak{F}}^N(t,\cdot)-\bar{\mathfrak{F}}(t,\cdot)\|_1 \to 0$ and $ \|\bar{\sI}^N(t,\mfa, \cdot)- \bar{\sI}(t,\mfa, \cdot)\|_1 \to 0$ in probability locally uniformly in $t$ and $\mfa$ as $N\to\infty$, where 
\begin{align}\label{eqn-bar-mfF-SIS}
\bar{\mathfrak{F}}(t,x) &=\int_0^\infty \bar{\lambda}(\mfa+t) \bar{\sT}(0,d \mfa, x) + \int_0^t  \bar{\lambda}(t-s)  \frac{ \bar{S}(s,x) }{\bar{B}(x)}\int_0^1 \beta(x,x') \bar{\mathfrak{F}}(s,x')dx'  ds\,,
\end{align}
and
\begin{align} \label{eqn-bar-sI-SIS}
\bar{\sI}(t,\mfa, x) &=\int_0^{(\mfa-t)^+} \frac{F^c(\mfa'+t)}{F^c(\mfa')} \bar{\sI}(0,d \mfa', x)  + \int_{(t-\mfa)^+}^t F^c(t-s) \frac{\bar{S}(s,x)}{\bar{B}(x)} \int_0^1 \beta(x,x') \bar{\mathfrak{F}}(s,x')dx'   ds \,, 
\end{align} 
with $ \bar{S}(t,x)$ satisfying 
\begin{equation}\label{eqn-bar-SI-sum}
\int_0^1 (\bar{S}(t,x) + \bar{\sI}(t,\infty, x) ) dx=1\,.
\end{equation}
Using $\bar{I}(t,x)=\bar{\sI}(t,\infty, x)$, 
we can write the last  equation as $\int_0^1 (\bar{S}(t,x) + \bar{I}(t,x)) dx=1$, and
 the limit $\bar{I}(t,x)$ is given by
\begin{align*}
\bar{I}(t,x) =\int_0^{\infty} \frac{F^c(\mfa'+t)}{F^c(\mfa')} \bar{\sI}(0,d \mfa', x)  + \int_{0}^t F^c(t-s) \frac{\bar{S}(s,x)}{\bar{B}(x)} \int_0^1 \beta(x,x') \bar{\mathfrak{F}}(s,x')dx'   ds\,. 
\end{align*}
\end{remark}

\bigskip
\section{PDE Models} \label{sec-PDE}

In this section we derive the PDE models associated with the limits from the FLLN. 
For each $t$, the limits $\bar{S}(t,x), \bar{\mathfrak{F}}(t,x), \bar{I}(t,x), \bar{R}(t,x)$ can be regarded as the densities of the quantities, susceptibles, aggregate infectivity, infected and recovered,  distributed over the location $x \in [0,1]$, and for each $t$ and $\mfa$, the function $\bar\sI(t,\mfa,x)$ can be also regarded as the density of the proportion of infected individuals at time $t$ with infection age less than or equal to $\mfa$, over the location $x\in [0,1]$. In addition, for each fixed $t$ and $x$,  $\bar\sI(t,\mfa,x)$ is increasing in $\mfa$, and can be regarded as a ``distribution" over the infection ages. If  $\bar\sI(t,\mfa,x)$ is absolutely continuous in $\mfa$, we let $\bar\mfi(t,\mfa,x) = \bar\sI_\mfa(t,\mfa,x)$ be the density function of $\bar{\sI}(t,\mfa, x)$ with respect to the infection age $\mfa$. 

In the following we will consider the dynamics of $\bar{S}(t,x), \bar{\mathfrak{F}}(t,x), \bar{I}(t,x), \bar{R}(t,x), \bar{\sI}(t,\mfa, x)$ in $t$ and $\mfa$, as a PDE model. 
Since there is no movement of individuals between locations, no derivative with respect to $x$ will  appear. However, the interaction among individuals in different locations will be captured in these dynamics, in particular, in the expression of $\bar\Upsilon(t,x)$ in \eqref{eqn-barUpsilon-tx}.

In this section, we restrict ourselves to $F$ being a continuous distribution.


\begin{prop} \label{prop-PDE-g}
Suppose that for each $x$, $\bar\sI(0,\mfa,x)$ is absolutely continuous with respect to $\mfa$ with density $\bar\mfi(0,\mfa,x) =\bar\sI_\mfa(0,\mfa,x)$. 
Then for $t,\mfa>0$ and $x\in [0,1]$, the function $\bar\sI(t,\mfa,x)$ is absolutely continuous in $t$ and $\mfa$, and its density $\bar{\mfi}(t,\mfa,x) =\bar\sI_\mfa(t,\mfa,x) $ with respect to $\mfa$ satisfies 
\begin{equation} \label{eqn-mfi-PDE-g}
\frac{\partial \bar{\mfi}(t,\mfa,x)}{\partial t} + \frac{\partial \bar{\mfi}(t,\mfa,x)}{\partial \mfa} = -\bar{\mfi}(t,\mfa,x) \, \frac{ F(d\mfa)}{F^c(\mfa)}  \,,
\end{equation}
$(t,\mfa,x)$ in $(0,\infty)^2\times [0,1]$, with the initial condition  $\bar\mfi(0,\mfa,x)=\bar\sI_\mfa(0,\mfa,x)$ for $(\mfa,x)\in (0,\infty) \times [0,1]$, and the boundary condition 
\begin{align} \label{eqn-mfi-PDE-BC1-g}
\bar\mfi(t,0,x) = \frac{\bar{S}(t,x)}{\bar{B}(x)} \int_0^1 \beta(x,x') \Bigg(\int_0^{\infty} \frac{\bar\lambda(\mfa')}{F^c(\mfa')} \, \bar\mfi(t,\mfa',x') d \mfa' \Bigg) d x' \,. 
\end{align}

The function $\bar{S}(t,x)$ satisfies
\begin{equation}\label{eqn-barS-PDE}
 \frac{\partial \bar{S}(t,x)}{ \partial t} = -\bar{\mfi}(t,0,x)\,, 
\end{equation}
with $\bar{S}(0,x)$ satisfying \eqref{eqn-barB-identity}. 

Moreover, the PDE \eqref{eqn-mfi-PDE-g}-\eqref{eqn-mfi-PDE-BC1-g}  has a unique  non-negative solution which is given as follows: 
for $\mfa\ge t$ and $x \in [0,1]$, 
\begin{equation} \label{eqn-bar-mfi-s1-g}
\bar\mfi(t,\mfa,x) = \frac{F^c(\mfa)}{F^c(\mfa-t)} \, \bar\mfi(0, \mfa-t, x), 
\end{equation}
and for $t>\mfa$ and $x\in [0,1]$,
\begin{equation}\label{eqn-bar-mfi-s2-g}
\bar\mfi(t,\mfa,x) = F^c(\mfa) \, \bar\mfi(t-\mfa,0, x), 
\end{equation}
and the boundary function is the unique  non-negative solution to the integral equation
\begin{align}\label{eqn-mfi-PDE-BC2-g}
\bar{\mfi}(t,0,x) & = (\bar{B}(x))^{-1} \Big(  \bar{S}(0,x) - \int_0^t \bar{\mfi}(s,0,x) ds \Big)  \non \\
& \quad \times \int_0^1 \beta(x,x') 
\left( \int_0^\infty \bar{\lambda}(\mfa+t)\, \frac{\bar{\mfi}(0,\mfa, x')}{F^c(\mfa)}d \mfa + \int_0^t  \bar{\lambda}(t-s) \, \bar{\mfi}(s,0,x') ds \right)  dx' \,. 
\end{align}

\end{prop}

Given  the PDE solution $\bar{\mfi}(t,\mfa,x)$ and $\bar\Upsilon(t,x)=\bar\mfi(t,0,x)$, the functions $\bar{I}(t,x)$ and $\bar{R}(t,x)$ are given by 
\begin{align*}
\bar{I}(t,x) &=\int_0^{\infty} \frac{F^c(\mfa'+t)}{F^c(\mfa')} \, \bar{\mfi}(0, \mfa', x)d \mfa'  + \int_{0}^t F^c(t-s) \, \bar\mfi(s,0,x) ds\,,\\
\bar{R}(t,x) &=  \bar{R}(0,x) + \int_0^{\infty}\Big(1- \frac{F^c(\mfa'+t)}{F^c(\mfa')} \Big) \, \bar{\mfi}(0, \mfa', x)  d \mfa' + \int_0^t F(t-s)\, \bar\mfi(s,0,x)  ds \,.
\end{align*}
Also, by definition,
\[
\bar{I}(t,x) = \bar{\sI}(t,\infty,x) = \int_0^\infty \bar\mfi(t,\mfa,x)d\mfa. 
\]

\begin{proof}
Using the expression of $ \bar\Upsilon(s,x)= \bar{\sI}_\mfa(s,0, x)$ in \eqref{eqn-barUpsilon-tx} and with $F^c$, we can equivalently rewrite \eqref{eqn-barsI-tax} as
\begin{align}\label{eqn-barsI-tax-g} 
\bar{\sI}(t,\mfa, x) &=\int_0^{(\mfa-t)^+} \frac{F^c(\mfa'+t)}{F^c(\mfa')} \bar{\sI}_\mfa(0, \mfa', x) d \mfa' + \int_{(t-\mfa)^+}^t F^c(t-s)  \bar{\sI}_\mfa(s,0, x)   ds \,. 
\end{align}

%

Exploiting the fact that $\frac{\partial}{\partial t}+\frac{\partial}{\partial \mfa}$ of a function of $t-\mfa$ vanishes,  we deduce from \eqref{eqn-barsI-tax-g} that 
\begin{align*}
\bar{\sI}_t(t,\mfa, x) +\bar{\sI}_\mfa(t,\mfa, x)  &=  - \int_0^{(\mfa-t)^+} \frac{1}{F^c(\mfa')} \bar{\sI}_\mfa(0,\mfa', x)F( t+d\mfa')    \\
& \quad + \bar{\sI}_\mfa(t,0, x) - \int_{(t-\mfa)^+}^t   \bar{\sI}_\mfa(s,0, x)  F(t-ds) \\
&=  - \int_t^{\mfa\vee t} \frac{1}{F^c(\mfa'-t)} \bar{\sI}_\mfa(0,\mfa'-t, x)F(d\mfa')    \\
& \quad + \bar{\sI}_\mfa(t,0, x) - \int_{0}^{\mfa\wedge t}   \bar{\sI}_\mfa(t-s,0, x)  F(ds)\,. 
\end{align*}
We then take derivative with respect to $\mfa$ on both sides of this equation (denoting $\bar{\sI}_{t,\mfa}(t,\mfa, x)$  and $\bar{\sI}_{\mfa,\mfa}(t,\mfa, x)$ as the derivatives of $\bar{\sI}_t(t,\mfa, x)$ and $\bar{\sI}_\mfa(t,\mfa, x) $ with respect to $\mfa$) and obtain the following: 
\begin{align*}
\bar{\sI}_{t,\mfa}(t,\mfa, x) +\bar{\sI}_{\mfa,\mfa}(t,\mfa, x)  &= -  {\bf1}_{\mfa \ge t}  \frac{F(d\mfa)}{F^c(\mfa-t)} \bar{\sI}_\mfa(0,\mfa-t, x)   - {\bf1}_{t>\mfa} F(d\mfa)  \bar{\sI}_\mfa(t-\mfa,0, x) \,. 
\end{align*}
Rewriting $\frac{\partial \bar{\mfi}(t,\mfa,x)}{\partial t}  =\bar{\sI}_{\mfa,t}(t,\mfa, x) = \bar{\sI}_{t,\mfa}(t,\mfa, x)$ and $\frac{\partial \bar{\mfi}(t,\mfa,x)}{\partial \mfa} = \bar{\sI}_{\mfa,\mfa}(t,\mfa, x)$, we obtain the PDE: 
\begin{align} \label{eqn-mfi-PDE-1-g}
\frac{\partial \bar{\mfi}(t,\mfa,x)}{\partial t} +\frac{\partial \bar{\mfi}(t,\mfa,x)}{\partial \mfa} 
 &= -  {\bf1}_{\mfa \ge t}  \frac{F(d\mfa)}{F^c(\mfa-t)} \, \bar{\mfi}(0,\mfa-t, x)   - {\bf1}_{t>\mfa} F(d\mfa)  \,\bar{\mfi}(t-\mfa,0, x) \,.
\end{align}
 In order to see that the right hand side coincides with that in \eqref{eqn-mfi-PDE-g}, we
first establish \eqref{eqn-bar-mfi-s1-g} and \eqref{eqn-bar-mfi-s2-g}. 
For $\mfa \ge t$, $0 \le s \le t$ and $x\in [0,1]$, 
\begin{align*}
\frac{d \bar{\mfi}(s,\mfa-t+s,x)}{d s}  = - \frac{F(\mfa-t+ds)}{F^c(\mfa-t)}\, \bar\mfi(0,\mfa-t,x)\,,
\end{align*}
and for $t>\mfa$, $0 \le s \le \mfa$ and $x \in [0,1]$,
\begin{align*}
\frac{d \bar{\mfi}(t-\mfa+s,s,x)}{d s}  = -F(ds) \,  \bar\mfi(t-\mfa,0,x)\,. 
\end{align*}
From these, by integration and simple calculations, we obtain 
 \eqref{eqn-bar-mfi-s1-g} and \eqref{eqn-bar-mfi-s2-g}. Now \eqref{eqn-mfi-PDE-g}  follows from \eqref{eqn-mfi-PDE-1-g}, \eqref{eqn-bar-mfi-s1-g}  and \eqref{eqn-bar-mfi-s2-g}. 

Then using \eqref{eqn-bar-mfi-s1-g} and \eqref{eqn-bar-mfi-s2-g}, by \eqref{eqn-barmfF-tx} and the second equality in \eqref{eqn-barUpsilon-tx}, we obtain
\begin{align} \label{eqn-mfF-tx-i}
\bar{\mathfrak{F}}(t,x) &=\int_0^\infty  \bar{\lambda}(\mfa+t)  \, \frac{ \bar{\mfi}(0,\mfa, x)}{F^c(a)}d \mfa + \int_0^t  \bar{\lambda}(t-s) \, \bar{\mfi}(s,0,x) ds\,.  
\end{align} 
The expression for the boundary condition in \eqref{eqn-mfi-PDE-BC2-g} then follows directly from 
 \eqref{eqn-barUpsilon-tx} using this expression of $\bar{\mathfrak{F}}(t,x)$. Again, using \eqref{eqn-bar-mfi-s1-g} and \eqref{eqn-bar-mfi-s2-g}, we see that the boundary condition  \eqref{eqn-mfi-PDE-BC2-g} is equivalent to  \eqref{eqn-mfi-PDE-BC1-g}. 
 
We now sketch the proof of existence and uniqueness of  a non-negative solution to \eqref{eqn-mfi-PDE-BC2-g}. Note that, thanks to 
\eqref{eqn-bar-mfi-s1-g}  and \eqref{eqn-bar-mfi-s2-g}, existence and uniqueness of a non-negative solution to the PDE \eqref{eqn-mfi-PDE-g}-\eqref{eqn-mfi-PDE-BC1-g} will follow from that result. 
First of all, let us rewrite that equation as
\[ u(t,x)=(\bar{B}(x))^{-1}\left(f(x)-\int_0^t u(s,x)ds\right)\times\int_0^1\beta(x,x')\left(g(t,x')+\int_0^t\bar{\lambda}(t-s)u(s,x')ds\right)dx',\]
where $0\le f(x)\le\bar{B}(x)$ and $0\le g(t,x)\le \lambda^\ast \bar{B}(x)$ are given from the initial conditions. Any nonnegative solution satisfies
\begin{align*}
u(t,x)&\le\int_0^1\beta(x,x')\left(g(t,x')+\int_0^t\bar{\lambda}(t-s)u(s,x')ds\right)dx',\ \text{hence}\\
\|u(t,\cdot)\|_\infty&\le C_\beta\lambda^\ast\left(C_B+\int_0^t \|u(s,\cdot)\|_\infty ds\right)\\
&\le C_\beta\lambda^\ast C_Be^{C_\beta\lambda^\ast t}\,.
\end{align*}
Here $\|u(t,\cdot)\|_\infty = \sup_{x\in [0,1]} |u(t,x)|$. 
Let now $u$ and $v$ be two non negative solutions. Then, 
\begin{align*}
|u(t,x)-v(t,x)|&\le  (\bar{B}(x))^{-1}\left(\int_0^1\beta(x,x')\left[g(t,x')+\int_0^t\bar{\lambda}(t-s) u(s,x')ds\right]dx' \right)\int_0^t|u(s,x)-v(s,x)|ds\\&\quad
+ (\bar{B}(x))^{-1} \left(f(x)+\int_0^tv(s,x)ds\right)\int_0^1\beta(x,x')\int_0^t\bar{\lambda}(t-s) |u(s,x')-v(s,x')|dsdx'\,.
\end{align*}
Integrating over $dx$, exploiting the previous a priori estimate and \eqref{eqn-C-beta}, we deduce the uniqueness from Gronwall's Lemma. Finally, the existence of a nonnegative $L^1([0,1])$-valued solution  can be established using a Picard iteration argument. Note that in the previous lines we have used the two distinct inequalities contained in \eqref{eqn-C-beta}.
\end{proof}

If $F$ is absolutely continuous, with density $f$, we denote by $h(\mfa)$ the hazard rate function, i.e., $h(\mfa) = \frac{f(\mfa)}{F^c(\mfa)}$ for all $\mfa\ge 0$.  We obtain the following corollary in this case.

\begin{coro} \label{coro-PDE-ac}
Under the assumptions of Proposition \ref{prop-PDE-g}, if $F$ is absolutely continuous with density $f$, then the PDE in \eqref{eqn-mfi-PDE-g} becomes 
\begin{equation} \label{eqn-mfi-PDE}
\frac{\partial \bar{\mfi}(t,\mfa,x)}{\partial t} + \frac{\partial \bar{\mfi}(t,\mfa,x)}{\partial \mfa} = - h(\mfa) \, \bar{\mfi}(t,\mfa,x) \,,
\end{equation}
with the initial condition $\bar\mfi(0,\mfa,x)=\bar\sI_\mfa(0,\mfa,x)$ for $(\mfa,x)\in (0,\infty) \times [0,1]$
and  the boundary condition  \eqref{eqn-mfi-PDE-BC1-g}.
The function $\bar{S}(t,x)$ satisfies \eqref{eqn-barS-PDE}, and 
the PDE \eqref{eqn-mfi-PDE} has a unique solution which is given 
 by \eqref{eqn-bar-mfi-s1-g}  and \eqref{eqn-bar-mfi-s2-g},
and the boundary function is the unique solution of  \eqref{eqn-mfi-PDE-BC2-g}.
\end{coro}

When the infectious periods are deterministic, using an approximation of the Dirac measure by continuous distributions,  we obtain the following corollary.

\begin{coro} \label{coro-PDE-det}
Suppose that the infectious periods are deterministic and equal to $t_i$, that is, $F(t) ={\bf1}_{t\ge t_i}$. Then
the PDE  in \eqref{eqn-mfi-PDE-g} becomes 
\begin{equation} \label{eqn-mfi-PDE-det}
\frac{\partial \bar{\mfi}(t,\mfa,x)}{\partial t} + \frac{\partial \bar{\mfi}(t,\mfa,x)}{\partial \mfa} = - \delta_{t_i}(\mfa) \, \bar{\mfi}(t,\mfa,x) \,,
\end{equation}
with $\delta_{t_i}(d\mfa)$ being the Dirac measure at $t_i$, 
with the initial condition  $\bar\mfi(0,\mfa,x)=\bar\sI_\mfa(0,\mfa,x)$ for $(\mfa,x)\in (0, t_i) \times [0,1]$, and the boundary condition 
\begin{align} \label{eqn-mfi-PDE-BC1-det}
\bar\mfi(t,0,x) =  \frac{\bar{S}(t,x)}{\bar{B}(x)} \int_0^1 \beta(x,x') \Bigg(\int_0^{t_i} \bar\lambda(\mfa') \, \bar\mfi(t,\mfa',x') d \mfa' \Bigg) d x' \,.
\end{align}
The PDE \eqref{eqn-mfi-PDE-det} has a unique solution which is given as follows: 
for $t \le \mfa < t_i$ and $x \in [0,1]$, 
\begin{equation} \label{eqn-bar-mfi-s1-det}
\bar\mfi(t,\mfa,x) = \bar\mfi(0, \mfa-t, x), 
\end{equation}
and for $\mfa< t \wedge t_i$ and $x\in [0,1]$,
\begin{equation}\label{eqn-bar-mfi-s2-det}
\bar\mfi(t,\mfa,x) =  \bar\mfi(t-\mfa,0, x), 
\end{equation}
and for $\mfa \ge t_i$ and $t\ge 0$, $\bar\mfi(t,\mfa,x)=0$.  
The boundary function is the unique solution to the integral equation: for $0 <t<t_i$,
\begin{align}\label{eqn-mfi-PDE-BC2-det1}
\bar{\mfi}(t,0,x) & = \bar{B}(x)^{-1}\Big(  \bar{S}(0,x) - \int_0^t \bar{\mfi}(s,0,x) ds \Big)  \non \\
& \quad \times \int_0^1 \beta(x,x') 
\left( \int_t^{t_i} \bar{\lambda}(\mfa) \, \bar{\mfi}(0,\mfa-t, x')d \mfa + \int_0^t  \bar{\lambda}(t-s) \,\bar{\mfi}(s,0,x') ds \right)  dx' \,, 
\end{align}
and for $t\ge t_i$,
\begin{align}\label{eqn-mfi-PDE-BC2-det2}
\bar{\mfi}(t,0,x) & =  \bar{B}(x)^{-1} \Big(  \bar{S}(0,x) - \int_0^t \bar{\mfi}(s,0,x) ds \Big)  \times \int_0^1 \beta(x,x') 
\int_0^{t_i}  \bar{\lambda}(t-s)\,  \bar{\mfi}(s,0,x') ds dx' \,. 
\end{align}

\end{coro}

\begin{remark} \label{rem-special} 
In the special case when $\lambda_i(t) =\tilde\lambda(t){\bf1}_{t<\eta_i}$ for a deterministic function $\tilde\lambda(t)$, the boundary condition \eqref{eqn-mfi-PDE-BC1-g} \ becomes
\begin{align} \label{eqn-mfi-PDE-BC1-special}
\bar\mfi(t,0,x) = \frac{\bar{S}(t,x)}{\bar{B}(x)} \int_0^1 \beta(x,x') \Bigg(\int_0^{\infty} \tilde\lambda(\mfa') \, \bar\mfi(t,\mfa',x') d \mfa' \Bigg) d x' \,.
\end{align}
This is because 
$\bar\lambda(t) =  \tilde{\lambda}(t) F^c(t)$.
This boundary condition resembles that given in the Diekmann  PDE model \cite{diekmann1978thresholds} (without $\bar{B}(x)$ in the denominator). See further discussions in Remark \ref{rem-Diekman-Kendall}.  
We remark that the PDE model first proposed by Kermack and McKendrick in \cite{KM32} also corresponds to this special infectivity function $\lambda_i(t) =\tilde\lambda(t){\bf1}_{t<\eta_i}$; see further discussions on the PDE models with infection-age dependent infectivity in Remarks 3.3 and 3.4 of \cite{PP-2021}. 
\end{remark}

\begin{remark}
By using the solution expressions in \eqref{eqn-bar-mfi-s1-g} and \eqref{eqn-bar-mfi-s2-g} together with the second identity $\bar\Upsilon(t,x) = \bar{\sI}_\mfa(t,0, x) $  in \eqref{eqn-barUpsilon-tx}, we can rewrite $\bar{\mathfrak{F}}(t,x)$ in \eqref{eqn-barmfF-tx} as 
\begin{align} \label{eqn-barmfF-tx-2}
\bar{\mathfrak{F}}(t,x) &=\int_0^\infty \bar{\lambda}(\mfa+t)  \, \frac{ \bar{\mfi}(0, \mfa, x)}{F^c(\mfa)} d \mfa+ \int_0^t  \bar{\lambda}(t-s)  \, \bar{\mfi}(s,0, x) ds \non\\
&= \int_0^\infty   \bar{\lambda}(\mfa+t)     \frac{1}{F^c(t+\mfa)}  \, \bar{\mfi}(t, t+ \mfa, x) d \mfa+ \int_0^t  \bar{\lambda}(\mfa) \frac{1}{F^c(\mfa)}\, \bar{\mfi}(t,\mfa, x) d\mfa \non \\
&= \int_0^{\infty} \frac{1}{F^c(\mfa)} \bar\lambda(\mfa) \, \bar{\mfi}(t,\mfa, x) d\mfa\,. 
\end{align}
In the special case when $\lambda_i(t) =\tilde\lambda(t){\bf1}_{t<\eta_i}$ as described in the previous remark, 
we obtain 
\begin{align} \label{eqn-barmfF-tx-2-special}
\bar{\mathfrak{F}}(t,x) = \int_0^{\infty} \tilde\lambda(\mfa) \, \bar{\mfi}(t,\mfa, x) d\mfa\,, 
\end{align}
which further gives 
\begin{align} \label{eqn-barUpsilon-tx-2-special}
\bar\Upsilon(t,x) & = \frac{\bar{S}(t,x)}{\bar{B}(x)} \int_0^1 \beta(x,x') \int_0^{\infty} \tilde\lambda(\mfa) \, \bar{\mfi}(t,\mfa, x') d\mfa dx' \non \\
&=\frac{\bar{S}(t,x)}{\bar{B}(x)}   \int_0^{\infty} \int_0^1 \beta(x,x') \tilde\lambda(\mfa) \, \bar{\mfi}(t,\mfa, x')  dx'd\mfa \,.
\end{align}
\end{remark} 

\begin{remark} \label{rem-Diekman-Kendall}
In Diekmann  \cite{diekmann1978thresholds}, the spatial-temporal deterministic model is specified as follows.
The function $\bar{I}(t,x)$ is written as an integral of the function $\bar{\mfi}(t, \mfa,x)$: 
\[
\bar{I}(t,x) = \int_0^\infty \bar{\mfi}(t, \mfa,x) d\mfa\,.
\]
The infectivity function   is given by 
\begin{equation}  \label{eqn-barUpsilon-Diekman}
\bar\Upsilon(t,x) = \bar{S}(t,x)\int_0^\infty \int_0^1 \bar{\mfi}(t, \mfa,x') A(\mfa, x, x') dx' d\mfa\,,
\end{equation}
where $A(\mfa, x, x')$ is the infectivity at $x$ due to the infected individual with the infection age $\mfa$ at $x'$. (Note the difference of $\bar\Upsilon(t,x)$ in \eqref{eqn-barUpsilon-Diekman} from our limit $\bar\Upsilon(t,x)$ in \eqref{eqn-barUpsilon-tx-2-special} with $\bar{B}(x)$ in the denominator, and abusing notation we use the same symbols in this remark).
Therefore, in order to match the model by Diekmann  \cite{diekmann1978thresholds}, we can take 
\begin{equation} \label{eqn-A-beta}
A(\mfa, x,x') = \beta(x,x')  \bar\lambda(\mfa). 
\end{equation}

By \eqref{eqn-barS-PDE} and \eqref{eqn-mfi-PDE-BC2-g}, we obtain
\begin{align} \label{eqn-barS-tx-D1}
 \frac{\partial \bar{S}(t,x)}{ \partial t}
 & =  - \frac{\bar{S}(t,x) }{\bar{B}(x)} \int_0^\infty  \int_0^1 \beta(x,x') 
\bar{\lambda}(\mfa+t)\, \bar{\mfi}(0,\mfa, x') dx' d \mfa   \non\\
& \quad  - \frac{\bar{S}(t,x)}{\bar{B}(x)}   \int_0^t \int_0^1 \beta(x,x') 
   \bar{\lambda}(t-s)  \, \bar{\mfi}(s,0,x')   dx'  ds  \non  \\
   &  =   \frac{\bar{S}(t,x) }{\bar{B}(x)}    \bigg( \int_0^t \int_0^1 \beta(x,x') 
   \bar{\lambda}(\mfa)  \frac{\partial \bar{S}(t-\mfa,x')}{ \partial t}   dx'  d\mfa - h(t,x)\bigg)\,,
\end{align} 
where
\[
h(t,x) =  \int_0^\infty  \int_0^1 \beta(x,x') 
\bar{\lambda}(\mfa+t)  \, \bar{\mfi}(0,\mfa, x') dx' d \mfa\,. 
\]
Then integrating \eqref{eqn-barS-tx-D1} with respect to $t$, we can calculate $u(t,x) = - \ln \frac{\bar{S}(t,x)}{\bar{S}(0,x)}$.
  If one were to assume $\bar{B}(x)$=1 in the denominator of \eqref{eqn-barS-tx-D1}, then we would obtain 
\[
u(t,x) = - \ln \frac{\bar{S}(t,x)}{\bar{S}(0,x)} = \int_0^t \int_0^1 (1- e^{-u(t-\mfa,x')}) \bar{S}(0, x')  \beta(x,x') 
   \bar{\lambda}(\mfa)   dx' d\mfa + \int_0^t h(s,x) ds,
\]
and next using \eqref{eqn-A-beta}, we would then obtain the specification of $u(t,x)$ in  \cite{diekmann1978thresholds}.


Moreover, if the infection rate is the constant $\lambda$ and the infectious periods are exponential of rate $\mu$, 
we have $\bar{\mathfrak{F}}(t,x) = \lambda \bar{I}(t,x) $, and as a result, the infectivity function of Diekmann in \eqref{eqn-barUpsilon-Diekman} becomes 
\begin{align} \label{eqn-barUpsilon-Kendall}
\bar\Upsilon(t,x) = \bar{S}(t,x)  \int_0^1 \beta(x,x') \lambda \bar{I}(t,x')dx'\,.
\end{align}
Because of the memoryless property of exponential periods, it is adequate to use the process $I(t,x)$ to describe the dynamics instead of $\sI(t,\mfa,x)$. In this case, we obtain
Kendall's spatial model \cite{kendall1957,kendall1965}, in which given the limit $\bar\Upsilon(t,x) $ in \eqref{eqn-barUpsilon-Kendall}, 
 \begin{align} \label{eqn-Kendall-ODE-model}
 \frac{\partial \bar{S}(t,x)}{\partial t} = -\bar\Upsilon(t,x), \quad     \frac{\partial \bar{I}(t,x)}{\partial t} =\bar\Upsilon(t,x) - \mu \bar{I}(t,x), \quad     \frac{\partial \bar{R}(t,x)}{\partial t}  = \mu \bar{I}(t,x)\,. 
 \end{align}

\end{remark}

\begin{remark} \label{rem-SIS-PDE}

Recall the spatial SIS model in Remark \ref{rem-SIS}. We obtain the same PDE in \eqref{eqn-mfi-PDE-g} with the boundary condition in \eqref{eqn-mfi-PDE-BC1-g}, in which $\bar{S}(t,x)$ is the solution to \eqref{eqn-barS-PDE} with $\bar{S}(0,x)$ satisfying $\int_0^1 (\bar{S}(0,x)+\bar{I}(0,x)) dx =1$ and $\bar{S}(0,x) = \bar{B}(x)-\bar{I}(0,x)$. 
The solution to the PDE is also given by \eqref{eqn-bar-mfi-s1-g}--\eqref{eqn-bar-mfi-s2-g} with the boundary condition in \eqref{eqn-mfi-PDE-BC1-g},
with $\bar{S}(0,x)$  mentioned above.
Similarly, we also obtain the expression of $\bar{\mathfrak{F}}(t,x)$ in \eqref{eqn-barmfF-tx-2}.

In the Markovian case with a constant infection rate $\lambda$ and recovery rate $\mu$, our model reduces to the following ODE with a spatial parameter:  
\begin{align} \label{eqn-SIS-Markov-ODE}
  \frac{\partial \bar{I}(t,x)}{\partial t} =\lambda \bar{S}(t,x)  \int_0^1 \beta(x,x') \bar{I}(t,x')dx' - \mu \bar{I}(t,x)
\end{align}
with $ \bar{S}(t,x)$ satisfying $\int_0^1 (\bar{S}(t,x) + \bar{I}(t,x)) dx=1$ $\bar{B}(x) = \bar{S}(t,x) + \bar{I}(t,x)$ for each $t\ge 0$. (This can be also seen from \eqref{eqn-Kendall-ODE-model} and \eqref{eqn-barUpsilon-Kendall}.)
This resembles  the ODE limit with a spatial parameter as established by Keliger et al. \cite{keliger2022local} for the finite-state Markov SIS model on a sampled graph from graphon (since there is only one individual at each node of the graph, $\bar{S}(t,x)$ in \eqref{eqn-SIS-Markov-ODE} is replaced by $1-\bar{I}(t,x)$, see equation (10) in that paper).

Returning to the integral limit for the spatial SIS model in Remark \ref{rem-SIS}, we assume that $\lim_{t\to\infty}\bar{\sI}(t,\mfa, x)$ exists and the limit is denoted as $\bar{\sI}^*(\mfa, x)$, and let $\bar{I}^*(x) =\lim_{t\to\infty} \bar{I}(t,x) =\bar{\sI}^*(\infty, x)$. Also let $\bar{S}^*(x) =\lim_{t\to\infty} \bar{S}(t,x)$. Note that 
\begin{equation} \label{eqn-SI-eqlm}
\int_0^1 (\bar{S}^*(x) + \bar{I}^*(x) ) dx=1.
\end{equation}
 Let $\beta^{-1}= \int_0^\infty F^c(\mfa) d \mfa$ and $F_e(\mfa)= \beta  \int_0^\mfa F^c(s) ds$. 

By \eqref{eqn-bar-sI-SIS} and \eqref{eqn-barmfF-tx-2}, we obtain 
\begin{align} \label{eqn-bar-sI-SIS-eqlm}
\bar{\sI}^*(\mfa, x) 
&= \int_0^\mfa F^c(s)  ds\,  \bar{S}^*(x)  \int_0^1 \beta(x,x') \int_0^{\infty} \frac{1}{F^c(\mfa')} \bar\lambda(\mfa') \bar{\mathfrak{I}}^*(d\mfa', x')   dx'   \non\\
&= \beta^{-1} F_e(\mfa) \bar{S}^*(x)  \int_0^1 \beta(x,x') \int_0^{\infty} \frac{1}{F^c(\mfa')} \bar\lambda(\mfa') \bar{\mathfrak{I}}^*(d\mfa', x')   dx'\,. 
\end{align} 
By letting $\mfa\to\infty$ on the both sides, we obtain
\begin{align} 
\bar{I}^*(x) &= \beta^{-1} \bar{S}^*(x)  \int_0^1 \beta(x,x') \int_0^{\infty} \frac{1}{F^c(\mfa')} \bar\lambda(\mfa') \bar{\mathfrak{I}}^*(d\mfa', x')   dx' \,. 
\end{align} 
This implies 
\[
\bar{\sI}^*(\mfa, x) 
= F_e(\mfa) \bar{I}^*(x)\,,
\]
which then gives 
\[
\frac{\partial}{\partial \mfa}\bar{\sI}^*(\mfa, x) 
= \beta F^c(\mfa) \bar{I}^*(x) \,. 
\]
Thus,
\begin{align*}
\bar{\sI}^*(\mfa, x) 
&= \beta^{-1} F_e(\mfa) \bar{S}^*(x)  \int_0^1 \beta(x,x') \int_0^{\infty} \frac{1}{F^c(\mfa')} \bar\lambda(\mfa') \beta F^c(\mfa') \bar{I}^*(x') d \mfa'   dx'   \\
&=  F_e(\mfa)\Big( \int_0^{\infty}  \bar\lambda(\mfa') d \mfa'   \Big) \bar{S}^*(x)  \int_0^1 \beta(x,x')   \bar{I}^*(x')   dx'  \,. 
\end{align*} 
Define $\bar\Lambda= \int_0^{\infty}  \bar\lambda(t) d t$.  
By letting $\mfa\to\infty$ again on both sides, we obtain that the equilibrium $\bar{I}^*(x)$ must satisfy
\begin{align*}
\bar{I}^*(x) 
&=  \bar\Lambda (\bar{B}(x) -  \bar{I}^*(x)) \int_0^1 \beta(x,x')   \bar{I}^*(x')   dx' \,, \quad x \in [0,1]. 
\end{align*} 
This equation has a solution $\bar{I}^*(x) \equiv 0$, which is the disease-free equilibrium. Let us now discuss the existence of an endemic equilibrium. 
First, observe that if $\bar{B}(x)\equiv \bar{B}$ and $\beta(x,y) \equiv \beta$ are constant, then the equation reduces to $\bar{I}^* = \bar\Lambda (\bar{B} - \bar{I}^*) \beta \bar{I}^*$, which has a positive solution if and only if $\bar\Lambda\bar{B} \beta>1$.  
Next, we give a sufficient condition under which the integral equation (below $k(x,y)=\bar\Lambda \beta(x,y)$)
\[ u(x)=(\bar{B}(x)-u(x))\int_0^1k(x,y)u(y)dy\]
has a non zero solution. The condition reads:
\begin{equation}\label{conditendemic}
\inf_{x\in[0,1]}\bar{B}(x)\int_0^1k(x,y)dy>1\,.
\end{equation}

We first note that condition \eqref{conditendemic}  implies that there exists $\delta>0$ such that 
\[ \bar{B}(x)\int_0^1 k(x,y)dy\ge 1+\delta,\ \forall x\in[0,1]\,.\]
From this and the assumption that $\sup_x\int_0^1 k(x,y)dy<\infty$, we deduce that there exists $\theta>0$ such that
\[ ( \bar{B}(x)-\theta)\int_0^1 k(x,y) dy\ge1,\ \forall x\in [0,1]\,.\]

Let us now remark that $u:[0,1]\mapsto\R_+$ is a solution of our integral equation iff
\[ u=\Phi(u),\
\text{where }\Phi(u)(x)= \bar{B}(x)\frac{\int_0^1k(x,y) u(y)dy}{1+\int_0^1k(x,y) u(y)dy}\,.\]
The mapping $\Phi$ has the three following properties (below $\theta(x):\equiv\theta$):
\[ \Phi(\theta)\ge\theta,\quad \Phi( \bar{B})\le  \bar{B}\quad \text{ and }u\le v\Rightarrow \Phi(u)\le \Phi(v)\,.\]
We deduce readily from those properties that
\[ \theta\le \Phi(\theta)\le \Phi( \bar{B})\le  \bar{B}\,.\]
Exploiting the third property of $\Phi$, we can iterate our argument, and deduce that
\[ \theta\le \Phi(\theta)\le\cdots\le\Phi^{\circ(n-1)}(\theta)\le\Phi^{\circ n}(\theta)\le\Phi^{\circ n}( \bar{B})\le \Phi^{\circ(n-1)}( \bar{B})\le\cdots\le \Phi( \bar{B})\le  \bar{B}\,.\]
Hence both sequences $\Phi^{\circ n}(\theta)$ and $\Phi^{\circ n}( \bar{B})$ have a limit, which are positive solutions of our integral equations.
We conjecture that one should be able to establish uniqueness of a non zero solution, possible under slightly different conditions.


%

\end{remark}

\section{Some technical preliminaries} \label{sec-technical}

We will use the following convergence criteria for the processes: a) $X^N(t,x)$ in $\bD(\R_+, L^1)$  and b) $X^N(t,s,x)$ in $\bD(\R_+,\bD(\R_+, L^1))$. 
 They extend the convergence criterion for the processes in $\bD$ (the Corollary on page 83 of \cite{billingsley1999convergence})  and in $\bD_\bD$ (\cite[Theorem 4.1]{PP-2021}). 
 The proof is a straightforward extension of those results (in \cite{billingsley1999convergence} it is noted that with very little change, the theory can be extended to functions taking values in metric spaces that are separable and complete).  
We remark that one may also replace the $L_1$ norm $\|\cdot \|_1$ by the $L_2$ norm in the following results.

\begin{theorem} \label{thm-D-conv-x}
Let $\{X^N(t,x): N \ge 1\}$ be a sequence of random elements such that $X^N$ is in $\bD(\R_+, L^1)$. 
If the following two conditions are satisfied: for any $T>0$, 
\begin{itemize}
\item[(i)] for any $\ep>0$, $  \sup_{t \in [0,T]} \P \big(  \|X^N(t, \cdot)\|_{1}> \ep \big) \to 0$ as $N\to\infty$, and
\item[(ii)] for any $\ep>0$, as $\delta\to0$,
\begin{align*} 
& \limsup_N  \sup_{t\in [0,T]} \frac{1}{\delta}  \P \bigg(  \sup_{u \in [0,\delta]}\|X^N(t+u,\cdot) - X^N(t,\cdot)\|_1 > \ep\bigg) \to 0, 
\end{align*}
\end{itemize}
then  $\|X^N(t,\cdot) \|_1 \to 0 $ in probability, locally uniformly in $t$, as $N\to \infty$. 
\end{theorem}

\begin{theorem} \label{thm-DD-conv-x}
Let $\{X^N: N \ge 1\}$ be a sequence of random elements such that $X^N$ is in
 $\bD(\R_+,\bD(\R_+, L^1))$. 
If the following two conditions are satisfied: for any $T, S>0$, 
\begin{itemize}
\item[(i)] for any $\ep>0$, $  \sup_{t \in [0,T]}\sup_{s\in [0,S]} \P \big(  \|X^N(t, s,\cdot)\|_1> \ep \big) \to 0$ as $N\to\infty$, and
\item[(ii)] for any $\ep>0$, as $\delta\to0$,
\begin{align*} 
& \limsup_N  \sup_{t\in [0,T]} \frac{1}{\delta}  \P \bigg(  \sup_{u \in [0,\delta]}\sup_{s \in [0,S]} \|X^N(t+u,s,\cdot) - X^N(t,s,\cdot)\|_1 > \ep\bigg) \to 0, \\
& \limsup_N  \sup_{s\in [0,S]} \frac{1}{\delta}  \P \bigg(  \sup_{v \in [0,\delta]}\sup_{t \in [0,T]} \|X^N(t,s+v,\cdot) - X^N(t,s,\cdot)\|_1 > \ep\bigg) \to 0, 
\end{align*}
\end{itemize}
then  $\|X^N(t,s,\cdot)\|_1 \to 0 $ in probability, locally uniformly in $t$ and $s$, as $N\to\infty$. 
\end{theorem}

\medskip

We shall also need the following Lemma.
\begin{lemma}\label{convPort}
For each $N\ge1$, let $f_N:\R_+\times[0,1]\mapsto\R_+$ be measurable and such that $t\mapsto f_N(t,x)$ is non--decreasing for each $x\in[0,1]$. 
Assume that there exists $f:\R_+\times[0,1]\mapsto\R_+$ such that 
 $t\mapsto f(t,x)$ is continuous for each $x\in[0,1]$, and for all $t\ge0$,
 as $N\to\infty$,
\begin{equation}\label{conv}
\|f_N(t,\cdot)-f(t,\cdot)\|_{1}\to0\, .
\end{equation}
Let $g\in \bD(\R;\R_+)$ be such that there exists $C>0$ with $g(t)\le C$ for all $t\ge0$. Define
\[ h_N(t,x)=\int_0^tg(s)f_N(ds,x),\quad h(t,x)=\int_0^tg(s)f(ds,x)\,.\]
Then for any $t>0$, $\|h_N(t,\cdot)- h(t,\cdot)\|_{1}\to 0$ as $N\to\infty$.
In addition, $\int_0^1h_N(t,x)dx\to\int_0^1h(t,x)dx$ locally uniformly in $t$, as $N\to\infty$.

Moreover, if for each $N\ge1$, $f_N$ is random and the convergence \eqref{conv} holds in probability, then 
the conclusion holds in probability as well.
\end{lemma}
\begin{proof}
Fix $T>0$.
Let $\{s_n,\ n\ge1\}$ be a countable dense subset of $[0,T]$. By successive extraction of subsequences we can extract a subsequence from the original sequence $\{f_N,\ N\ge1\}$, which by an abuse of notation we still denote as the original sequence, and which is such that there exists a subset $\mathcal{N}\subset[0,1]$ with zero Lebesgue measure, such that
for all $n\ge1$ and $x\in[0,1]\backslash\mathcal{N}$, $f_N(s_n,x)\to f(s_n,x)$. Since for all $N$ and $x$,  $s\mapsto f_N(s,x)$ is nondecreasing and  $s\mapsto f(s,x)$ is continuous, we deduce that for all $s\in[0,T]$ and $x\in[0,1]\backslash\mathcal{N}$, $f_N(s,x)\to f(s,x)$. Consequently, for all $x\in[0,1]\backslash\mathcal{N}$, the sequence
of measures $f_N(ds,x)$ on $[0,T]$ converges weakly to the measure $f(ds,x)$. Since the set of points of discontinuity of $g$ on $[0,T]$ is at most countable and $s\mapsto f(s,x)$ is continuous, that set is of zero $f(ds,x)$ measure. Hence a slight extension of the Portmanteau theorem (see Theorem 1.2.1 in \cite{billingsley1999convergence}) yields that for all
$x\in[0,1]\backslash\mathcal{N}$, $h_N(t,x)\to h(t,x)$. Moreover, $0\le h_N(t,x)\le C f_N(t,x)$, and the upper bound converges in $L^1([0,1])$, hence the sequence $h_N(t,\cdot)$ is uniformly integrable and converges in 
$L^1([0,1])$ towards $h(t,x)$. Now all converging subsequences have the same limit, so the the whole sequence converges.

The ``locally uniform in $t$" convergence of the integrals follows from the second Dini theorem (see, e.g., Problem 127 on pages 81 and 270 in \cite{Polya-Szego}). Indeed the convergence
$\int_0^1h_N(t,x)dx\to\int_0^1h(t,x)dx$ for each $t$ follows from the above arguments, for each $N\ge1$, 
$t\mapsto\int_0^1h_N(t,x)dx$ is non--decreasing and the limit $t\mapsto\int_0^1h(t,x)dx$ is continuous.

The case of random $f_N$ is treated similarly. The extraction of subsequences is done in such a way that for each $n$, $f_N(s_n,x)$ converges as $N\to\infty$ on a subset of $\Omega\times[0,1]$ of full $d\,\P \otimes dx$ measure. We conclude that from any subsequence of the original sequence $\{h_N(t,\cdot),\ N\ge1\}$, we can extract a further subsequence which converges a.s. in $L^1([0,1])$, hence the convergence in probability in $L^1([0,1])$, as claimed.
\end{proof}

\bigskip 

\section{Proof of the Convergence of $\bar{S}^N(t,x)$ and $\bar{\mathfrak{F}}^N(t,x)$} \label{sec-proof-conv-S-F}

In this section we prove the convergence of $\bar{S}^N(t,x)$ and $\bar{\mathfrak{F}}^N(t,x)$ to 
 $\bar{S}(t,x)$ and $\bar{\mathfrak{F}}(t,x)$ given by the set of equations \eqref{eqn-barS-tx} and \eqref{eqn-barmfF-tx} together with \eqref{eqn-barUpsilon-tx}.  
We first write $S^N_k(t) = S^N_k(0)-A^N_k(t)$ as follows by \eqref{eqn-An-k-rep}:
\begin{align*}
S^N_k(t) = S^N_k(0) -  \int_0^t \int_0^\infty {\bf 1}_{u \le \Upsilon^N_k(s) } Q_{k}(ds, d u),
\end{align*}
and recall $\mathfrak{F}_k^N(t)$ in \eqref{eqn-mfk}. Then, we have
\begin{align} \label{eqn-barSn-tx}
\bar{S}^N(t,x) &= \bar{S}^N(0,x) - \sum_{k=1}^{K^N} \frac{K^N}{N} \int_0^t \int_0^\infty {\bf 1}_{u \le \Upsilon^N_k(s) } Q_{k}(ds, d u)\, {\bf 1}_{\mathtt{I}_k}(x) \non \\
&=  \bar{S}^N(0,x) -   \int_0^t  \bar\Upsilon^N(s,x)  ds \,  - \bar{M}^N_A(t,x)\,,
\end{align}
where $\overline{Q}_k(ds,du)= Q_k(ds,du)-dsdu$ and 
\begin{align} \label{eqn-barM-An-tx}
 \bar{M}^N_A(t,x):= \sum_{k=1}^{K^N}  \frac{K^N}{N}\int_0^t \int_0^\infty {\bf 1}_{u \le \Upsilon^N_k(s) } \overline{Q}_{k}(ds, d u)\, {\bf 1}_{\mathtt{I}_k}(x)\,. 
\end{align} 
We then write 
\begin{align}  \label{eqn-bar-mfFn-tx}
\bar{\mathfrak{F}}^N(t,x) 
&= \bar{\mathfrak{F}}^N_0(t,x) +  \int_0^t  \bar{\lambda} (t-s) \bar{\Upsilon}^N(s,x) ds    +  \Delta^{N}_{1,1}(t,x)+ \Delta^{N}_{1,2}(t,x)\,,
\end{align}
where 
 \begin{equation} \label{eqn-bar-mfFn-0}
\bar{\mathfrak{F}}^N_0(t,x) = \sum_{k=1}^{K^N} \frac{K^N}{N} \sum_{j: -j \in \cT^N_k(0)}  \lambda_{-j,k}(\tilde{\tau}^N_{-j,k}+ t){\bf 1}_{\mathtt{I}_k}(x),
\end{equation}
\begin{equation} \label{eqn-Delta-11}
 \Delta^{N}_{1,1}(t,x)= \sum_{k=1}^{K^N} \frac{K^N}{N} \sum_{j=1}^{A^N_k(t)} \Big( \lambda_{j,k} (t-\tau^N_{j,k}) - \bar\lambda(t-\tau^N_{j,k})\Big){\bf 1}_{\mathtt{I}_k}(x) \,,
\end{equation} 
and
\begin{align} \label{eqn-Delta-12}
\Delta^{N}_{1,2}(t,x) = \sum_{k=1}^{K^N} \frac{K^N}{N} \int_0^t  \bar{\lambda} (t-s) \int_0^\infty {\bf 1}_{u \le \Upsilon^N_k(s) } \overline{Q}_{k}(ds, d u)  \,{\bf 1}_{\mathtt{I}_k}(x)\,. 
\end{align}
Observe that 
\begin{align} \label{eqn-barUpsilon-n-sx}
\bar{\Upsilon}^N(s,x)  & =  \sum_{k=1}^{K^N} \frac{K^N}{N} \frac{S^N_k(s)}{B^N_k} \frac{1}{K^N} \sum_{k'=1}^{K^N}\beta^N_{k,k'} \mathfrak{F}_{k'}^N(s){\bf 1}_{\mathtt{I}_k}(x)  \non \\
& =     \sum_{k=1}^{K^N}  \frac{\bar{S}^N_k(s)}{\bar{B}^N_k}{\bf 1}_{\mathtt{I}_k}(x)  \int_0^1\sum_{k'=1}^{K^N} \beta^N_{k,k'}  \bar{\mathfrak{F}}_{k'}^N(s)  {\bf 1}_{\mathtt{I}_{k'}}(x')dx' \non \\
&= \frac{ \bar{S}^N(s,x)}{\bar{B}^N(x)}\int_0^1\beta^N(x,x')\bar{\mathfrak{F}}^N(s,x') dx'    \, ,
\end{align} 
where $\beta^N(x,x')$ is defined in \eqref{eqn-betaN-xx'}. 


Before proceeding to prove the convergence of $\bar{S}^N(t,x)$ and $\bar{\mathfrak{F}}^N(t,x)$, we describe the proof strategy as follows.
In the expressions of $\bar{S}^N(t,x)$ and $\bar{\mathfrak{F}}^N(t,x)$ in \eqref{eqn-barSn-tx} and \eqref{eqn-bar-mfFn-tx}, the stochastic terms $ \bar{M}^N_A(t,x)$,  $\Delta^{N}_{1,1}(t,x)$ and $ \Delta^{N}_{1,2}(t,x)$ will converge to zero in probability as $N\to \infty$, which are proved in Lemmas \ref{lem-barAn-tight} and \ref{lem-mfN1-conv}. The term $\bar{\mathfrak{F}}^N_0(t,\cdot)$ will converge to a limit  $\bar{\mathfrak{F}}_0(t,\cdot)$  (in the $\|\cdot\|_1$ norm in probability), which is proved in Lemma \ref{lem-mfN0-conv}. 
Thus, the proof for the convergence of $\bar{S}(t,x)$ and $\bar{\mathfrak{F}}^N(t,x)$ can be carried out by studying the set of integral equations 
 \eqref{eqn-barSn-tx} and \eqref{eqn-bar-mfFn-tx} together with the expression of $\bar\Upsilon^N(s,x)$ above, given the convergence of the terms $\bar{S}^N(0,\cdot)$, $\bar{\mathfrak{F}}^N_0(t,\cdot)$,
 $ \bar{M}^N_A(t,x)$,  $\Delta^{N}_{1,1}(t,x)$ and $ \Delta^{N}_{1,2}(t,x)$. 
 In the following we will first provide this argument in Proposition \ref{prop-conv-S-mfF} and then provide the proofs for the convergence of the required individual terms. 
 
The following Lemma follows readily from
 \eqref{eqn-N-B-2} and \eqref{eqn-mfk}, and the conditions on $\bar{B}^N(x)$ in \eqref{eqn-CB}.
\begin{lemma}\label{aprioriN}
The processes $\bar{S}^N(t,x)$ and $\bar{\mathfrak{F}}^N(t,x)$ are nonnegative and satisfy the following a priori bounds: 
\[ 
\sup_N \sup_{ t\ge0,\ x\in[0,1]}\bar{S}^N(t,x)\le C_B\qandq  \sup_N \sup_{t\ge0, \ x\in[0,1]}\bar{\mathfrak{F}}^N(t,x)\le\lambda^\ast C_B \quad \text{a.s.}  \]
\end{lemma}

Next, recall the set of the limiting equations: 
\begin{equation}\label{eq:SF}
\begin{aligned}
\bar{S}(t,x)&=\bar{S}(0,x)-\int_0^t\frac{\bar{S}(s,x)}{\bar{B}(x)}\int_0^1\beta(x,y)\bar{\mathfrak{F}}(s,y)dyds,\\
\bar{\mathfrak{F}}(t,x)&=\bar{\mathfrak{F}}_0(t,x)+\int_0^t\bar{\lambda}(t-s) \frac{\bar{S}(s,x)}{\bar{B}(x)} \int_0^1\beta(x,y)\bar{\mathfrak{F}}(s,y)dyds\,,
\end{aligned}
\end{equation}
where
\begin{equation} \label{eqn-bar-mfF-0}
 \bar{\mathfrak{F}}_0(t,x) :=
  \int_0^\infty  \frac{ \bar{\lambda}(\mfa+t)}{F^c(\mfa)}\bar{\sI}(0,d \mfa, x).
\end{equation}

We have the following lemmas on the solution properties to this set of equations, and also the existence and uniqueness of its solution. 
\begin{lemma}\label{apriori}
Under Assumptions \ref{AS-LLN-1} and \ref{AS-lambda}, any  $(L^\infty([0,1]))^2$--valued  solution 
$(\bar{S}(t,x),\bar{\mathfrak{F}}(t,x))$ of equation \eqref{eq:SF} is nonnegative, and 
satisfies $\sup_{t\ge0}\bar{S}(t,x)\le\bar{S}(0,x)\le C_B$ and for any $T>0$, there exists $C_T>0$ such that
\[ \sup_{0\le t\le T, x\in[0,1]}\bar{\mathfrak{F}}(t,x)\le C_T\,.\]
\end{lemma}
\begin{proof}
The non--negativity of $\bar{S}$ follows from that of the initial condition and the linearity of the equation.
 For the second statement, we first note that
$\bar{\sT}(0,\infty, x)\le C_B$, hence from \eqref{eqn-bar-mfF-0} and Assumption \ref{AS-lambda},
$0\le \bar{\mathfrak{F}}_0(t,x)\le\lambda^\ast C_B$. Hence from the second line of \eqref{eq:SF} and \eqref{eqn-C-beta} and from the assumption that $\bar{B}(x) \ge c_B>0$ for each $x \in [0,1]$ in \eqref{eqn-barB-condition},
 we obtain 
\[ \|\bar{\mathfrak{F}}(t,\cdot)\|_\infty\le\lambda^\ast C_B+\frac{C_\beta}{c_B}\lambda^\ast C_B\int_0^t \|\bar{\mathfrak{F}}(s,\cdot)\|_\infty ds.\]
Thus, the second statement with $C_T=\lambda^\ast C_B \exp\big(\frac{C_\beta}{c_B}\lambda^\ast C_BT \big)$ follows from Gronwall's lemma.
We next show that $\bar{\mathfrak{F}}(t,x)\ge0$. Suppose that $\bar{\mathfrak{F}}(t,x)=\bar{\mathfrak{F}}_+(t,x)-\bar{\mathfrak{F}}_-(t,x)$. Then we have
\[ \bar{\mathfrak{F}}_-(t,x)\le \int_0^t \bar{\lambda}(t-s) \frac{\bar{S}(s,x)}{\bar{B}(x)}\int_0^1\beta(x,y)\bar{\mathfrak{F}}_-(s,y)dyds,\]
and by a similar argument as above using Gronwall's Lemma, we deduce that $\|\bar{\mathfrak{F}}_-(t,\cdot)\|_\infty=0$, hence the result.
Finally it follows readily from Assumption \ref{AS-LLN-1} that $\bar{S}(0,x)\le \sup_N\bar{S}^N(0,x)\le C_B$ for all $x$. From the first line of \eqref{eq:SF}, 
since $\bar{S}$ and $\bar{\mathfrak{F}}$ are nonnegative, $\bar{S}(t,x)\le \bar{S}(0,x)$, hence the first statement.
\end{proof}

\begin{lemma}\label{existuniq}
Under Assumptions \ref{AS-LLN-1} and \ref{AS-lambda}, equation \eqref{eq:SF} has a unique   $(L^\infty([0,1]))^2$--valued  solution.
\end{lemma}
\begin{proof}
We already know that any solution is nonnegative and locally bounded. Uniqueness is then easy to deduce from the following estimate.
Consider two solutions $(\bar{S},\bar{\mathfrak{F}})$ and $(\bar{S}',\bar{\mathfrak{F}}')$, and define
$\bar{\Upsilon}(t,x)=\frac{\bar{S}(t,x)}{\bar{B}(x)}\int_0^1\beta(x,y)\bar{\mathfrak{F}}(t,y)dy$, $\bar{\Upsilon}'(t,x)$ similarly, replacing 
$(\bar{S},\bar{\mathfrak{F}})$ by $(\bar{S}',\bar{\mathfrak{F}}')$. 

Since from \eqref{eqn-barB-condition} $\bar{B}(x)\ge c_B$, and from Lemma \ref{apriori} $ \bar{S}(t,x)\le C_B$ and  for $0\le t\le T,
x\in[0,1]$, $\bar{\mathfrak{F}}(t,x)\le C_T$, we obtain
\begin{align*}
\|\bar{\Upsilon}(t,\cdot)-\bar{\Upsilon}'(t,\cdot)\|_\infty&\le\sup_{x \in [0,1]}\bigg|\frac{\bar{S}(t,x)}{\bar{B}(x)}- \frac{\bar{S}'(t,x)}{\bar{B}(x)}\bigg|\int_0^1\beta(x,y)\bar{\mathfrak{F}}(t,y)dy\\
&\quad+\sup_{x \in [0,1]}  \frac{\bar{S}'(t,x)}{\bar{B}(x)}\int_0^1\beta(x,y)|\bar{\mathfrak{F}}(t,y)-\bar{\mathfrak{F}}'(t,y)|dy\\
& \le \frac{1}{c_B} \|\bar{S}(t,\cdot)-\bar{S}'(t,\cdot)\|_\infty\sup_{x \in [0,1]} \int_0^1\beta(x,y)\bar{\mathfrak{F}}(t,y)dy\\
& \quad + \frac{C_\beta C_B}{c_B} \|\bar{\mathfrak{F}}(t,\cdot)-\bar{\mathfrak{F}}'(t,\cdot)\|_\infty \\
& \le \frac{C_\beta}{c_B}C_T \|\bar{S}(t,\cdot)-\bar{S}'(t,\cdot)\|_\infty   + \frac{C_\beta C_B}{c_B} \|\bar{\mathfrak{F}}(t,\cdot)-\bar{\mathfrak{F}}'(t,\cdot)\|_\infty.
\end{align*}

From this inequality, we see that uniqueness follows from Gronwall's Lemma. The same estimate can be used repeatedly for proving convergence in $L^\infty([0,1])$  of the Picard iteration procedure, which establishes existence.
\end{proof}

We can now prove the main result of this section. Let us first introduce a notation. We let $\mathcal{E}^N_{\mathfrak{F}}(t,x) =\Delta^{N}_{1,1}(t,x)+ \Delta^{N}_{1,2}(t,x)$ and 
\begin{align*}
\Psi^N(t):=\int_0^1|\bar{\mathfrak{F}}^N_0(t,x)-\bar{\mathfrak{F}}_0(t,x)|dx 
+ \int_0^1|\bar{M}^N_A(t,x)|dx
+\int_0^1|\mathcal{E}^N_{\mathfrak{F}}(t,x)|dx\,.
\end{align*}
\begin{prop} \label{prop-conv-S-mfF}
Let $T>0$ be arbitrary. Given that $\int_0^1|\bar{S}^N(0,x)-\bar{S}(0,x)|dx\to0$ in Assumption \ref{AS-LLN-1}, and assuming that 
$\sup_{0\le t\le T}\Psi^N(t)\to0$ in probability as $N\to\infty$, we have
\[ \sup_{0\le t\le T}\left(\|\bar{S}^N(t,\cdot)-\bar{S}(t,\cdot)\|_1+\|\bar{\mathfrak{F}}^N(t,\cdot)-\bar{\mathfrak{F}}(t,\cdot)\|_1\right)\to0\]
in probability  as $N\to\infty$. 
\end{prop} 
\begin{proof}
Referring to the notations in Lemmas \ref{aprioriN} and \ref{apriori}, let us assume that $\lambda^\ast\le C_T$. We first upper bound the following difference
\begin{align*}
&\frac{\bar{S}(t,x)}{\bar{B}(x)}\int_0^1\beta(x,y)\bar{\mathfrak{F}}(t,y)dy- \frac{\bar{S}^N(t,x)}{\bar{B}^N(x)}\int_0^1\beta^N(x,y)\bar{\mathfrak{F}}^N(t,y)dy\\
&=\bigg( \frac{\bar{S}(t,x)}{\bar{B}(x)} - \frac{\bar{S}^N(t,x)}{\bar{B}^N(x)} \bigg)\int_0^1\beta^N(x,y)\bar{\mathfrak{F}}^N(t,y)dy\\
&\quad+\frac{\bar{S}(t,x)}{\bar{B}(x)} \left(\int_0^1\beta(x,y)\bar{\mathfrak{F}}(t,y)dy-\int_0^1\beta^N(x,y)\bar{\mathfrak{F}}^N(t,y)dy\right)\\
&\le  C_\beta C_T\bigg| \frac{\bar{S}(t,x)}{\bar{B}(x)} - \frac{\bar{S}^N(t,x)}{\bar{B}^N(x)} \bigg|+\int_0^1\beta^N(x,y)(\bar{\mathfrak{F}}(t,y)-\bar{\mathfrak{F}}^N(t,y))dy \\
&\qquad +\int_0^1(\beta(x,y)-\beta^N(x,y))\bar{\mathfrak{F}}(t,y)dy
\, .
\end{align*}
Note that by \eqref{eqn-barB-condition} and \eqref{eqn-CB}, 
\begin{align*} 
\bigg| \frac{\bar{S}(t,x)}{\bar{B}(x)} - \frac{\bar{S}^N(t,x)}{\bar{B}^N(x)} \bigg| 
&= \bigg| \frac{\bar{S}(t,x)-\bar{S}^N(t,x) }{\bar{B}(x)} + \bar{S}^N(t,x) \bigg(  \frac{1}{\bar{B}(x)} - \frac{1}{\bar{B}^N(x)}\bigg)  \bigg|  \\
& \le c_B^{-1}  |\bar{S}(t,x)-\bar{S}^N(t,x)| + c_B^{-2} C_B|\bar{B}^N(x) - \bar{B}(x) | \,.
\end{align*}
Consequently, 
\begin{align*}
\bigg\| \frac{\bar{S}(t,\cdot)}{\bar{B}(\cdot)} &\int_0^1\beta(\cdot,y)\bar{\mathfrak{F}}(t,y)dy-\frac{\bar{S}^N(t,\cdot)}{\bar{B}^N(\cdot)}\int_0^1\beta^N(\cdot,y)\bar{\mathfrak{F}}^N(t,y)dy\bigg\|_1\\
&\le C_\beta C_T c_B^{-1}\|\bar{S}(t,\cdot)-\bar{S}^N(t,\cdot)\|_1 +  C_\beta C_T  c_B^{-2} C_B\|\bar{B}^N(\cdot) - \bar{B}(\cdot) \|_1  \\
& \quad +\left(\sup_{N,y}\int_0^1\beta^N(x,y)dx\right)\|\bar{\mathfrak{F}}(t,\cdot)-\bar{\mathfrak{F}}^N(t,\cdot)\|_1\\
&\quad+\int_0^1\left|\int_0^1(\beta(x,y)-\beta^N(x,y))\bar{\mathfrak{F}}(t,y)dx\right|dy
\,.
\end{align*}
We can now estimate the norm $\|\bar{S}(t,\cdot)-\bar{S}^N(t,\cdot)\|_1$ and $\|\bar{\mathfrak{F}}(t,\cdot)-\bar{\mathfrak{F}}^N(t,\cdot)\|_1$. Let $\bar{C}:=\max\{C_\beta,C_\beta C_Tc_B^{-1}, \\ C_\beta C_Tc_B^{-2} C_B \}$. 
We  now deduce from  \eqref{eqn-barSn-tx}, \eqref{eq:SF} and the last computation that 
\begin{align*}
\|\bar{S}(t,\cdot)-\bar{S}^N(t,\cdot)\|_1&\le\|\bar{S}(0,\cdot)-\bar{S}^N(0,\cdot)\|_1+\|\bar{M}^N_A(t,\cdot)\|_1 \\
 & \qquad +\int_0^t\int_0^1\left|\int_0^1(\beta(x,y)-\beta^N(x,y))\bar{\mathfrak{F}}(s,y)dx\right|dyds\\&\qquad+ 
\bar{C}\int_0^t\|\bar{S}(s,\cdot)-\bar{S}^N(s,\cdot)\|_1 ds + \bar{C} \|\bar{B}^N(\cdot) - \bar{B}(\cdot) \|_1  \\
&\qquad +\bar{C}\int_0^t\|\bar{\mathfrak{F}}(s,\cdot)-\bar{\mathfrak{F}}^N(s,\cdot)\|_1 ds.
\end{align*}
Next from \eqref{eqn-bar-mfFn-tx} and \eqref{eq:SF}, we get 
\begin{align*}
\|\bar{\mathfrak{F}}(t,\cdot)-\bar{\mathfrak{F}}^N(t,\cdot)\|_1&\le \|\bar{\mathfrak{F}}_0(t,\cdot)-\bar{\mathfrak{F}}^N_0(t,\cdot)\|_1+\|\mathcal{E}^N_{\mathfrak{F}}(t,\cdot)\|_1 \\
& \qquad +\int_0^t\int_0^1\left|\int_0^1(\beta(x,y)-\beta^N(x,y))\bar{\mathfrak{F}}(s,y)dx\right|dyds\\&\qquad+ 
\bar{C}\int_0^t\|\bar{S}(s,\cdot)-\bar{S}^N(s,\cdot)\|_1 ds + \bar{C} \|\bar{B}^N(\cdot) - \bar{B}(\cdot) \|_1 \\
& \qquad +\bar{C}\int_0^t\|\bar{\mathfrak{F}}(s,\cdot)-\bar{\mathfrak{F}}^N(s,\cdot)\|_1 ds\,.
\end{align*}
Adding those two inequalities, the result follows from our assumptions, the fact that \eqref{conv-beta} in Assumption 
\ref{AS-LLN-2} implies that
\[ \int_0^t\int_0^1\left|\int_0^1(\beta(x,y)-\beta^N(x,y))\bar{\mathfrak{F}}(s,y)dx\right|dyds\to0\quad \text{ as } \quad N\to\infty,\]
and
the following variant of Gronwall's Lemma: if $f(t)$ and $g(t)$
are nonnegative real-valued functions of $t$ and satisfy $f(t)\le g(t)+c\int_0^t f(s)ds$ for all $0\le t\le T$ and for some $c>0$, then for those $t$, $f(t)\le g(t)+c\int_0^te^{c(t-s)}g(s)ds$.
\end{proof}

 It remains to show that $\sup_{0\le t\le T}\Psi^N(t)\to0$ in probability, which follows from the 
  next three lemmas, where we establish the convergence of $\bar{\mathfrak{F}}^N_0(t,\cdot)$ to $ \bar{\mathfrak{F}}_0(t,x)$,  and that the stochastic terms  $ \bar{M}^N_A(t,x)$,  $\Delta^{N}_{1,1}(t,x)$ and $ \Delta^{N}_{1,2}(t,x)$  of \eqref{eqn-barM-An-tx}, \eqref{eqn-Delta-11} and \eqref{eqn-Delta-12} tend to $0$ in probability, as $N\to\infty$.


\begin{lemma} \label{lem-mfN0-conv}
Under Assumptions \ref{AS-LLN-1} and \ref{AS-lambda},  
\begin{equation} 
\|\bar{\mathfrak{F}}^N_0(t,\cdot) - \bar{\mathfrak{F}}_0(t,\cdot)\|_1 \to 0
\end{equation}
in probability, locally uniformly in $t$, as $N \to \infty$, where $\bar{\mathfrak{F}}_0(t,x) $ is defined in \eqref{eqn-bar-mfF-0}. 
\end{lemma}

\begin{proof}
We apply Theorem \ref{thm-D-conv-x}. 
First, we have
\[
 \bar{\mathfrak{F}}^N_0(t,x) - \bar{\mathfrak{F}}_0(t,x) = \Delta^{N}_{0,1}(t,x) +   \Delta^{N}_{0,2}(t,x) ,
\]
where 
\begin{align*}
 \Delta^{N}_{0,1}(t,x) &=\sum_{k=1}^{K^N}  \frac{K^N}{N} \sum_{j: -j\in \cT^N_k(0)} \Big(\lambda_{-j,k}(\tilde{\tau}^N_{-j,k}+ t) -  \bar\lambda(\tilde{\tau}^N_{-j,k}+ t) \Big){\bf 1}_{\mathtt{I}_k}(x)\,,  \\
  \Delta^{N}_{0,2}(t,x) &=  \sum_{k=1}^{K^N}  \frac{K^N}{N} \sum_{j: -j\in \cT^N_k(0)}   \bar\lambda(\tilde{\tau}^N_{-j,k}+ t){\bf 1}_{\mathtt{I}_k}(x) -   \int_0^{\bar{\mfa}} \bar{\lambda}(\mfa+t)  \bar{\sT}(0,d \mfa, x)\\
&  =  \int_0^{\bar{\mfa}} \bar{\lambda}(\mfa+t)  [\bar{\sT}^N(0,d \mfa, x)-\bar{\sT}(0,d \mfa, x)]\,.
\end{align*}

We now verify condition (i) of Theorem \ref{thm-D-conv-x}. 
For the first term $ \Delta^{N}_{0,1}(t,x)$,  we have
\begin{align*}
\|  \Delta^{N}_{0,1}(t,\cdot)\|_1 
& \le   \frac{1}{K^N} \sum_{k=1}^{K^N}  \frac{K^N}{N} \left|\sum_{j: -j\in \cT^N_k(0)}  \Big(\lambda_{-j,k}(\tilde{\tau}^N_{-j,k}+ t) -  \bar\lambda(\tilde{\tau}^N_{-j,k}+ t) \Big)\right|\,.
\end{align*}
Here the summands over $k$ are independent, and for each $k$, conditional on $\{\tilde{\tau}^N_{-j,k}\}_j$, the summands over $j$ are also independent and centered. 
Using Jensen's inequality for the sum over $k$, and the conditional independence for the sum over $j$, we deduce
\begin{align*}
&  \E\left[ \left( \frac{1}{K^N} \sum_{k=1}^{K^N} \frac{K^N}{N} \left|\sum_{j: -j\in \cT^N_k(0)}  \Big(\lambda_{-j,k}(\tilde{\tau}^N_{-j,k}+ t) -  \bar\lambda(\tilde{\tau}^N_{-j,k}+ t) \Big)\right| \right)^2 \right]  \\
& \le   \E\left[ \frac{1}{K^N}\sum_{k=1}^{K^N}  \frac{K^N}{N} \int_0^{\bar{\mfa}} v(\mfa+t) \bar{\sT}_k^N(0,d \mfa)\right] \to 0 \qasq N \to \infty,
\end{align*}
since under Assumption \ref{AS-LLN-1},  thanks to Lemma \ref{convPort}, 
\[
\frac{1}{K^N}\sum_{k=1}^{K^N}   \int_0^{\bar{\mfa}} v(\mfa+t)\bar{\sT}_k^N(0,d \mfa)  \to   \int_0^1 \int_0^{\bar{\mfa}} v(\mfa+t) \bar{\sT}(0,d \mfa,x)dx 
\]
in probability  and $\frac{K^N}{N} \to 0 $ as $N \to \infty$. Recall that $v(t)$ is the variance of the random function $\lambda(t)$ in Assumption \ref{AS-lambda}, which is bounded.

The fact that $\|\Delta^N_{0,2}\|_1\to0$ in probability follows again from Lemma \ref{convPort} and Assumption \ref{AS-LLN-1}. 
\smallskip

Now to check condition (ii) of Theorem \ref{thm-D-conv-x}, we first have for $t, u >0$, 
\begin{align*}
 & \Delta^{N}_{0,1}(t+u,x) -  \Delta^{N}_{0,1}(t,x) \\
 & = \sum_{k=1}^{K^N}  \frac{K^N}{N} \sum_{j: -j \in \cT^N_k(0)} \Big(\lambda_{-j,k}(\tilde{\tau}^N_{-j,k}+ t+u) -  \lambda_{-j,k}(\tilde{\tau}^N_{-j,k}+ t)  \Big){\bf 1}_{\mathtt{I}_k}(x) \\
 & \quad - \sum_{k=1}^{K^N} \frac{K^N}{N} \sum_{j: -j \in \cT^N_k(0)} \Big( \bar\lambda(\tilde{\tau}^N_{-j,k}+ t+u) - \bar\lambda(\tilde{\tau}^N_{-j,k}+ t) \Big){\bf 1}_{\mathtt{I}_k}(x)\,. 
 \end{align*}
Observe that
\begin{align*}
& \left\|\sum_{k=1}^{K^N} \frac{K^N}{N} \sum_{j: -j \in \cT^N_k(0)} \Big(\lambda_{-j,k}(\tilde{\tau}^N_{-j,k}+ t+u) - \lambda_{-j,k}(\tilde{\tau}^N_{-j,k}+ t)  \Big){\bf 1}_{\mathtt{I}_k}(x)  \right\|_1\\
&\quad \le \frac{1}{K^N} \sum_{k=1}^{K^N} \frac{K^N}{N} \sum_{j: -j \in \cT^N_k(0)} \Big|\lambda_{-j,k}(\tilde{\tau}^N_{-j,k}+ t+u) - \lambda_{-j,k}(\tilde{\tau}^N_{-j,k}+ t)  \Big| \,, 
\end{align*}
and similarly for the second term. Thus, 
\begin{align*}
 \| \Delta^{N}_{0,1}(t+u,x) -  \Delta^{N}_{0,1}(t,x)\|_1 & \le \frac{1}{K^N} \sum_{k=1}^{K^N} \frac{K^N}{N} \sum_{j: -j \in \cT^N_k(0)} \Big|\lambda_{-j,k}(\tilde{\tau}^N_{-j,k}+ t+u) - \lambda_{-j,k}(\tilde{\tau}^N_{-j,k}+ t)  \Big| \\
 & \quad + \frac{1}{K^N} \sum_{k=1}^{K^N}  \frac{K^N}{N} \sum_{j: -j \in \cT^N_k(0)}   \Big| \bar\lambda(\tilde{\tau}^N_{-j,k}+ t+u) - \bar\lambda(\tilde{\tau}^N_{-j,k}+ t) \Big| \\
 &=:  \Delta^{N, (1)}_{0,1}(t,u) + \Delta^{N, (2)}_{0,1}(t,u)\,.
 \end{align*}
 By Assumption \ref{AS-lambda}, using the expression of $\lambda(t)$ in \eqref{eqn-lambda-assump}, that is, $\lambda_{-j,k}(t) = \sum_{\ell=1}^\kappa \lambda^\ell_{-j,k}(t) \bone_{[\zeta_{-j,k}^{\ell-1},\zeta_{-j,k}^\ell)}(t)$, we obtain 
\begin{align} \label{eqn-Delta01-1-bound} 
& \Delta^{N, (1)}_{0,1}(t,u)  \\
&=  \frac{1}{K^N}\sum_{k=1}^{K^N} \frac{K^N}{N} \sum_{j: -j \in \cT^N_k(0)} \bigg|\sum_{\ell=1}^\kappa \lambda^\ell_{-j,k}(\tilde{\tau}^N_{-j,k}+ t+u) \bone_{[\zeta_{-j,k}^{\ell-1},\zeta_{-j,k}^\ell)}(\tilde{\tau}^N_{-j,k}+ t+u)  \non \\
& \qquad \qquad \qquad \qquad - \sum_{\ell=1}^\kappa \lambda^\ell_{-j,k}(\tilde{\tau}^N_{-j,k}+ t) \bone_{[\zeta_{-j,k}^{\ell-1},\zeta_{-j,k}^\ell)}(\tilde{\tau}^N_{-j,k}+ t) \bigg|  \non \\
&\le  \frac{1}{K^N} \sum_{k=1}^{K^N}  \frac{K^N}{N} \sum_{j: -j \in \cT^N_k(0)} \sum_{\ell=1}^\kappa \Big| \lambda^\ell_{-j,k}(\tilde{\tau}^N_{-j,k}+ t+u) - \lambda^\ell_{-j,k}(\tilde{\tau}^N_{-j,k}+ t)\Big| \bone_{\zeta_{-j,k}^{\ell-1} \le \tilde{\tau}^N_{-j,k}+ t \le \tilde{\tau}^N_{-j,k}+ t+u \le \zeta_{-j,k}^\ell} \non  \\
& \quad +  \lambda^* \frac{1}{K^N} \sum_{k=1}^{K^N}  \frac{K^N}{N} \sum_{j: -j \in \cT^N_k(0)} \sum_{\ell=1}^\kappa  
 \bone_{\tilde{\tau}^N_{-j,k}+ t \le \zeta_{-j,k}^\ell \le \tilde{\tau}^N_{-j,k}+ t+u } \non \\
 & \le \varphi(u)  \frac{1}{K^N} \sum_{k=1}^{K^N}  (\bar{I}^N_k(0) + \bar{R}^N_k(0))+  \lambda^* \frac{1}{K^N} \sum_{k=1}^{K^N}  \frac{K^N}{N} \sum_{j: -j \in \cT^N_k(0)}  \sum_{\ell=1}^\kappa  
 \bone_{\tilde{\tau}^N_{-j,k}+ t \le \zeta_{-j,k}^\ell \le \tilde{\tau}^N_{-j,k}+ t+u }\,.
\end{align}
Since both terms in the right hand side are increasing in $u$, we obtain
\begin{align}\label{eqn-mfN0-conv-p0}
\sup_{u \in [0,\delta]}\Delta^{N, (1)}_{0,1}(t,u)  &  \le \varphi(\delta) \frac{1}{K^N} \sum_{k=1}^{K^N}  (\bar{I}^N_k(0) + \bar{R}^N_k(0) )  \nonumber \\
&\quad + \lambda^*  \sum_{\ell=1}^\kappa \frac{1}{K^N} \sum_{k=1}^{K^N} \frac{K^N}{N} \sum_{j: -j \in \cT^N_k(0)} 
 \bone_{\tilde{\tau}^N_{-j,k}+ t \le \zeta_{-j,k}^\ell \le \tilde{\tau}^N_{-j,k}+ t+\delta}\,. 
\end{align}
Note that \[\frac{1}{K^N} \sum_{k=1}^{K^N}  ( \bar{I}^N_k(0)  + \bar{R}^N_k(0) ) \to \int_0^1 ( \bar{I}(0,x)+ \bar{R}(0,x)) dx \qasq N \to \infty\]  under Assumption \ref{AS-LLN-1}. 
For the second term in \eqref{eqn-mfN0-conv-p0}, we have
\begin{align}\label{eqn-mfN0-conv-p1}
&   \sum_{\ell=1}^\kappa \frac{1}{K^N} \sum_{k=1}^{K^N} \frac{K^N}{N} \sum_{j: -j \in \cT^N_k(0)}  
 \bone_{\tilde{\tau}^N_{-j,k}+ t \le \zeta_{-j,k}^\ell \le \tilde{\tau}^N_{-j,k}+ t+\delta  } \non\\
 & =    \sum_{\ell=1}^\kappa \frac{1}{K^N} \sum_{k=1}^{K^N}  \frac{K^N}{N} \sum_{j: -j \in \cT^N_k(0)}    
\Big[  \bone_{\tilde{\tau}^N_{-j,k}+ t \le \zeta_{-j,k}^\ell \le \tilde{\tau}^N_{-j,k}+ t+\delta  }  - \Big(F_\ell(\tilde{\tau}^N_{-j,k}+ t+\delta) -F_\ell(\tilde{\tau}^N_{-j,k}+ t)   \Big)  \Big] \non \\
& \quad +  \sum_{\ell=1}^\kappa \frac{1}{K^N} \sum_{k=1}^{K^N}  \frac{K^N}{N} \sum_{j: -j \in \cT^N_k(0)}  
\Big(F_\ell(\tilde{\tau}^N_{-j,k}+ t+\delta) -F_\ell(\tilde{\tau}^N_{-j,k}+ t)   \Big) \,. 
\end{align}
In the first expression, for each $k$, conditional on $\{\tilde{\tau}^N_{-j,k}\}_j$, the summands over $j$ are  independent. 
We have
\begin{align*}
& \E\left[ \left(\frac{1}{K^N} \sum_{k=1}^{K^N}  \frac{K^N}{N} \sum_{j: -j \in \cT^N_k(0)}    
\Big[  \bone_{\tilde{\tau}^N_{-j,k}+ t \le \zeta_{-j,k}^\ell \le \tilde{\tau}^N_{-j,k}+ t+\delta  }  - \Big(F_\ell(\tilde{\tau}^N_{-j,k}+ t+\delta) -F_\ell(\tilde{\tau}^N_{-j,k}+ t)   \Big)  \Big] \right)^2\right] \\
& \le    \E \left[  \frac{1}{K^N} \sum_{k=1}^{K^N}  \Big(\frac{K^N}{N}\Big)^2 \Bigg( \sum_{j: -j \in \cT^N_k(0)}  
\Big[  \bone_{\tilde{\tau}^N_{-j,k}+ t \le \zeta_{-j,k}^\ell \le \tilde{\tau}^N_{-j,k}+ t+\delta  }  - \Big(F_\ell(\tilde{\tau}^N_{-j,k}+ t+\delta) -F_\ell(\tilde{\tau}^N_{-j,k}+ t)   \Big)  \Big] \Bigg)^2  \right] \\
& =    \E \left[  \frac{1}{K^N} \sum_{k=1}^{K^N}  \Big(\frac{K^N}{N}\Big)^2 \sum_{j: -j \in \cT^N_k(0)}    
\Big[  \bone_{\tilde{\tau}^N_{-j,k}+ t \le \zeta_{-j,k}^\ell \le \tilde{\tau}^N_{-j,k}+ t+\delta  }  - \Big(F_\ell(\tilde{\tau}^N_{-j,k}+ t+\delta) -F_\ell(\tilde{\tau}^N_{-j,k}+ t)   \Big)  \Big]^2  \right] \\
&=   \E \Bigg[  \frac{1}{K^N} \sum_{k=1}^{K^N} \Big(\frac{K^N}{N}\Big)^2 \sum_{j: -j \in \cT^N_k(0)}   
\Big[  \Big(F_\ell(\tilde{\tau}^N_{-j,k}+ t+\delta) -F_\ell(\tilde{\tau}^N_{-j,k}+ t)   \Big)  \\
& \qquad \qquad \qquad \qquad \qquad \qquad  \times \Big(1-\Big(F_\ell(\tilde{\tau}^N_{-j,k}+ t+\delta) -F_\ell(\tilde{\tau}^N_{-j,k}+ t)   \Big) \Big)  \Big]  \Bigg] \\
& \le \frac{1}{K^N} \sum_{k=1}^{K^N} \Big(\frac{K^N}{N}\Big)^2  ( I^N_k(0) + R^N_k(0)) \frac{1}{4}   \\
& \le  \frac{1}{4}  \frac{1}{K^N} \Big(\frac{K^N}{N}\Big)^2 N  =  \frac{1}{4}\frac{K^N}{N} \\
&\to 0\qasq N \to \infty\,. 
\end{align*}
Hence, the first term in \eqref{eqn-mfN0-conv-p1} converges to zero in probability as $N\to\infty$. For the second term in \eqref{eqn-mfN0-conv-p1}, we have 
\begin{align*}
&   \sum_{\ell=1}^\kappa \frac{1}{K^N} \sum_{k=1}^{K^N}   \frac{K^N}{N} \sum_{j: -j \in \cT^N_k(0)} 
\Big(F_\ell(\tilde{\tau}^N_{-j,k}+ t+\delta) -F_\ell(\tilde{\tau}^N_{-j,k}+ t)   \Big) \\
& =  \sum_{\ell=1}^\kappa \frac{1}{K^N} \sum_{k=1}^{K^N}  \int_0^{\bar{\mfa}} 
\Big(F_\ell(\mfa+ t+\delta) -F_\ell(\mfa+ t)   \Big)  \bar{\sI}^N_k(0, d\mfa) \\
& \to   \sum_{\ell=1}^\kappa\int_0^1 \int_0^{\bar{\mfa}} 
\Big(F_\ell(\mfa+ t+\delta) -F_\ell(\mfa+ t)   \Big)  \bar{\sI}(0, d\mfa, x)dx \,,
\end{align*}
in probability as $N\to \infty$. 
For each $\ell=1,\dots, \kappa$,  the function $\delta \to  \int_0^1 \int_0^{\bar{x}} 
\Big(F_\ell(\mfa+ t+\delta) -F_\ell(\mfa+ t)   \Big)  \bar{\sI}(0, d\mfa, x) dx$ is continuous and equal to zero at $\delta=0$.
Thus we have shown that for any $\ep>0$, there exists $\delta>0$ small enough such that 
\begin{align} \label{eqn-mfN0-conv-p2}
\limsup_{N\to\infty}\sup_{t\in [0,T]}  \frac{1}{\delta}\P \left( \sup_{0 \le u \le \delta}  \Delta^{N, (1)}_{0,1}(t,u)  > \ep/2\right)  = 0. 
\end{align}

 Note that
\begin{align}\label{201}
\Delta^{N, (2)}_{0,1}(t,u)=\int_0^1\int_0^{\bar{\mfa}}  \left|\bar{\lambda}(\mfa+t+u)-\bar{\lambda}(\mfa+t)\right|\bar{\sI}^N(0,d\mfa,x)dx\,. 
\end{align} 
By similar calculations leading to \eqref{eqn-mfN0-conv-p0}, we obtain  for any small enough $\delta>0$, 
\begin{align*}
\sup_{u \in [0,\delta]}\Delta^{N, (2)}_{0,1}(t,u)  & \le \varphi(\delta) \frac{1}{K^N} \sum_{k=1}^{K^N}  \bar{I}^N_k(0) \\
& \qquad + \lambda^*  \sum_{\ell=1}^\kappa  \frac{1}{K^N} \sum_{k=1}^{K^N}  \frac{K^N}{N}  \sum_{j: -j \in \cT^N_k(0)}   
\Big(F_\ell(\tilde{\tau}^N_{-j,k}+ t+\delta) -F_\ell(\tilde{\tau}^N_{-j,k}+ t)   \Big) \,. 
\end{align*}
Thus, by the same arguments for these two terms as in the proof for \eqref{eqn-mfN0-conv-p2}, we obtain that \eqref{eqn-mfN0-conv-p2} holds for $\Delta^{N, (2)}_{0,1}(t,u)$. 
Thus, combining these two results, we obtain that for any $\ep>0$, for $\delta>0$ small enough, 
\begin{align} \label{eqn-mfN0-conv-p3}
\limsup_{N\to\infty}\sup_{t\in [0,T]}  \frac{1}{\delta}\P \left( \sup_{0 \le u \le \delta}   \| \Delta^{N}_{0,1}(t+u,x) -  \Delta^{N}_{0,1}(t,x)\|_1   > \ep\right)  =0\,. 
\end{align}

Now for $\Delta^{N}_{0,2}(t,x)$, we have for $t,u>0$,
\begin{align*}
& \|\Delta^{N}_{0,2}(t+u,x) - \Delta^{N}_{0,2}(t,x)\|_1 \\
&\le\int_0^1\int_0^{\bar{\mfa}}  |\bar{\lambda}(\mfa+t+u)-\bar{\lambda}(\mfa+t)|[\bar{\sI}^N(0,d\mfa,x)+\bar{\sI}(0,d\mfa,x)]dx,
\end{align*}
which  is treated exactly as $\Delta^{N, (2)}_{0,1}(t,u)$, see formula \eqref{201}. 
This completes the proof of the lemma. 
\end{proof}

\begin{lemma} \label{lem-barAn-tight}
Under Assumptions \ref{AS-LLN-1}, \ref{AS-LLN-2} and \ref{AS-lambda}, 
 for all $T>0$, 
 \begin{align}\label{MAto0}  
\E\bigg[\sup_{t \in [0,T]} \| \bar{M}^N_A(t,\cdot)\|^2_1\bigg]\to0\, ,
\end{align}
and thus,
\begin{equation}\label{eqn-barA-Upsilon-int}
\Big\| \bar{A}^N(t,\cdot) - \int_0^t \bar{\Upsilon}^N(s,\cdot) ds\Big\|_1 \to 0
\end{equation}
in probability, locally uniformly in $t$.

In addition,  there exists $C_T>0$ such that for all $N\ge1$,
\begin{equation}\label{momentestim}
\E\left[\sup_{t\le T}\|\bar{A}^N(t,\cdot)\|_1\right] \le C_T\,.
\end{equation}


\end{lemma}

\begin{proof}
Recall the expressions of $A^N_k(t)$ in \eqref{eqn-An-k-rep} and $\Upsilon^N_k(t)$ in \eqref{eqn-upsilon}. By \eqref{eqn-mfk}, under Assumption \ref{AS-lambda} that $\lambda(t) \le \lambda^*$,  under the condition on $\bar{B}(x)$ in \eqref{eqn-barB-condition}, and  \eqref{eqn-CB}, we have $\bar{\mathfrak{F}}^N(t,x) \le \lambda^* C_B$ and thus, under Assumption \ref{AS-LLN-2}, $\bar\Upsilon^N(t,x)\le \lambda^* C_BC_\beta $, where we have used \eqref{eqn-C-beta}. 
Hence 
\begin{equation} \label{eqn-int-Upsilon-bound-1} 
\| \bar\Upsilon^N(t,\cdot) \|_1\le\lambda^* C_BC_\beta,
\end{equation}
 and
\begin{align} \label{eqn-int-Upsilon-bound} 
\Big\| \int_0^t \bar\Upsilon^N(r,\cdot)dr -\int_0^s \bar\Upsilon^N(r,\cdot)dr  \Big\|_1  \le \lambda^*  C_BC_\beta (t-s) \,. 
\end{align}

For each $k$, we can write  
\begin{equation*}\label{eqn-barAn-decomp}
\bar{A}^N_k(t) = \int_0^t \bar\Upsilon^N_k(s)ds + \bar{M}^N_{A,k}(t)
\end{equation*}
where 
\[
\bar{M}^N_{A,k}(t) = \frac{K^N}{N}\int_0^t \int_0^\infty \bone_{u \le \Upsilon^N_k(s^-)}\bar{Q}_k(ds,du)
\]
with $\bar{Q}_k(ds,du) = Q_k(ds,du) - ds du$ being the compensated PRM associated with $Q_k$. 
We have the  representation: 
\begin{equation} \label{eqn-barAN}
\bar{A}^N(t,x)=\int_0^t \bar\Upsilon^N(r,x)dr+\bar{M}^N_A(t,x)\,. 
\end{equation}
It is clear that for each $k$, $\{\bar{M}^N_{A,k}(t): t \ge 0\}$ is a square-integrable martingale with respect to 
the filtration $\sF^N_A= \{\sF^N_A(t): t\ge0\}$ where 
\begin{align*}
\sF^N_A(t) &:=  \sigma\big\{ I^N_k(0), \tilde{\tau}^N_{-j,k}: j =1,\dots, I^N_k(0), k=1,\dots,K\big\} \vee \sigma \big\{\lambda_{j,k}(\cdot), \, j \in \ZZ\setminus\{0\},  k=1,\dots, K \big\} \\
 &\qquad \vee \sigma\bigg\{ \int_0^{t'} \int_0^\infty  \bone_{u \le \Upsilon^N_k(s^{-})} Q_k(ds,du):  0 \le t' \le t, \, k=1, \dots,K \bigg\}, 
\end{align*}
and has the quadratic variation 
\begin{equation*} \label{eqn-MA-qv}
\langle \bar{M}^N_{A,k} \rangle(t) = \frac{K^N}{N}\int_0^t \bar\Upsilon^N_k(s)ds, \quad t \ge 0. 
\end{equation*}
Then, 
\begin{align}\label{estimM}
\| \bar{M}^N_A(t,\cdot)\|_1 
= \int_0^1 \bigg|\sum_{k=1}^{K^N} \bar{M}^N_{A,k}(t){\bf 1}_{\mathtt{I}_k}(x)\bigg| dx   = \frac{1}{K^N} \sum_{k=1}^{K^N} \big|\bar{M}^N_{A,k}(t) \big|\,. 
\end{align}
By Doob's inequality for submartingales,
\begin{equation*} \label{eqn-barM-supbound-conv0}
\E\bigg[\sup_{t \in [0,T]}  \big|\bar{M}^N_{A,k}(t) \big|^2 \bigg] \le  4 \E\big[\big|\bar{M}^N_{A,k}(T) \big|^2\big] = 4
\E\left[  \frac{K^N}{N} \int_0^T \bar\Upsilon^N_k(s)ds\right] \le   4 \lambda^*C_B C_\beta  T  \frac{K^N}{N} \,.
\end{equation*}
Since $ \frac{K^N}{N} \to 0$ as $N\to\infty$, the last inequality entails that as $N\to\infty$,
\[ \sup_{1\le k\le K^N} \E\bigg[\sup_{t \in [0,T]}  \big|\bar{M}^N_{A,k}(t) \big|^2 \bigg]\to0.\]
This combined with \eqref{estimM} implies that \eqref{MAto0} holds. 

Note that the above computations, combined with \eqref{eqn-barAN} and \eqref{eqn-int-Upsilon-bound-1}, yield \eqref{momentestim}.

Finally \eqref{eqn-barA-Upsilon-int} follows directly from \eqref{eqn-barAN} and \eqref{MAto0}. 
\end{proof}

We finally show that $ \Delta^{N}_{1,1}(t,\cdot)$ and $ \Delta^{N}_{1,2}(t,\cdot)$ tend to $0$.

\begin{lemma} \label{lem-mfN1-conv} 
Under Assumptions \ref{AS-LLN-1}, \ref{AS-LLN-2} and \ref{AS-lambda},   as $N\to \infty$, both $ \Delta^{N}_{1,1}(t,\cdot)$ and $ \Delta^{N}_{1,2}(t,\cdot)$ defined in \eqref{eqn-Delta-11} and \eqref{eqn-Delta-12} converge to zero 
in $L^1([0,1])$ in probability, locally uniformly in $t$. 
\end{lemma}

\begin{proof}

We apply Theorem \ref{thm-D-conv-x}. 
We first consider $ \Delta^{N}_{1,1}(t,x)$. 
To verify condition (i) of Theorem \ref{thm-D-conv-x}, 
we have 
\begin{align*}
\| \Delta^{N}_{1,1}(t,\cdot)\|_1 & \le \frac{1}{K^N}\sum_{k=1}^{K^N}   \frac{K^N}{N}  \left|\sum_{j=1}^{A^N_k(t)}  \left(\lambda_{j,k} (t-\tau^N_{j,k}) - \bar\lambda(t-\tau^N_{j,k})\right)\right|\,.
\end{align*}

Recall the expression of $A^N_k(t)$ in \eqref{eqn-An-k-rep} and the associated  $\Upsilon^N_k(t) $ in \eqref{eqn-upsilon}. It is clear that the summands over $k$ are not independent due to the interactions among individuals in different locations in the infection process. 
Using first Jensen's inequality, and then the fact that
for each $k$, conditional on the arrivals $\{\tau^N_{j,k}\}_j$, the summands over $j$ are independent  and centered, we have
\begin{align*}
& \E \left[ \left( \frac{1}{K^N}\sum_{k=1}^{K^N}  \frac{K^N}{N}  \left|\sum_{j=1}^{A^N_k(t)} \Big( \lambda_{j,k} (t-\tau^N_{j,k}) - \bar\lambda(t-\tau^N_{j,k})\Big) \right|\right)^2 \right] \\
& \le  \E \left[  \frac{1}{K^N}\sum_{k=1}^{K^N} \left(  \frac{K^N}{N} \sum_{j=1}^{A^N_k(t)} \Big( \lambda_{j,k} (t-\tau^N_{j,k}) - \bar\lambda(t-\tau^N_{j,k})\Big) \right)^2 \right] \\
& =  \E \left[  \frac{1}{K^N}\sum_{k=1}^{K^N}  \Big(\frac{K^N}{N}\Big)^2 \sum_{j=1}^{A^N_k(t)} \Big| \lambda_{j,k} (t-\tau^N_{j,k}) - \bar\lambda(t-\tau^N_{j,k})\Big|^2 \right]\\
&=  \E \left[  \frac{1}{K^N}\sum_{k=1}^{K^N} \Big(\frac{K^N}{N}\Big)^2 \int_0^t v(t-s) d A^N_k(s) \right] \\
& \le (\lambda^*)^2  \E \left[ \frac{1}{K^N}\sum_{k=1}^{K^N} \frac{K^N}{N}   \bar{A}^N_k(t)   \right] \\
& =  (\lambda^*)^2  \frac{K^N}{N}  \E \left[ \|\bar{A}^N(t,\cdot)\|_{1}   \right] \to0 \qasq N \to \infty\,, 
\end{align*} 
where we used $v(t) \le (\lambda^*)^2$ under Assumption \ref{AS-lambda}, and the convergence follows from 
the assumption that $\frac{K^N}{N} \to 0$ as $N\to\infty$, and \eqref{momentestim} in Lemma \ref{lem-barAn-tight}.

\smallskip

We next check condition (ii) in Theorem \ref{thm-D-conv-x} for $ \Delta^{N}_{1,1}(t,x)$.
We have 
\begin{align*}
 \Delta^{N}_{1,1}(t+u,x) -  \Delta^{N}_{1,1}(t,x) &= \sum_{k=1}^{K^N} \frac{K^N}{N} \sum_{j=1}^{A^N_k(t)} \Big( \lambda_{j,k} (t+u-\tau^N_{j,k}) - \lambda_{j,k} (t-\tau^N_{j,k}) \Big){\bf 1}_{\mathtt{I}_k}(x) \\
& \quad -   \sum_{k=1}^{K^N} \frac{K^N}{N} \sum_{j=1}^{A^N_k(t)} \Big( \bar\lambda(t+u-\tau^N_{j,k})- \bar\lambda(t-\tau^N_{j,k})\Big){\bf 1}_{\mathtt{I}_k}(x) \\
 & \quad + \sum_{k=1}^{K^N} \frac{K^N}{N} \sum_{j=A^N_k(t)+1}^{A^N_k(t+u)} \Big( \lambda_{j,k} (t+u-\tau^N_{j,k}) - \bar\lambda(t+u-\tau^N_{j,k})\Big){\bf 1}_{\mathtt{I}_k}(x)\,,
\end{align*} 
and 
\begin{align*}
 \|\Delta^{N}_{1,1}(t+u,x) -  \Delta^{N}_{1,1}(t,x) \|_1 
 & \le  \frac{1}{K^N}\sum_{k=1}^{K^N}  \frac{K^N}{N} \sum_{j=1}^{A^N_k(t)} \Big| \lambda_{j,k} (t+u-\tau^N_{j,k}) - \lambda_{j,k} (t-\tau^N_{j,k}) \Big|  \\
& \quad + \frac{1}{K^N} \sum_{k=1}^{K^N} \frac{K^N}{N} \sum_{j=1}^{A^N_k(t)} \Big| \bar\lambda(t+u-\tau^N_{j,k})- \bar\lambda(t-\tau^N_{j,k})\Big| \\
 & \quad + \frac{1}{K^N} \sum_{k=1}^{K^N} \frac{K^N}{N}  \sum_{j=A^N_k(t)+1}^{A^N_k(t+u)} \Big| \lambda_{j,k} (t+u-\tau^N_{j,k}) - \bar\lambda(t+u-\tau^N_{j,k})\Big|\\
 & =:  \Delta^{N,(1)}_{1,1}(t,u) + \Delta^{N,(2)}_{1,1}(t,u) + \Delta^{N,(3)}_{1,1}(t,u)\,. 
\end{align*}

Similar to $\Delta^{N,(1)}_{0,1}(t,u)$ in \eqref{eqn-Delta01-1-bound}, we have
\begin{align*}
\sup_{u \in [0,\delta]}\Delta^{N,(1)}_{1,1}(t,u) &\le \varphi(\delta)  \int_0^1\bar{A}^N(t,x)dx +  \lambda^* \frac{1}{K^N} \sum_{k=1}^{K^N}  \frac{K^N}{N}  \sum_{j=1}^{A^N_k(t)} \sum_{\ell=1}^\kappa  
 \bone_{t-\tau^N_{j,k} \le \zeta_{j,k}^\ell \le  t+\delta -\tau^N_{j,k} }\,.
\end{align*} 
We note that
\begin{align*}
\int_0^1\bar{A}^N(t,x)dx&= \int_0^1\int_0^t\bar{\Upsilon}^N(s,x)dsdx+\int_0^1\bar{M}^N_A(t,x)dx\\
 &\le\lambda^\ast C_B C_\beta t+\int_0^1\bar{M}^N_A(t,x)dx\,.
\end{align*}
Hence, we deduce from \eqref{MAto0} that as soon as $\delta>0$ is small enough such that $\varphi(\delta)\lambda^\ast C_B C_\beta t<\ep/6$,
\begin{equation}\label{equals0}
\limsup_{N}\frac{1}{\delta}\P\left(\varphi(\delta)  \int_0^1\bar{A}^N(t,x)dx>\ep/6\right)=0\,.
\end{equation}
 For the second term, we have
\begin{align*}
& \E\left[\left(\sum_{\ell=1}^\kappa   \frac{1}{K^N} \sum_{k=1}^{K^N}  \frac{K^N}{N} \sum_{j=1}^{A^N_k(t)}  
\bone_{t-\tau^N_{j,k} \le \zeta_{j,k}^\ell \le  t+\delta -\tau^N_{j,k} } \right)^2\right ] \\
 & \le 2 \E\left[\left(\sum_{\ell=1}^\kappa   \frac{1}{K^N} \sum_{k=1}^{K^N} \frac{K^N}{N}
 \int_0^t  \int_0^\infty \int_{t-s}^{t+\delta -s} \bone_{r \le \Upsilon^N_k(s^-)} \overline{Q}_{k,\ell}(ds,dr,d\zeta)\right)^2\right ] \\
 & \quad + 2 \E\left[\left(\sum_{\ell=1}^\kappa   \frac{1}{K^N} \sum_{k=1}^{K^N} \frac{K^N}{N} 
 \int_0^t \Big(F_\ell( t+\delta-s ) - F_\ell( t-s )\Big) \Upsilon^N_k(s) ds \right)^2\right ] 
\end{align*} 
where  $Q_{k,\ell}(ds,dr,d\zeta)$ is a PRM on $\RR_+^3$ with mean measure $dsdrF_\ell(d\zeta)$  whose projection on the first two coordinates is $Q_k$, and  $\overline{Q}_{k,\ell}(ds,dr,d\zeta)$ is the corresponding compensated PRM. 
Observe that 
\begin{align*}
&  \E\left[\left(   \sum_{\ell=1}^\kappa  \frac{1}{K^N} \sum_{k=1}^{K^N}  \frac{K^N}{N}
 \int_0^t \int_0^\infty \int_{t-s}^{t+\delta -s} \bone_{r \le \Upsilon^N_k(s^-)} \overline{Q}_{k,\ell}(ds,dr,d\zeta)\right)^2\right ] \\
 & \le \kappa \sum_{\ell=1}^\kappa \E\left[    \frac{1}{K^N} \sum_{k=1}^{K^N}  \left(  \frac{K^N}{N} 
 \int_0^t \int_0^\infty \int_{t-s}^{t+\delta -s} \bone_{r \le \Upsilon^N_k(s^-)} \overline{Q}_{k,\ell}(ds,dr,d\zeta)\right)^2\right ] \\
 & = \kappa \sum_{\ell=1}^\kappa \frac{1}{K^N} \sum_{k=1}^{K^N}  \Big(\frac{K^N}{N}\Big)^2  \E\left[  \int_0^t \Big(F_\ell( t+\delta-s ) - F_\ell( t-s )\Big) \Upsilon^N_k(s) ds \right ]  \\
 & \le  \lambda^* C_BC_\beta \kappa \sum_{\ell=1}^\kappa \frac{1}{K^N} \sum_{k=1}^{K^N}  \frac{K^N}{N}\int_0^t \Big(F_\ell( t+\delta-s ) - F_\ell( t-s )\Big)ds \\
 & \le \lambda^* C_B C_\beta \kappa^2 \delta   \frac{K^N}{N}    \to 0 \qasq N \to \infty,
\end{align*} 
where we have used the inequality
\begin{align}\label{delta}
0\le\int_0^t[F_\ell(s+\delta)-F_\ell(s)]ds\le\int_0^{t+\delta}F_\ell(s)ds-\int_0^tF_\ell(s)ds\le\delta\,,
\end{align}
and 
\begin{align*}
& \E\left[\left(\sum_{\ell=1}^\kappa \frac{1}{K^N} \sum_{k=1}^{K^N}  \frac{K^N}{N} 
 \int_0^t \Big(F_\ell( t+\delta-s ) - F_\ell( t-s )\Big) \Upsilon^N_k(s) ds \right)^2\right ] \\
& \le\kappa \sum_{\ell=1}^\kappa\frac{1}{K^N} \sum_{k=1}^{K^N}   \E\left[\left(\frac{K^N}{N} 
 \int_0^t \Big(F_\ell( t+\delta-s ) - F_\ell( t-s )\Big) \Upsilon^N_k(s) ds \right)^2\right ]  \\
 &\le \kappa ( \lambda^* C_B C_\beta  )^2  \sum_{\ell=1}^\kappa\left(\int_0^t[F_\ell(s+\delta)-F_\ell(s)]ds\right)^2\\
 &\le(\kappa\lambda^\ast C_B C_\beta \delta)^2  \,. 
\end{align*}

This combined with \eqref{equals0} shows that
\begin{align} \label{eqn-mfN1-conv11}
\limsup_{N\to\infty}\sup_{t\in [0,T]}  \frac{1}{\delta}\P \left( \sup_{0 \le u \le \delta}  \Delta^{N, (1)}_{1,1}(t,u)  > \ep/3\right)  \to 0\,\quad  \text{ as } \quad \delta\to0\, . 
\end{align}

Next, similar to $\Delta^{N,(1)}_{0,1}(t,u)$ in \eqref{eqn-Delta01-1-bound}, we have
\begin{align*}
\sup_{u \in [0,\delta]} \Delta^{N,(2)}_{1,1}(t,u) &\le \varphi(\delta)  \frac{1}{K^N} \sum_{k=1}^{K^N}  \bar{A}^N_k(t) +  \lambda^* \frac{1}{K^N} \sum_{k=1}^{K^N}  \frac{K^N}{N} \sum_{j=1}^{A^N_k(t)} \sum_{\ell=1}^\kappa  
\Big(F_\ell(t+\delta -\tau^N_{j,k} ) - F_\ell( t-\tau^N_{j,k} )\Big) \,.
\end{align*} 
Then using the same arguments leading to \eqref{eqn-mfN1-conv11}, we obtain that \eqref{eqn-mfN1-conv11} holds for $ \Delta^{N,(2)}_{1,1}(t,u) $. 

Finally, for $\Delta^{N,(3)}_{1,1}(t,u)$, we have 
\begin{align*}
\sup_{0 \le u \le \delta}\Delta^{N,(3)}_{1,1}(t,u) 
& \le  \lambda^* \frac{1}{K^N} \sum_{k=1}^{K^N}   (\bar{A}^N_k(t+\delta) - \bar{A}^N_k(t)) \\
& = \lambda^* \int_0^1 \int_t^{t+\delta} \bar{A}^N(ds, x)dx\,.
\end{align*}
So
\begin{align*}
\P\left(\sup_{0 \le u \le \delta}\Delta^{N,(3)}_{1,1}(t,u)>\ep/3\right)
&\le\frac{18(\lambda^\ast)^2}{\ep^2}\Bigg\{\E\left[\left(\int_0^1\int_t^{t+\delta}\bar{\Upsilon}^N(s,x)dsdx\right)^2\right] \\
& \qquad \qquad \qquad  +
\E\left[\|\bar{M}_A^N(t+\delta,\cdot)-\bar{M}_A^N(t,\cdot)\|_1^2\right]\Bigg\},
\end{align*}
and from \eqref{MAto0}  and \eqref{eqn-int-Upsilon-bound-1},
\begin{align*}
\limsup_{N\to\infty}\sup_{t\in [0,T]}\frac{1}{\delta}\P\left(\sup_{0 \le u \le \delta}\Delta^{N,(3)}_{1,1}(t,u)>\ep/3\right)&\le
\frac{18(\lambda^\ast)^4(C_B)^2 C_\beta^2}{\ep^2}\delta\\
&\to0,\quad\text{ as } \quad \delta\to0\,.
\end{align*}
Consequently \eqref{eqn-mfN1-conv11} holds for $ \Delta^{N,(3)}_{1,1}(t,u) $. 

Thus combining the three last results, we obtain
\begin{align} \label{eqn-mfN11-conv}
\limsup_{N\to\infty}\sup_{t\in [0,T]}  \frac{1}{\delta}\P \left( \sup_{0 \le u \le \delta}   \| \Delta^{N}_{1,1}(t+u,x) -  \Delta^{N}_{1,1}(t,x)\|_1   > \ep\right)  \to 0, \qasq \delta \to 0. 
\end{align}
Thus we have shown that $\Delta^{N}_{1,1}(t,\cdot)\to 0$ in $L^1([0,1])$ in probability, locally uniformly in $t$, as $N\to\infty$.

We now consider $\Delta^{N}_{1,2}(t,x)$. To  check condition (i) in Theorem \ref{thm-D-conv-x}, we have for each $t\le T$,  
\begin{align*}
\E\big[\| \Delta^{N}_{1,2}(t,\cdot)\|_1^2\big] &\le  \E\left[ \left( \frac{1}{K^N} \sum_{k=1}^{K^N}\frac{K^N}{N}  \int_0^t   \int_0^\infty \bar{\lambda} (t-s) {\bf 1}_{u \le \Upsilon^N_k(s) } \overline{Q}_{k}(ds, d u) \right)^2     \right]  \\
& \le   \E\left[  \frac{1}{K^N} \sum_{k=1}^{K^N} \Big(\frac{K^N}{N}\Big)^2 \left(\int_0^t   \int_0^\infty \bar{\lambda} (t-s) {\bf 1}_{u \le \Upsilon^N_k(s) } \overline{Q}_{k}(ds, d u) \right)^2     \right]  \\
& =   \E\left[  \frac{1}{K^N} \sum_{k=1}^{K^N} \Big(\frac{K^N}{N}\Big)^2 \int_0^t   \bar{\lambda} (t-s)^2 \Upsilon^N_k(s) ds     \right]  \\
& \le (\lambda^*)^2\frac{K^N}{N}   \E\left[  \frac{1}{K^N} \sum_{k=1}^{K^N}  \int_0^t   \bar\Upsilon^N_k(s) ds     \right]  \\
& \le (\lambda^*)^3 C_B C_\beta T\frac{K^N}{N}    \to 0
\end{align*} 
as $N\to \infty$. 
 To check condition (ii) in Theorem \ref{thm-D-conv-x},  we have  
\begin{align*}
&  \Delta^{N}_{1,2}(t+u,x) -  \Delta^{N}_{1,2}(t,x)
 \\&= \sum_{k=1}^{K^N}  \frac{K^N}{N} \int_0^{t+u}  \int_0^\infty \big( \bar{\lambda} (t+u-s)  -\bar{\lambda} (t-s)\big) {\bf 1}_{r \le \Upsilon^N_k(s) } \overline{Q}_{k}(ds, d r)  \,{\bf 1}_{\mathtt{I}_k}(x)  \\
 & \qquad + \sum_{k=1}^{K^N} \frac{K^N}{N}  \int_t^{t+u} \int_0^\infty   \bar{\lambda} (t-s) {\bf 1}_{r \le \Upsilon^N_k(s) } \overline{Q}_{k}(ds, d r)  \,{\bf 1}_{\mathtt{I}_k}(x)\,. 
\end{align*} 
Thus, 
\begin{align*}
& \| \Delta^{N}_{1,2}(t+u,\cdot) -  \Delta^{N}_{1,2}(t,\cdot)\|_1 \\
& \le \frac{1}{K^N} \sum_{k=1}^{K^N} \bigg| \frac{K^N}{N} \int_0^{t+u}  \int_0^\infty \big( \bar{\lambda} (t+u-s)  -\bar{\lambda} (t-s)\big) {\bf 1}_{r \le \Upsilon^N_k(s) } \overline{Q}_{k}(ds, d r)  \bigg|  \\
 & \qquad + \frac{1}{K^N}\sum_{k=1}^{K^N}  \bigg| \frac{K^N}{N}  \int_t^{t+u} \int_0^\infty   \bar{\lambda} (t-s) {\bf 1}_{r \le \Upsilon^N_k(s) } \overline{Q}_{k}(ds, d r) \bigg|  \\
 & \le  \frac{1}{K^N} \sum_{k=1}^{K^N} \frac{K^N}{N}  \int_0^{t+u}  \int_0^\infty \big| \bar{\lambda} (t+u-s)  -\bar{\lambda} (t-s)\big| {\bf 1}_{r \le \Upsilon^N_k(s) } Q_{k}(ds, d r)   \\
& \quad + \frac{1}{K^N} \sum_{k=1}^{K^N}  \frac{K^N}{N} \int_0^{t+u}  \big| \bar{\lambda} (t+u-s)  -\bar{\lambda} (t-s)\big|  \Upsilon^N_k(s)  ds  \\
& \quad +  \frac{1}{K^N}\sum_{k=1}^{K^N} \frac{K^N}{N} \int_t^{t+u} \int_0^\infty   \bar{\lambda} (t-s) {\bf 1}_{r \le \Upsilon^N_k(s) } Q_{k}(ds, d r)  \\ 
& \quad +  \frac{1}{K^N}\sum_{k=1}^{K^N}  \frac{K^N}{N} \int_t^{t+u}   \bar{\lambda} (t-s)  \Upsilon^N_k(s)  ds\,,
\end{align*}
from which we obtain 
\begin{align*}
& \sup_{0 \le u\le \delta }\| \Delta^{N}_{1,2}(t+u,\cdot) -  \Delta^{N}_{1,2}(t,\cdot)\|_1 \\
 & \le  \frac{1}{K^N} \sum_{k=1}^{K^N} \frac{K^N}{N} \int_0^{t+\delta}  \int_0^\infty  \Big[\varphi(\delta) + \lambda^* \sum_{\ell=1}^\kappa  
\Big(F_\ell(t+\delta-s ) - F_\ell( t-s) \Big)
  \Big]{\bf 1}_{r \le \Upsilon^N_k(s) } Q_{k}(ds, d r)   \\
& \quad + \frac{1}{K^N} \sum_{k=1}^{K^N} \frac{K^N}{N} \int_0^{t+\delta} \Big[\varphi(\delta) + \lambda^* \sum_{\ell=1}^\kappa  
\Big(F_\ell(t+\delta-s ) - F_\ell( t-s) \Big)
  \Big] \Upsilon^N_k(s)  ds  \\
& \quad +  \frac{1}{K^N}\sum_{k=1}^{K^N} \frac{K^N}{N} \int_t^{t+\delta} \int_0^\infty   \bar{\lambda} (t-s) {\bf 1}_{r \le \Upsilon^N_k(s) } Q_{k}(ds, d r)  \\ 
& \quad +  \frac{1}{K^N}\sum_{k=1}^{K^N}  \frac{K^N}{N}  \int_t^{t+\delta}   \bar{\lambda} (t-s)  \Upsilon^N_k(s)  ds\,. 
\end{align*}
For the first term, we have
\begin{align*}
&\E\left[ \left(\frac{1}{K^N} \sum_{k=1}^{K^N}  \frac{K^N}{N}  \int_0^{t+\delta}  \int_0^\infty  \Big[\varphi(\delta) + \lambda^* \sum_{\ell=1}^\kappa  
\Big(F_\ell(t+\delta-s ) - F_\ell( t-s) \Big)
  \Big]{\bf 1}_{r \le \Upsilon^N_k(s) } Q_{k}(ds, d r) \right)^2\right] \\
& \le 2 \E\left[ \left(\frac{1}{K^N} \sum_{k=1}^{K^N}  \frac{K^N}{N} \int_0^{t+\delta}  \int_0^\infty  \Big[\varphi(\delta) + \lambda^* \sum_{\ell=1}^\kappa  
\Big(F_\ell(t+\delta-s ) - F_\ell( t-s) \Big)
  \Big]{\bf 1}_{r \le \Upsilon^N_k(s) } \overline{Q}_{k}(ds, d r) \right)^2\right] \\
  & \quad + 2  \E\left[ \left( \frac{1}{K^N} \sum_{k=1}^{K^N} \frac{K^N}{N} \int_0^{t+\delta} \Big[\varphi(\delta) + \lambda^* \sum_{\ell=1}^\kappa  
\Big(F_\ell(t+\delta-s ) - F_\ell( t-s) \Big)
  \Big] \Upsilon^N_k(s)  ds \right)^2\right] \\
  & \le 2  \E\left[ \frac{1}{K^N} \sum_{k=1}^{K^N} \left( \frac{K^N}{N} \int_0^{t+\delta}  \int_0^\infty  \Big[\varphi(\delta) + \lambda^* \sum_{\ell=1}^\kappa  
\Big(F_\ell(t+\delta-s ) - F_\ell( t-s) \Big)
  \Big]{\bf 1}_{r \le \Upsilon^N_k(s) } \overline{Q}_{k}(ds, d r) \right)^2\right] \\
  & \quad + 2  \E\left[ \frac{1}{K^N} \sum_{k=1}^{K^N} \left(  \frac{K^N}{N}\int_0^{t+\delta} \Big[\varphi(\delta) + \lambda^* \sum_{\ell=1}^\kappa  
\Big(F_\ell(t+\delta-s ) - F_\ell( t-s) \Big)
  \Big] \Upsilon^N_k(s)  ds \right)^2\right] \\
  & \le 2  \E\left[ \frac{1}{K^N} \sum_{k=1}^{K^N} \frac{K^N}{N} \int_0^{t+\delta}    \Big[\varphi(\delta) + \lambda^* \sum_{\ell=1}^\kappa  
\Big(F_\ell(t+\delta-s ) - F_\ell( t-s) \Big)
  \Big]^2  \bar\Upsilon^N_k(s) ds  \right] \\
  & \quad + 2  \E\left[ \frac{1}{K^N} \sum_{k=1}^{K^N} \left(  \int_0^{t+\delta} \Big[\varphi(\delta) + \lambda^* \sum_{\ell=1}^\kappa  
\Big(F_\ell(t+\delta-s ) - F_\ell( t-s) \Big)
  \Big] \bar\Upsilon^N_k(s)  ds \right)^2\right] \\
  & \le 2 \frac{K^N}{N}  \lambda^* C_B C_\beta  \int_0^{t+\delta}    \Big[\varphi(\delta) + \lambda^* \sum_{\ell=1}^\kappa  
\Big(F_\ell(t+\delta-s ) - F_\ell( t-s) \Big)
  \Big]^2   ds  \\
  & \quad + 2 (\lambda^*C_BC_\beta)^2   \left(  \int_0^{t+\delta} \Big[\varphi(\delta) + \lambda^* \sum_{\ell=1}^\kappa  
\Big(F_\ell(t+\delta-s ) - F_\ell( t-s) \Big)
  \Big]  ds \right)^2 \,. 
\end{align*}
Since the integral terms can be made arbitrarily small by choosing $\delta>0$ small enough, we have that
\[\limsup_{N\to\infty}\sup_{t\in [0,T]} \P \left(\sup_{0 \le u\le \delta }  \Delta^{N, (1)}_{1,2}(t,u)  > \ep/4\right)=0\]
for $\delta>0$ small enough.
The second term is already treated above as the second component in the upper bound. The other two terms can be treated in a similar but simpler way. Thus we have shown that 
\begin{align} \label{eqn-mfN12-conv}
\limsup_{N\to\infty}\sup_{t\in [0,T]}  \frac{1}{\delta}\P \left( \sup_{0 \le u \le \delta}   \| \Delta^{N}_{1,2}(t+u,x) -  \Delta^{N}_{1,2}(t,x)\|_1   > \ep\right)  \to 0, \qasq \delta \to 0. 
\end{align}
Thus we have shown that $\Delta^{N}_{1,2}(t,\cdot)\to 0$ in $L^1([0,1])$ in probability, locally uniformly in $t$, as $N\to\infty$. 
The proof for the lemma is complete. 
\end{proof}

We now deduce the following Corollary from the results in Proposition \ref{prop-conv-S-mfF} and Lemmas  \ref{lem-barAn-tight}, \ref{lem-mfN0-conv} and \ref{lem-mfN1-conv}. 

\begin{coro} \label{coro-conv-A} 
Under Assumptions \ref{AS-LLN-1}, \ref{AS-LLN-2} and \ref{AS-lambda}, we have that  $\|\bar\Upsilon^N(t,\cdot) - \bar\Upsilon(t,\cdot)\|_1\to 0$ in probability,  locally uniformly in $t$,  as $N\to \infty$ where $\bar\Upsilon(t,x)$ is given in \eqref{eqn-barUpsilon-tx}, and thus, 
$\|\bar{A}^N(t,\cdot) - \bar{A}(t,\cdot)\|_1\to 0$ in probability, locally uniformly in $t$, as $N\to \infty$, where 
\begin{equation} \label{eqn-barA-tx}
\bar{A}(t,x)= \int_0^t\frac{\bar{S}(s,x)}{\bar{B}(x)}\int_0^1\beta(x,x')\bar{\mathfrak{F}}(s,x')dx'ds = \int_0^t \bar\Upsilon(s,x)ds\,. 
\end{equation}
\end{coro}

\begin{proof}
Combining the results in Lemmas \ref{lem-barAn-tight}, \ref{lem-mfN0-conv} and \ref{lem-mfN1-conv} we have shown that 
$\sup_{0\le t\le T}\Psi^N(t)\to0$ in probability as $N\to\infty$. Thus by Proposition \ref{prop-conv-S-mfF}, we can conclude the convergence of $\bar{S}^N(t,\cdot)$ and $\bar{\mathfrak{F}}^N(t,\cdot)$ in $L^1([0,1])$ in probability, locally uniformly in $t$. 
By the expression of $\bar\Upsilon^N(t,x)$ in \eqref{eqn-barUpsilon-n-sx}, we immediately
obtain the convergence of $\bar\Upsilon^N(t,\cdot)$. 
Then by the expression of $\bar{A}^N(t,x)$ in \eqref{eqn-barAN}, we obtain the convergence in probability of $\bar{A}^N(t,\cdot)$ to $\bar{A}(t,\cdot)$ given in \eqref{eqn-barA-tx}, as announced. The uniformity in $t$ follows from the second Dini theorem.
\end{proof}

\bigskip

\section{Proof for the Convergence of $\bar{\sI}^N(t,\mfa,x)$} \label{sec-proof-conv-I}

In this section, we prove the convergence of  $\bar{\sI}^N(t,\mfa,x)$ to  $\bar{\sI}(t,\mfa,x)$ as stated in Proposition \ref{prop-sIn-conv} below. 
Recall $\sI^N_k(t,\mfa) $ in \eqref{eqn-In-k-rep}. We write the two decomposed processes:
\begin{equation}  \label{eqn-bar-sIn-0}
\bar{\sI}^N_{0}(t,\mfa,x)  =\sum_{k=1}^{K^N}\frac{K^N}{N}  \sum_{j: -j \in \cI^N_k(0)} {\bf1}_{\eta^0_{-j,k} >t} {\bf 1}_{ \tilde{\tau}^N_{-j,k} \le (\mfa-t)^+}   {\bf 1}_{\mathtt{I}_k}(x)  =\sum_{k=1}^{K^N}\frac{K^N}{N}  \sum_{j=1}^{\sI^N_k(0,(\mfa-t)^+)} {\bf1}_{\eta^0_{-j,k} >t}  {\bf 1}_{\mathtt{I}_k}(x) \,\,,
\end{equation}
and
\begin{equation} \label{eqn-bar-sIn-1}
\bar{\sI}^N_{1}(t,\mfa,x)  =\sum_{k=1}^{K^N} \frac{K^N}{N} \sum_{j=A^N_k((t-\mfa)^+)+1}^{A^N_k(t)} {\bf1}_{\tau^N_{j,k} + \eta_{j,k} >t} {\bf 1}_{\mathtt{I}_k}(x)\,\,. 
\end{equation}

\begin{lemma} \label{lem-sIn0-conv}
Under Assumptions \ref{AS-LLN-1} and \ref{AS-lambda},  
\begin{equation} 
\|\bar{\sI}^N_0(t,\mfa, \cdot) - \bar{\sI}_0(t, \mfa, \cdot)\|_1 \to 0
\end{equation}
in probability, locally uniformly in $t$ and $\mfa$, as $N \to \infty$, where
\begin{equation} \label{eqn-bar-sI-0}
 \bar{\sI}_0(t,\mfa, x) := \int_0^{(\mfa-t)^+} \frac{F^c(\mfa'+t)}{F^c(\mfa')} \bar{\sI}(0,d \mfa', x) \,. 
\end{equation}
\end{lemma}

\begin{proof}
We first write 
\[
\bar{\sI}^{N}_{0}(t,\mfa, x) =\bar{\sI}^{N}_{0,1}(t,\mfa,x) + \bar{\sI}^{N}_{0,2}(t,\mfa,x)
\]
where 
\begin{align} 
\bar{\sI}^{N}_{0,1}(t,\mfa,x)  &=\sum_{k=1}^{K^N} \frac{K^N}{N}   \sum_{j=1}^{\sI^N_k(0,(\mfa-t)^+)} \frac{F^c(\tilde{\tau}^N_{-j,k}+t)}{F^c(\tilde{\tau}^N_{-j,k})} {\bf 1}_{\mathtt{I}_k}(x) = \int_0^{(\mfa-t)^+} \frac{F^c(\mfa'+t)}{F^c(\mfa')} \bar{\sI}^N(0,d \mfa', x) \,,\label{eqn-bar-sI-01} \\
\bar{\sI}^{N}_{0,2}(t,\mfa,x)  &=\sum_{k=1}^{K^N} \frac{K^N}{N}   \sum_{j=1}^{\sI^N_k(0,(\mfa-t)^+)} \bigg( {\bf1}_{\eta^0_{-j,k} >t}  -\frac{F^c(\tilde{\tau}^N_{-j,k}+t)}{F^c(\tilde{\tau}^N_{-j,k})}   \bigg){\bf 1}_{\mathtt{I}_k}(x)\,.\label{eqn-bar-sI-02}
\end{align}

We apply Theorem \ref{thm-DD-conv-x}. We first consider the process $\bar{\sI}^{N}_{0,1}(t,\mfa,x)$ and show that 
\begin{equation} \label{eqn-bar-sI-01-conv}
\|\bar{\sI}^N_{0,1}(t,\mfa, \cdot) - \bar{\sI}_{0}(t, \mfa, \cdot)\|_1 \to 0, \quad \text{in probability, locally uniformly in $t$ and $\mfa$,} 
\end{equation}
as $N \to \infty$. 
We first check condition (i) of Theorem \ref{thm-DD-conv-x}.
we have 
\begin{align*}
\bar{\sI}^{N}_{0,1}(t,\mfa,x) - \bar{\sI}_{0}(t,\mfa,x) = 
\int_0^{(\mfa-t)^+}\frac{F^c(\mfa'+t)}{F^c(\mfa')}[\bar{\sI}^N(0,d\mfa',x)-\bar{\sI}(0,d\mfa',x)]\,.
\end{align*}
Condition (i) of Theorem \ref{thm-DD-conv-x} follows from Lemma \ref{convPort} and Assumption \ref{AS-LLN-1}.

Next, we check condition (ii) of Theorem \ref{thm-DD-conv-x} for the processes $\bar{\sI}^{N}_{0,1}(t,
\mfa,x) - \bar{\sI}_{0}(t,\mfa,x)$. We verify the condition for $\bar{\sI}^{N}_{0,1}(t,
\mfa,x) $ in detail below, since the similar calculations can be done for $\bar{\sI}_{0}(t,
\mfa,x) $. Namely, we show that 
 for any $\ep>0$, and for any $T, \bar{\mfa}'>0$, as $\delta\to0$,
\begin{align} 
& \limsup_N  \sup_{t\in [0,T]} \frac{1}{\delta}  \P \bigg(  \sup_{u \in [0,\delta]}\sup_{\mfa \in [0,\infty']} \|\bar{\sI}^{N}_{0,1}(t+u,\mfa,\cdot) - \bar{\sI}^{N}_{0,1}(t,\mfa,\cdot) \|_1 > \ep\bigg) \to 0\,,  \label{eqn-sI01-conv-u}\\
& \limsup_N  \sup_{\mfa\in [0,\infty']} \frac{1}{\delta}  \P \bigg(  \sup_{v \in [0,\delta]}\sup_{t \in [0,T]} \|\bar{\sI}^{N}_{0,1}(t,\mfa+v,\cdot) -  \bar{\sI}^{N}_{0,1}(t,\mfa,\cdot) \|_1 > \ep\bigg) \to 0\,. \label{eqn-sI01-conv-v}
\end{align}

To prove \eqref{eqn-sI01-conv-u}, we have
\begin{align*}
& \bar{\sI}^{N}_{0,1}(t+u,\mfa,x) - \bar{\sI}^{N}_{0,1}(t,\mfa,x) \\
&=  \int_0^{(\mfa-t-u)^+} \frac{F^c(\mfa'+t+u)}{F^c(\mfa')} \bar{\sI}^N(0,d \mfa', x)  -  \int_0^{(\mfa-t)^+} \frac{F^c(\mfa'+t)}{F^c(\mfa')} \bar{\sI}^N(0,d \mfa', x) 
\,,
\end{align*}
and
\begin{align*}
 \big\| \bar{\sI}^{N}_{0,1}(t+u,\mfa,\cdot) - \bar{\sI}^{N}_{0,1}(t,\mfa,\cdot) \big\|_1 
& \le  \int_0^1 \int_0^{(\mfa-t-u)^+} \frac{F^c(\mfa'+t)-F^c(\mfa'+t+u)}{F^c(\mfa')} \bar{\sI}^N(0,d \mfa',x)dx  \\
& \qquad  + \int_0^1 \int_{(\mfa-t-u)^+}^{(\mfa-t)^+} \frac{F^c(\mfa'+t)}{F^c(\mfa')} \bar{\sI}^N(0,d \mfa',x)dx \,.
\end{align*}
Thus,
\begin{align*}
 \sup_{u \in [0,\delta]}\sup_{\mfa \in [0,\infty']}  \big\| \bar{\sI}^{N}_{0,1}(t+u,\mfa,\cdot) - \bar{\sI}^{N}_{0,1}(t,\mfa,\cdot) \big\|_1 
& \le \int_0^1\int_0^{(\infty'-t)^+} \frac{F^c(\mfa'+t)-F^c(\mfa'+t+\delta)}{F^c(\mfa')} \bar{\sI}^N(0,d \mfa',x)dx  \\
& \qquad  + \sup_{\mfa \in [0,\infty']} \int_0^1 \int_{(\mfa-t-\delta)^+}^{(\mfa-t)^+} \frac{F^c(\mfa'+t)}{F^c(\mfa')} \bar{\sI}^N(0,d \mfa',x)dx \,. 
\end{align*}
Thanks to Lemma \ref{convPort} and Assumption \ref{AS-LLN-1}, the first term on the right converges in probability as $N \to \infty$ to 
\[
\int_0^1 \int_0^{(\infty'-t)^+} \frac{F^c(\mfa'+t)-F^c(\mfa'+t+\delta)}{F^c(\mfa')} \bar{\sI}(0,d \mfa',x)dx\,,
\]
which converges to zero as $\delta \to 0$.  It follows from the uniform convergence established in Lemma \ref{convPort} that
the second term on the right converges in probability as $N\to\infty$,  to 
\[
\sup_{\mfa \in [0,\infty']} \int_0^1\int_{(\mfa-t-\delta)^+}^{(\mfa-t)^+} \frac{F^c(\mfa'+t)}{F^c(\mfa')} \bar{\sI}(0,d \mfa',x)dx \le \sup_{\mfa \in [0,\infty']} \int_0^1\int_{(\mfa-t-\delta)^+}^{(\mfa-t)^+} \bar{\sI}(0,d \mfa',x)dx\,.
\]
Under Assumption \ref{AS-LLN-1}, it is clear that the upper bound converges to zero at $\delta \to 0$. Thus we have shown that for $\ep>0$, if $\delta>0$ is small enough,
\[
\limsup_N  \sup_{t\in [0,T]}  \P \bigg(  \sup_{u \in [0,\delta]}\sup_{\mfa \in [0,\infty]} \|\bar{\sI}^{N}_{0,1}(t+u,\mfa,\cdot) - \bar{\sI}^{N}_{0,1}(t,\mfa,\cdot) \|_1 > \ep\bigg) = 0\,.
\]
To prove \eqref{eqn-sI01-conv-v}, we have
\begin{align*}
 \bar{\sI}^{N}_{0,1}(t,\mfa+v,x) - \bar{\sI}^{N}_{0,1}(t,\mfa,x) =  \int_0^1 \int_{(\mfa-t)^+}^{(\mfa+v-t)^+} \frac{F^c(\mfa'+t)}{F^c(\mfa')} \bar{\sI}^N(0,d \mfa', x)dx\,,
\end{align*} 
and
\begin{align*}
  \sup_{v \in [0,\delta]}\sup_{t \in [0,T]} \big\|\bar{\sI}^{N}_{0,1}(t,\mfa+v,\cdot) - \bar{\sI}^{N}_{0,1}(t,\mfa,\cdot)\|_1 \le   \sup_{t \in [0,T]}\int_0^1 \int_{(\mfa-t)^+}^{(\mfa+\delta-t)^+} \frac{F^c(\mfa'+t)}{F^c(\mfa')} \bar{\sI}^N(0,d \mfa', x)dx\,.
\end{align*} 
In order to show that the $\sup_t$ on the above right hand side converges in probability, as $N\to\infty$, to
\begin{align}\label{limsup}
  \sup_{t \in [0,T]} \int_0^1 \int_{(\mfa-t)^+}^{(\mfa+\delta-t)^+} \frac{F^c(\mfa'+t)}{F^c(\mfa')} \bar{\sI}(0,d \mfa', x)dx \le 
    \sup_{t \in [0,T]} \int_0^1 \int_{(\mfa-t)^+}^{(\mfa+\delta-t)^+}  \bar{\sI}(0,d \mfa', x)dx\,,
\end{align}
it suffices to show that the convergence of $\int_0^1 \int_{(\mfa-t)^+}^{(\mfa+\delta-t)^+} \frac{F^c(\mfa'+t)}{F^c(\mfa')} \bar{\sI}^N(0,d \mfa', x)dx$ is uniform in $t$. Indeed, we note that
\begin{align*}
\int_0^1 \int_{(\mfa-t)^+}^{(\mfa+\delta-t)^+} &\frac{F^c(\mfa'+t)}{F^c(\mfa')} \bar{\sI}^N(0,d \mfa', x)dx\\
&=\int_0^1 \int_0^{(\mfa+\delta-t)^+} \frac{F^c(\mfa'+t)}{F^c(\mfa')} \bar{\sI}^N(0,d \mfa', x)dx -
\int_0^{(\mfa-t)^+}\frac{F^c(\mfa'+t)}{F^c(\mfa')} \bar{\sI}^N(0,d \mfa', x)dx\,.
\end{align*}
This right hand side is the difference of two non--increasing functions of $t$ which converge pointwise to their limit in probability, as $N\to\infty$, and both limits are continuous in $t$. Hence the uniform convergence follows from the second Dini theorem, exactly as in the proof of Lemma \ref{convPort}. Going back to \eqref{limsup}, we note that,
 under Assumption \ref{AS-LLN-1}, the right hand side converges to zero at $\delta \to 0$. 
Thus we have shown that for $\ep>0$, if $\delta>0$ is small enough,
\[
\limsup_N  \sup_{\mfa\in [0,\infty]}  \P \bigg(    \sup_{v \in [0,\delta]}\sup_{t \in [0,T]}  \|\bar{\sI}^{N}_{0,1}(t,\mfa+v,\cdot) - \bar{\sI}^{N}_{0,1}(t,\mfa,\cdot) \|_1 > \ep\bigg) = 0\,.
\]
Thus we have verified condition (ii) of Theorem \ref{thm-DD-conv-x} for the processes $\bar{\sI}^{N}_{0,1}(t,
\mfa,x)$, and with a similar argument for $\bar{\sI}_{0}(t,\mfa,x)$, and thus, for the difference $ \bar{\sI}^{N}_{0,1}(t,\mfa,x)- \bar{\sI}_{0}(t,\mfa,x)$. Therefore, the claim on the convergence of $\bar{\sI}^{N}_{0,1}(t,\mfa,x)$  in \eqref{eqn-bar-sI-01-conv} is proved. 

We next prove the convergence of $\bar{\sI}^{N}_{0,2}(t,\mfa,x)$: 
\begin{equation} \label{eqn-bar-sI-02-conv}
\|\bar{\sI}^N_{0,2}(t,\mfa, \cdot)\|_1 \to 0, \quad \text{in probability, locally uniformly in $t$ and $\mfa$, as $N \to \infty$.} 
\end{equation}
To check condition (i) of Theorem \ref{thm-DD-conv-x}, we have 
\begin{align*}
\|\bar{\sI}^{N}_{0,2}(t,\mfa,\cdot)\|_1 \le \frac{1}{K^N}\sum_{k=1}^{K^N} \Bigg|\frac{K^N}{N}   \sum_{j=1}^{\sI^N_k(0,(\mfa-t)^+)} \bigg( {\bf1}_{\eta^0_{-j,k} >t}  -\frac{F^c(\tilde{\tau}^N_{-j,k}+t)}{F^c(\tilde{\tau}^N_{-j,k})}   \bigg) \Bigg| \,.
\end{align*}
We deduce from Jensen's inequality that 
\begin{align} \label{eqn-bar-sI-02-conv-p1}
& \E\Bigg[ \Bigg(  \frac{1}{K^N}\sum_{k=1}^{K^N}\frac{K^N}{N} \Bigg|  \sum_{j=1}^{\sI^N_k(0,(\mfa-t)^+)} \bigg( {\bf1}_{\eta^0_{-j,k} >t}  -\frac{F^c(\tilde{\tau}^N_{-j,k}+t)}{F^c(\tilde{\tau}^N_{-j,k})}   \bigg)\Bigg| \Bigg)^2 \Bigg] \non \\
& \le   \frac{1}{K^N}\sum_{k=1}^{K^N}   \frac{K^N}{N}  \E \Bigg[ \int_0^{(\mfa-t)^+} \frac{F^c(\mfa'+t)}{F^c(\mfa')}  \Big( 1-  \frac{F^c(\mfa'+t)}{F^c(\mfa')} \Big) \bar{\sI}^N_k(0, d\mfa')  \  \Bigg] \,,
\end{align} 
where we have used the fact that the $\eta^0_{-j,k}$'s are conditionally independent, given the $\tilde{\tau}^N_{-j,k}$'s.
Note that under Assumption \ref{AS-LLN-1}, thanks to Lemma \ref{convPort}, as $N\to\infty$, in probability,
\begin{align*}
\frac{1}{K^N}\sum_{k=1}^{K^N}    &\int_0^{(\mfa-t)^+} \frac{F^c(\mfa'+t)}{F^c(\mfa')}  \Big( 1-  \frac{F^c(\mfa'+t)}{F^c(\mfa')} \Big) \bar{\sI}^N_k(0, d\mfa')\\
 &=\int_0^1  \int_0^{(\mfa-t)^+} \frac{F^c(\mfa'+t)}{F^c(\mfa')}  \Big( 1-  \frac{F^c(\mfa'+t)}{F^c(\mfa')} \Big) \bar{\sI}^N(0, d\mfa',x)dx\\
& \to \int_0^1  \int_0^{(\mfa-t)^+} \frac{F^c(\mfa'+t)}{F^c(\mfa')}  \Big( 1-  \frac{F^c(\mfa'+t)}{F^c(\mfa')} \Big) \bar{\sI}(0, d\mfa',x)dx \,. 
\end{align*}  
Thus, the upper bound in \eqref{eqn-bar-sI-02-conv-p1} 
converges to zero as $N\to \infty$. 
This implies that for any $\ep>0$, 
\[
\sup_{t \in [0,T]}\sup_{\mfa\in [0,\infty]} \P \big( \|\bar{\sI}^{N}_{0,2}(t,\mfa,\cdot)\|_1> \ep \big) \to 0 \qasq N \to \infty. 
\]
Next, to check condition (ii) of Theorem \ref{thm-DD-conv-x}, we show that 
 for any $\ep>0$, and for any $T, \bar{\mfa}'>0$, as $\delta\to0$,
\begin{align} 
& \limsup_N  \sup_{t\in [0,T]} \frac{1}{\delta}  \P \bigg(  \sup_{u \in [0,\delta]}\sup_{\mfa \in [0,\infty']}\big \|\bar{\sI}^{N}_{0,2}(t+u,\mfa,\cdot) - \bar{\sI}^{N}_{0,2}(t,\mfa,\cdot) \big\|_1 > \ep\bigg) \to 0\,,  \label{eqn-sI02-conv-u}\\
& \limsup_N  \sup_{\mfa\in [0,\infty]} \frac{1}{\delta}  \P \bigg(  \sup_{v \in [0,\delta]}\sup_{t \in [0,T]} \big\|\bar{\sI}^{N}_{0,2}(t,\mfa+v,\cdot) -  \bar{\sI}^{N}_{0,2}(t,\mfa,\cdot) \big\|_1 > \ep\bigg) \to 0\,. \label{eqn-sI02-conv-v}
\end{align}
To prove \eqref{eqn-sI02-conv-u}, we have
\begin{align*}
&\bar{\sI}^{N}_{0,2}(t+u,\mfa,x) - \bar{\sI}^{N}_{0,2}(t,\mfa,x) \\
& = \sum_{k=1}^{K^N}\frac{K^N}{N}   \sum_{j=1}^{\sI^N_k(0,(\mfa-t-u)^+)} \bigg( {\bf1}_{t <\eta^0_{-j,k} \le t+u}  -\frac{F^c(\tilde{\tau}^N_{-j,k}+t)-F^c(\tilde{\tau}^N_{-j,k}+t+u)}{F^c(\tilde{\tau}^N_{-j,k})}   \bigg){\bf 1}_{\mathtt{I}_k}(x)\\
& \quad - \sum_{k=1}^{K^N}\frac{K^N}{N}  \sum_{j=\sI^N_k(0,(\mfa-t-u)^+)+1}^{\sI^N_k(0,(\mfa-t)^+)} \bigg( {\bf1}_{\eta^0_{-j,k} >t} -\frac{F^c(\tilde{\tau}^N_{-j,k}+t)}{F^c(\tilde{\tau}^N_{-j,k})}   \bigg){\bf 1}_{\mathtt{I}_k}(x)\,,
\end{align*}
and
\begin{align} \label{eqn-sI02-conv-u-p1}
&\big\|\bar{\sI}^{N}_{0,2}(t+u,\mfa,\cdot) - \bar{\sI}^{N}_{0,2}(t,\mfa,\cdot) \big\|_1  \non \\
& \le \frac{1}{K^N}\sum_{k=1}^{K^N} \Bigg|\frac{K^N}{N}   \sum_{j=1}^{\sI^N_k(0,(\mfa-t-u)^+)} \bigg( {\bf1}_{t <\eta^0_{-j,k} \le t+u}  -\frac{F^c(\tilde{\tau}^N_{-j,k}+t)-F^c(\tilde{\tau}^N_{-j,k}+t+u)}{F^c(\tilde{\tau}^N_{-j,k})}   \bigg) \Bigg| \non \\
& \quad + \frac{1}{K^N} \sum_{k=1}^{K^N} \Bigg|\frac{K^N}{N}   \sum_{j=\sI^N_k(0,(\mfa-t-u)^+)+1}^{\sI^N_k(0,(\mfa-t)^+)} \bigg( {\bf1}_{\eta^0_{-j,k} >t} -\frac{F^c(\tilde{\tau}^N_{-j,k}+t)}{F^c(\tilde{\tau}^N_{-j,k})}   \bigg) \Bigg| \non \\  
& \le \frac{1}{K^N}\sum_{k=1}^{K^N}  \frac{K^N}{N}   \sum_{j=1}^{\sI^N_k(0,(\mfa-t-u)^+)} {\bf1}_{t <\eta^0_{-j,k} \le t+u}  \non \\
& \quad + \frac{1}{K^N}\sum_{k=1}^{K^N} \frac{K^N}{N}   \sum_{j=1}^{\sI^N_k(0,(\mfa-t-u)^+)}  \frac{F^c(\tilde{\tau}^N_{-j,k}+t)-F^c(\tilde{\tau}^N_{-j,k}+t+u)}{F^c(\tilde{\tau}^N_{-j,k})} \non \\
& \quad +  \frac{1}{K^N} \sum_{k=1}^{K^N} \Big(\bar{\sI}^N_k(0,(\mfa-t)^+) - \bar{\sI}^N_k(0,(\mfa-t-u)^+) \Big) \,. 
\end{align} 
For the first term on the right, we have
\begin{align} \label{eqn-sI02-conv-u-p2}
& \P\Bigg(  \sup_{u \in [0,\delta]}\sup_{\mfa \in [0,\infty']}    \frac{1}{K^N} \sum_{k=1}^{K^N}\frac{K^N}{N}   \sum_{j=1}^{\sI^N_k(0,(\mfa-t-u)^+)} {\bf1}_{t <\eta^0_{-j,k} \le t+u}  > \ep \Bigg) \non\\
& \le  \P\Bigg(   \frac{1}{K^N} \sum_{k=1}^{K^N}\frac{K^N}{N} \sum_{j=1}^{\sI^N_k(0,(\infty'-t)^+)} {\bf1}_{t <\eta^0_{-j,k} \le t+\delta}  > \ep \Bigg) \non \\
& \le \P\Bigg(    \frac{1}{K^N} \sum_{k=1}^{K^N} \frac{K^N}{N}   \sum_{j=1}^{\sI^N_k(0,(\infty-t)^+)}  \bigg({\bf1}_{t <\eta^0_{-j,k} \le t+\delta}-\frac{F^c(\tilde{\tau}^N_{-j,k}+t)-F^c(\tilde{\tau}^N_{-j,k}+t+\delta)}{F^c(\tilde{\tau}^N_{-j,k})}   \bigg) > \ep/2 \Bigg) \non \\
& \quad + \P\Bigg(   \frac{1}{K^N} \sum_{k=1}^{K^N}\frac{K^N}{N}   \sum_{j=1}^{\sI^N_k(0,(\infty-t)^+)} \frac{F^c(\tilde{\tau}^N_{-j,k}+t)-F^c(\tilde{\tau}^N_{-j,k}+t+\delta)}{F^c(\tilde{\tau}^N_{-j,k})} >  \ep/2 \Bigg)\,.
\end{align}
Here using Jensen's inequality and the fact that the summands over $j$ are independent, conditionally upon the 
$\tilde{\tau}^N_{-j,k}$'s, the first probability is bounded by
\begin{align} \label{eqn-sI02-conv-u-p3}
& \frac{4}{\ep^2} \E \Bigg[ \bigg( \frac{1}{K^N} \sum_{k=1}^{K^N}\frac{K^N}{N}    \sum_{j=1}^{\sI^N_k(0,(\infty-t)^+)}  \bigg({\bf1}_{t <\eta^0_{-j,k} \le t+\delta}-\frac{F^c(\tilde{\tau}^N_{-j,k}+t)-F^c(\tilde{\tau}^N_{-j,k}+t+\delta)}{F^c(\tilde{\tau}^N_{-j,k})}   \bigg) \bigg)^2 \Bigg] \non \\
& \le \frac{K^N}{N}    \frac{4}{\ep^2}  \E \int_0^1  \int_0^{(\infty-t)^+}
\frac{F^c(\mfa'+t)-F^c(\mfa'+t+\delta)}  {F^c(\mfa')} \bar{\sI}^N(0, d \mfa', x) dx.
\end{align}
Now under Assumption \ref{AS-LLN-1}, it follows from Lemma \ref{convPort} that
\begin{align*}
&\int_0^1 \int_0^{(\infty-t)^+}
\frac{F^c(\mfa'+t)-F^c(\mfa'+t+\delta)}  {F^c(\mfa')} \bar{\sI}^N(0, d \mfa',x)dx \\
& \to  \int_0^1 \int_0^{(\infty-t)^+}
\frac{F^c(\mfa'+t)-F^c(\mfa'+t+\delta)}  {F^c(\mfa')} \bar{\sI}(0, d \mfa',x)dx
\end{align*} 
in probability as $N \to \infty$. Hence the upper bound in \eqref{eqn-sI02-conv-u-p3} converges to zero, as $N\to\infty$. Inside the second probability in \eqref{eqn-sI02-conv-u-p2}, we have 
\begin{align*}
&  \frac{1}{K^N} \sum_{k=1}^{K^N} \frac{K^N}{N}    \sum_{j=1}^{\sI^N_k(0,(\infty-t)^+)} \frac{F^c(\tilde{\tau}^N_{-j,k}+t)-F^c(\tilde{\tau}^N_{-j,k}+t+\delta)}{F^c(\tilde{\tau}^N_{-j,k})} \\
&= \int_0^1 \int_0^{(\infty-t)^+}   \frac{F^c(\mfa'+t)-F^c(\mfa'+t+\delta)}{F^c(\mfa')}\bar{\sI}^N(0, d \mfa',x)dx   \\
& \to \int_0^1 \int_0^{(\infty-t)^+}   \frac{F^c(\mfa'+t)-F^c(\mfa'+t+\delta)}{F^c(\mfa')}\bar{\sI}(0, d \mfa',x)dx 
\end{align*}
in probability as $N \to \infty$, again from Lemma \ref{convPort}, and the limit converges to zero as $\delta\to 0$. Hence for any $\ep>0$, if $\delta>0$ is small enough,
$\limsup_N$ of the second term in the right hand side of \eqref{eqn-sI02-conv-u-p2} is zero.

For the second term on the right of \eqref{eqn-sI02-conv-u-p1}, we have
\begin{align*}
&\sup_{u \in [0,\delta]}\sup_{\mfa \in [0,\infty']}  \frac{1}{K^N}\sum_{k=1}^{K^N}   \frac{K^N}{N}   \sum_{j=1}^{\sI^N_k(0,(\mfa-t-u)^+)}  \frac{F^c(\tilde{\tau}^N_{-j,k}+t)-F^c(\tilde{\tau}^N_{-j,k}+t+u)}  {F^c(\tilde{\tau}^N_{-j,k})} \\
&  \le  \int_0^1  \int_0^{(\infty'-t)^+}
\frac{F^c(\mfa'+t)-F^c(\mfa'+t+\delta)}  {F^c(\mfa')} \bar{\sI}^N(0, d \mfa',x)dx
\end{align*}
which, thanks to Lemma \ref{convPort} and Assumption \ref{AS-LLN-1}, converges in probability as $N\to \infty$, to 
\begin{align*}
& \int_0^1 \int_0^{(\infty-t)^+}
\frac{F^c(\mfa'+t)-F^c(\mfa'+t+\delta)}  {F^c(\mfa')} \bar{\sI}(0, d \mfa', x)dx.
\end{align*}
This expression will also converge to zero as $\delta \to 0$. 
For the third term on the right of \eqref{eqn-sI02-conv-u-p1},
we have 
\begin{align*}
 & \sup_{u \in [0,\delta]} \int_0^1 \Big(\bar{\sI}^N(0,(\mfa-t)^+,x) - \bar{\sI}^N(0,(\mfa-t-u)^+,x) \Big)dx \\
 & \le  \int_0^1 \Big(\bar{\sI}^N(0,(\mfa-t)^+,x) - \bar{\sI}^N(0,(\mfa-t-\delta)^+,x) \Big) dx
\end{align*}
which converges in probability  to 
\[
\int_0^1  \Big(\bar{\sI}(0,(\mfa-t)^+,x) - \bar{\sI}(0,(\mfa-t-\delta)^+,x) \Big) dx
\]
as $N\to\infty$. Since $\bar{\sI}^N(0,\cdot,x)$ and $\bar{\sI}(0,\cdot,x)$ are nondecreasing and the limit is continuous, the convergence also holds uniformly over $\mfa\in [0,\infty']$. 
Moreover, we also have that
\[
 \sup_{\mfa \in [0,\infty]} \int_0^1  \Big(\bar{\sI}(0,(\mfa-t)^+,x) - \bar{\sI}(0,(\mfa-t-\delta)^+,x) \Big) dx\to0,
\]
as $\delta \to 0$. 
Combining the results on the three terms on the right of \eqref{eqn-sI02-conv-u-p1}, we have shown that \eqref{eqn-sI02-conv-u} holds.


We next prove \eqref{eqn-sI02-conv-v}. We have 
\begin{align*}
\bar{\sI}^{N}_{0,2}(t,\mfa+v,x) - \bar{\sI}^{N}_{0,2}(t,\mfa,x) =  \sum_{k=1}^{K^N} \frac{K^N}{N}  \sum_{j=\sI^N_k(0,(\mfa-t)^++1}^{\sI^N_k(0,(\mfa+v-t)^+)} \bigg( {\bf1}_{\eta^0_{-j,k} >t+u}  -\frac{F^c(\tilde{\tau}^N_{-j,k}+t)}{F^c(\tilde{\tau}^N_{-j,k})}   \bigg){\bf 1}_{\mathtt{I}_k}(x)\,,
\end{align*}
and
\begin{align*}
\big\|\bar{\sI}^{N}_{0,2}(t,\mfa+v,\cdot) - \bar{\sI}^{N}_{0,2}(t,\mfa,\cdot)\|_1 
&\le \frac{1}{K^N}  \sum_{k=1}^{K^N} \Bigg|\frac{K^N}{N}   \sum_{j=\sI^N_k(0,(\mfa-t)^++1}^{\sI^N_k(0,(\mfa+v-t)^+)} \bigg( {\bf1}_{\eta^0_{-j,k} >t+u}  -\frac{F^c(\tilde{\tau}^N_{-j,k}+t)}{F^c(\tilde{\tau}^N_{-j,k})}   \bigg) \Bigg| \\
& \le \frac{1}{K^N}  \sum_{k=1}^{K^N}  \big|\bar{\sI}^N_k(0,(\mfa+v-t)^+) - \bar{\sI}^N_k(0,(\mfa-t)^+ \big| \,. 
\end{align*}
Thus, 
\begin{align*}
& \sup_{v \in [0,\delta]}\sup_{t \in [0,T]} \big\|\bar{\sI}^{N}_{0,2}(t,\mfa+v,\cdot) - \bar{\sI}^{N}_{0,2}(t,\mfa,\cdot)\|_1 \\
  &\le \sup_{t \in [0,T]} \frac{1}{K^N}  \sum_{k=1}^{K^N}  \Big(\bar{\sI}^N_k(0,(\mfa+\delta-t)^+) - \bar{\sI}^N_k(0,(\mfa-t)^+) \Big) \\
&=\sup_{t \in [0,T]}\int_0^1\Big(\bar{\sI}^N(0,(\mfa+\delta-t)^+,x) - \bar{\sI}^N(0,(\mfa-t)^+,x) \Big) dx
\end{align*} 
and we claim that the right hand side converges in probability as $N\to\infty$, to 
\[
 \sup_{t \in [0,T]} \int_0^1  \Big( \bar{\sI}(0,(\mfa+\delta-t)^+,x) - \bar{\sI}(0,(\mfa-t)^+ ,x)\Big)dx \,.
\]
Indeed, the convergence without the $\sup_t$ follows from Assumption \ref{AS-LLN-1}, and both
$t\mapsto \int_0^1\bar{\sI}^N(0,(\mfa+\delta-t)^+,x)dx$ and $t\mapsto \int_0^1\bar{\sI}^N(0,(\mfa-t)^+,x)dx$
are non--increasing, while the limits are continuous. Hence again an application of the second Dini theorem 
implies that the convergence is locally uniform in $t$, hence the claim.
The limit then converges to zero as $\delta \to 0$.  Thus we have shown  \eqref{eqn-sI02-conv-v}. This completes the proof of the lemma. 
\end{proof}

\begin{lemma}  \label{lem-sIn1-conv}
Under Assumptions \ref{AS-LLN-1}, \ref{AS-LLN-2} and \ref{AS-lambda},  
\begin{equation} 
\|\bar{\sI}^N_1(t,\mfa,\cdot) - \bar{\sI}_1(t,\mfa,\cdot)\|_1 \to 0
\end{equation}
in probability, locally uniformly in $t$ and $\mfa$, as $N \to \infty$, where
\begin{equation}\label{eqn-bar-sI-1}
 \bar{\sI}_1(t,\mfa,x) :=  \int_{(t-\mfa)^+}^t F^c(t-s)  \bar{A}(ds,x) \,, 
\end{equation}
where $\bar{A}(t,x)$ is given in \eqref{eqn-barA-tx}. 
\end{lemma}

\begin{proof}
We first write 
\[
\bar{\sI}^{N}_{1}(t,\mfa,x) = \bar{\sI}^{N}_{1,1}(t,\mfa,x) + \bar{\sI}^{N}_{1,2}(t,\mfa,x)
\]
where 
\begin{align}
\bar{\sI}^{N}_{1,1}(t,\mfa,x) &=\sum_{k=1}^{K^N} \frac{K^N}{N} \sum_{j=A^N_k((t-\mfa)^+)+1}^{A^N_k(t)}  F^c(t-\tau^N_{j,k} ){\bf 1}_{\mathtt{I}_k}(x)  =\int _{(t-\mfa)^+}^t F^c(t-s)  \bar{A}^N(ds,x)  \,\,,\label{eqn-bar-sI-11}\\
\bar{\sI}^{N}_{1,2}(t,\mfa,x) & =\sum_{k=1}^{K^N} \frac{K^N}{N}  \sum_{j=A^N_k((t-\mfa)^+)+1}^{A^N_k(t)} \Big( {\bf1}_{\tau^N_{j,k} + \eta_{j,k} >t} - F^c(t-\tau^N_{j,k} )\Big){\bf 1}_{\mathtt{I}_k}(x)\,\,. \label{eqn-bar-sI-12}
\end{align}
We apply Theorem \ref{thm-DD-conv-x}. 
We start with the process $\bar{\sI}^{N}_{1,1}(t,\mfa,x)$ and show that 
\begin{equation} \label{eqn-bar-sI-11-conv}
\big\|\bar{\sI}^N_{1,1}(t,\mfa, \cdot) - \bar{\sI}_{1}(t, \mfa, \cdot)\big\|_1 \to 0, \quad \text{in probability, locally uniformly in $t$ and $\mfa$,} 
\end{equation}
as $N \to \infty$. 
Since
\begin{align*}
\bar{\sI}^N_{1,1}(t,\mfa, x) - \bar{\sI}_{1}(t, \mfa, x)
&= \int _{(t-\mfa)^+}^t F^c(t-s)  \Big(\bar{A}^N(ds,x) - \bar{A}(ds,x)    \Big) ,
\end{align*} 
 condition (i) of Theorem \ref{thm-DD-conv-x} follows from Lemma \ref{convPort} and Corollary \ref{coro-conv-A}.
In other words, we have that for each $t$ and $\mfa$, and for any $\epsilon>0$, 
\[\P(\|\bar{\sI}^{N}_{1,1}(t,\mfa,\cdot) - \bar{\sI}_{1}(t,\mfa,\cdot) \|_1>\epsilon) \to 0 \qasq N \to \infty.
\]

We next want to check (ii) of Theorem \ref{thm-DD-conv-x} for the processes $\bar{\sI}^{N}_{1,1}(t,
\mfa,x) - \bar{\sI}_{1}(t,\mfa,x)$. 
We will verify the following conditions for $\bar{\sI}^{N}_{1,1}(t,
\mfa,x) $:  
 for any $\ep>0$, and for any $T, \bar{\mfa}'>0$, as $\delta\to0$,
\begin{align} 
& \limsup_N  \sup_{t\in [0,T]} \frac{1}{\delta}  \P \bigg(  \sup_{u \in [0,\delta]}\sup_{\mfa \in [0,\infty']} \|\bar{\sI}^{N}_{1,1}(t+u,\mfa,\cdot) - \bar{\sI}^{N}_{1,1}(t,\mfa,\cdot) \|_1 > \ep\bigg) \to 0\,,  \label{eqn-sI11-conv-u}\\
& \limsup_N  \sup_{\mfa\in [0,\infty']} \frac{1}{\delta}  \P \bigg(  \sup_{v \in [0,\delta]}\sup_{t \in [0,T]} \|\bar{\sI}^{N}_{1,1}(t,\mfa+v,\cdot) -  \bar{\sI}^{N}_{1,1}(t,\mfa,\cdot) \|_1 > \ep\bigg) \to 0\,. \label{eqn-sI11-conv-v}
\end{align}
It will be clear that the same results hold (and are simpler to prove) for $\bar{\sI}_{1}(t,\mfa,\cdot)$.
To prove \eqref{eqn-sI11-conv-u}, we have
\begin{align*}
& \bar{\sI}^{N}_{1,1}(t+u,\mfa,x) - \bar{\sI}^{N}_{1,1}(t,\mfa,x) \\
&=  \int _{(t+u-\mfa)^+}^{t+u} F^c(t+u-s)  \bar{A}^N(ds,x)  -  \int _{(t-\mfa)^+}^t F^c(t-s) \bar{A}^N(ds,x)  \\
&=  \int _{(t-\mfa)^+}^{t+u} \Big(F^c(t+u-s) - F^c(t-s) \Big) \bar{A}^N(ds,x)   \\
& \quad -    \int _{(t-\mfa)^+}^{(t+u-\mfa)^+} F^c(t+u-s)  \bar{A}^N(ds,x) 
+ \int _{t}^{t+u} F^c(t-s)  \bar{A}^N(ds,x) \,,
\end{align*} 
and
\begin{align}\label{eqn-sI11-conv-u-p1}
& \big\|\bar{\sI}^{N}_{1,1}(t+u,\mfa,\cdot) - \bar{\sI}^{N}_{1,1}(t,\mfa,\cdot) \big\|_1 \non\\
&\le \int_0^1 \int _{(t-\mfa)^+}^{t+u} \Big(F^c(t-s) - F^c(t+u-s) \Big)  \bar{A}^N(ds,x)dx   \non\\
& \quad +    \int _{(t-\mfa)^+}^{(t+u-\mfa)^+} F^c(t+u-s) \bar{A}^N(ds,x)dx   
+  \int _{t}^{t+u} F^c(t-s) \bar{A}^N(ds,x)dx \,.
\end{align} 
Here the first term on the right satisfies
\begin{align*}
\sup_{u \in [0,\delta]}\sup_{\mfa \in [0,\infty']} & \int_0^1\int _{(t-\mfa)^+}^{t+u} \Big(F^c(t-s) - F^c(t+u-s) \Big)  \bar{A}^N(ds,x)dx   \\
& \le \int_0^1\int _{(t-\infty')^+}^{t+\delta} \Big(F^c(t-s) - F^c(t+\delta-s) \Big) \bar{A}^N(ds,x)dx   \\
& \to \int_0^1   \int _{(t-\infty')^+}^{t+\delta} \Big(F^c(t-s) - F^c(t+\delta-s) \Big)  \bar{A}(ds,x)dx  
\end{align*}
in probability as $N\to\infty$ by Lemma  \ref{lem-barAn-tight} and Corollary \ref{coro-conv-A}, and the limit converges to zero as $\delta \to 0$. The second term on the right side of \eqref{eqn-sI11-conv-u-p1} satisfies 
\begin{align*}
&\sup_{u \in [0,\delta]}\sup_{\mfa \in [0,\infty']}  \int_0^1 \int _{(t-\mfa)^+}^{(t+u-\mfa)^+} F^c(t+u-s)  \bar{A}^N(ds,x)dx  \\
& \le \sup_{\mfa \in [0,\infty']}   \int_0^1 \Big( \bar{A}^N((t+\delta-\mfa)^+,x)  -  \bar{A}^N((t-\mfa)^+,x)  \Big)dx \\
& \to \sup_{\mfa \in [0,\infty']}  \int_0^1  \Big( \bar{A}((t+\delta-\mfa)^+,x)  -  \bar{A}((t-\mfa)^+,x)  \Big)dx 
\end{align*} 
in probability as $N\to\infty$ by Corollary  \ref{coro-conv-A} and the second Dini theorem, and the limit converges to zero as $\delta \to 0$. 
 The third term on the right side of \eqref{eqn-sI11-conv-u-p1} does not depend on $\mfa$ and satisfies 
\begin{align*}
&\sup_{u \in [0,\delta]} \int_0^1  \int _{t}^{t+u} F^c(t-s) \bar{A}^N(ds,x)dx \\
& \le   \int_0^1\Big( \bar{A}^N(t+\delta,x)  -  \bar{A}^N(t,x)  \Big)dx \to   \int_0^1  \Big( \bar{A}(t+\delta,x)  -  \bar{A}(t,x)  \Big)dx 
\end{align*} 
in probability as $N\to\infty$ by Corollary  \ref{coro-conv-A}, and the limit converges to zero as $\delta \to 0$. 
Thus we have shown that for small enough $\delta>0$, 
 for any $\ep>0$, and for any $T, \bar{\mfa}'>0$, 
\[
 \limsup_N  \sup_{t\in [0,T]}  \P \bigg(  \sup_{u \in [0,\delta]}\sup_{\mfa \in [0,\infty']} \|\bar{\sI}^{N}_{1,1}(t+u,\mfa,\cdot) - \bar{\sI}^{N}_{1,1}(t,\mfa,\cdot) \|_1 > \ep\bigg) = 0\,.
\]

To prove \eqref{eqn-sI11-conv-v}, we have
\begin{align*}
 \bar{\sI}^{N}_{1,1}(t,\mfa+v,x) - \bar{\sI}^{N}_{1,1}(t,\mfa,x)  =   \int _{(t-\mfa-v)^+}^{(t-\mfa)^+} F^c(t-s)  \bar{A}^N(ds,x)\,,
\end{align*} 
and
\begin{align*}
\big\| \bar{\sI}^{N}_{1,1}(t,\mfa+v,\cdot) - \bar{\sI}^{N}_{1,1}(t,\mfa,\cdot) \big\|_1  = \int_0^1  \int _{(t-\mfa-v)^+}^{(t-\mfa)^+} F^c(t-s)  \bar{A}^N(ds,x)dx \,.
\end{align*} 
Hence,
\begin{align*}
& \sup_{v \in [0,\delta]}\sup_{t \in [0,T]}  \big\| \bar{\sI}^{N}_{1,1}(t,\mfa+v,\cdot) - \bar{\sI}^{N}_{1,1}(t,\mfa,\cdot) \big\|_1  \\
& \le \sup_{t \in [0,T]}   \int_0^1 \Big( \bar{A}^N((t-\mfa)^+,x)  -  \bar{A}^N((t-\mfa-\delta)^+,x)  \Big) dx\\
& \to \sup_{t \in [0,T]} \int_0^1 \Big( \bar{A}((t-\mfa)^+,x)  -  \bar{A}((t-\mfa-\delta)^+,x)  \Big) dx\, 
\end{align*} 
in probability as $N\to\infty$ by Corollary  \ref{coro-conv-A} and again the second Dini theorem. Moreover, the limit converges to zero as $\delta \to 0$. 
Thus we have shown that for small enough $\delta>0$, 
 for any $\ep>0$, and for any $T, \bar{\mfa}'>0$, 
\[
\limsup_N  \sup_{\mfa\in [0,\infty']} \P \bigg(  \sup_{v \in [0,\delta]}\sup_{t \in [0,T]} \|\bar{\sI}^{N}_{1,1}(t,\mfa+v,\cdot) -  \bar{\sI}^{N}_{1,1}(t,\mfa,\cdot) \|_1 > \ep\bigg) = 0\,. 
\]
Therefore, combining the above, we have proved the convergence of  $\bar{\sI}^{N}_{1,1}(t,\mfa,x)$ as stated in \eqref{eqn-bar-sI-11-conv}.  
\smallskip 

We next consider the process $\bar{\sI}^{N}_{1,2}(t,\mfa,x)$ and show that 
\begin{equation} \label{eqn-bar-sI-12-conv}
\big\|\bar{\sI}^N_{1,2}(t,\mfa, \cdot) \big\|_1 \to 0, \quad \text{in probability, locally uniformly in $t$ and $\mfa$, as $N \to \infty$.} 
\end{equation}

To check condition (i) of Theorem \ref{thm-DD-conv-x}, we have
\begin{align*}
\big\|\bar{\sI}^N_{1,2}(t,\mfa, \cdot)\|_1 
&= \frac{1}{K^N}\sum_{k=1}^{K^N} \bigg| \frac{K^N}{N}   \sum_{j=A^N_k((t-\mfa)^+)+1}^{A^N_k(t)} \Big( {\bf1}_{\tau^N_{j,k} + \eta_{j,k} >t} - F^c(t-\tau^N_{j,k} )\Big)\bigg|\,,
\end{align*} 
and
\begin{align*}
\E \big[ \big\|\bar{\sI}^N_{1,2}(t,\mfa, \cdot)\|_1^2  \big] 
&=\E \Bigg[ \Bigg( \frac{1}{K^N}\sum_{k=1}^{K^N} \bigg| \frac{K^N}{N}  \sum_{j=A^N_k((t-\mfa)^+)+1}^{A^N_k(t)} \Big( {\bf1}_{\tau^N_{j,k} + \eta_{j,k} >t} - F^c(t-\tau^N_{j,k} )\Big)\bigg| \Bigg)^2 \Bigg] \\
& \le \E \Bigg[  \frac{1}{K^N}\sum_{k=1}^{K^N} \Bigg( \frac{K^N}{N}  \sum_{j=A^N_k((t-\mfa)^+)+1}^{A^N_k(t)} \Big( {\bf1}_{\tau^N_{j,k} + \eta_{j,k} >t} - F^c(t-\tau^N_{j,k} )\Big)  \Bigg)^2 \Bigg] \\
& =  \E \Bigg[  \frac{1}{K^N}\sum_{k=1}^{K^N}   \Big(\frac{K^N}{N}\Big)^2  \sum_{j=A^N_k((t-\mfa)^+)+1}^{A^N_k(t)} F(t-\tau^N_{j,k} ) F^c(t-\tau^N_{j,k} )  \Bigg] \\
& \le  \frac{K^N}{N} \E \left[  \int_0^1 \int_{(t-\mfa)^+}^{t} F(t-s ) F^c(t-s ) \bar{A}^N(ds,x)dx \right]\, .
\end{align*} 
By Corollary \ref{coro-conv-A} and Lemma \ref{convPort}, we obtain the convergence 
\[
\int_0^1 
\int_{(t-\mfa)^+}^{t} F(t-s ) F^c(t-s ) \bar{A}^N(ds,x)dx
\to \int_0^1 \int_{(t-\mfa)^+}^{t} F(t-s ) F^c(t-s )\bar{A}(ds,x)dx
 \]
 in probability as $N\to \infty$.
 This implies that for any $\ep>0$, 
\[
\sup_{t \in [0,T]}\sup_{\mfa\in [0,\infty']} \P \big( \|\bar{\sI}^{N}_{1,2}(t,\mfa,\cdot)\|_1> \ep \big) \to 0 \qasq N \to \infty. 
\]
Next, to check condition (ii) of Theorem \ref{thm-DD-conv-x}, we need to show that
 for any $\ep>0$, and for any $T, \bar{\mfa}'>0$, as $\delta\to0$,
\begin{align} 
& \limsup_N  \sup_{t\in [0,T]} \frac{1}{\delta}  \P \bigg(  \sup_{u \in [0,\delta]}\sup_{\mfa \in [0,\infty']} \|\bar{\sI}^{N}_{1,2}(t+u,\mfa,\cdot) - \bar{\sI}^{N}_{1,2}(t,\mfa,\cdot) \|_1 > \ep\bigg) \to 0\,,  \label{eqn-sI12-conv-u}\\
& \limsup_N  \sup_{\mfa\in [0,\infty']} \frac{1}{\delta}  \P \bigg(  \sup_{v \in [0,\delta]}\sup_{t \in [0,T]} \|\bar{\sI}^{N}_{1,2}(t,\mfa+v,\cdot) -  \bar{\sI}^{N}_{1,2}(t,\mfa,\cdot) \|_1 > \ep\bigg) \to 0\,. \label{eqn-sI12-conv-v}
\end{align}
To prove \eqref{eqn-sI12-conv-u}, we have
\begin{align*}
& \bar{\sI}^{N}_{1,2}(t+u,\mfa,x) - \bar{\sI}^{N}_{1,2}(t,\mfa,x)  \\
& = \sum_{k=1}^{K^N} \frac{K^N}{N}   \sum_{j=A^N_k((t+u-\mfa)^+)+1}^{A^N_k(t+u)} \Big( {\bf1}_{\tau^N_{j,k} + \eta_{j,k} >t+u} - F^c(t+u-\tau^N_{j,k} )\Big){\bf 1}_{\mathtt{I}_k}(x) \\
& \qquad  - \sum_{k=1}^{K^N} \frac{K^N}{N}  \sum_{j=A^N_k((t-\mfa)^+)+1}^{A^N_k(t)} \Big( {\bf1}_{\tau^N_{j,k} + \eta_{j,k} >t} - F^c(t-\tau^N_{j,k} )\Big){\bf 1}_{\mathtt{I}_k}(x) \\
&=   \sum_{k=1}^{K^N}\frac{K^N}{N}  \sum_{j=A^N_k((t-\mfa)^+)+1}^{A^N_k(t+u)} \Big( {\bf1}_{\tau^N_{j,k} + \eta_{j,k} >t+u} - F^c(t+u-\tau^N_{j,k} )\Big){\bf 1}_{\mathtt{I}_k}(x) \\
& \qquad -  \sum_{k=1}^{K^N} \frac{K^N}{N}   \sum_{j=A^N_k((t-\mfa)^+)+1}^{A^N_k((t+u-\mfa)^+)} \Big( {\bf1}_{\tau^N_{j,k} + \eta_{j,k} >t+u} - F^c(t+u-\tau^N_{j,k} )\Big){\bf 1}_{\mathtt{I}_k}(x) \\
& \qquad  - \sum_{k=1}^{K^N} \frac{K^N}{N}  \sum_{j=A^N_k((t-\mfa)^+)+1}^{A^N_k(t+u)} \Big( {\bf1}_{\tau^N_{j,k} + \eta_{j,k} >t} - F^c(t-\tau^N_{j,k} )\Big){\bf 1}_{\mathtt{I}_k}(x) \\ 
& \qquad  + \sum_{k=1}^{K^N} \frac{K^N}{N}  \sum_{j=A^N_k(t)+1}^{A^N_k(t+u)} \Big( {\bf1}_{\tau^N_{j,k} + \eta_{j,k} >t} - F^c(t-\tau^N_{j,k} )\Big){\bf 1}_{\mathtt{I}_k}(x)\\
& =  -\sum_{k=1}^{K^N} \frac{K^N}{N}  \sum_{j=A^N_k((t-\mfa)^+)+1}^{A^N_k(t+u)} \Big( {\bf1}_{t <\tau^N_{j,k} + \eta_{j,k} \le t+u} - \big(  F^c(t-\tau^N_{j,k} )-F^c(t+u-\tau^N_{j,k} )\big)\Big){\bf 1}_{\mathtt{I}_k}(x) \\
& \qquad -  \sum_{k=1}^{K^N} \frac{K^N}{N}   \sum_{j=A^N_k((t-\mfa)^+)+1}^{A^N_k((t+u-\mfa)^+\wedge t)} \Big( {\bf1}_{\tau^N_{j,k} + \eta_{j,k} >t+u} - F^c(t+u-\tau^N_{j,k} )\Big){\bf 1}_{\mathtt{I}_k}(x) \\
& \qquad  + \sum_{k=1}^{K^N} \frac{K^N}{N}  \sum_{j=A^N_k(t)+1}^{A^N_k(t+u)} \Big( {\bf1}_{\tau^N_{j,k} + \eta_{j,k} >t} - F^c(t-\tau^N_{j,k} )\Big){\bf 1}_{\mathtt{I}_k}(x)\,.
\end{align*} 
Thus we obtain 
\begin{align}\label{eqn-sI12-conv-u-p1} 
& \big\|\bar{\sI}^{N}_{1,2}(t+u,\mfa,\cdot) - \bar{\sI}^{N}_{1,2}(t,\mfa,\cdot) \big\|_1 \non \\
& \le  \frac{1}{K^N}\sum_{k=1}^{K^N} \bigg| \frac{K^N}{N}  \sum_{j=A^N_k((t-\mfa)^+)+1}^{A^N_k(t+u)} \Big( {\bf1}_{t <\tau^N_{j,k} + \eta_{j,k} \le t+u} - \big( F^c(t-\tau^N_{j,k}) - F^c(t+u-\tau^N_{j,k} )\big)\Big)\bigg|  \non \\
& \quad +  \frac{1}{K^N} \sum_{k=1}^{K^N} \bigg| \frac{K^N}{N}  \sum_{j=A^N_k((t-\mfa)^+\wedge t)+1}^{A^N_k((t+u-\mfa)^+} \Big( {\bf1}_{\tau^N_{j,k} + \eta_{j,k} >t+u} - F^c(t+u-\tau^N_{j,k} )\Big) \bigg|  \non \\
& \quad  +\frac{1}{K^N}   \sum_{k=1}^{K^N} \bigg| \frac{K^N}{N}  \sum_{j=A^N_k(t)+1}^{A^N_k((t+u)^+\wedge t)} \Big( {\bf1}_{\tau^N_{j,k} + \eta_{j,k} >t} - F^c(t-\tau^N_{j,k} )\Big)\bigg| \non \\
& \le  \frac{1}{K^N}\sum_{k=1}^{K^N} \frac{K^N}{N} \sum_{j=A^N_k((t-\mfa)^+)+1}^{A^N_k(t+u)}  {\bf1}_{t <\tau^N_{j,k} + \eta_{j,k} \le t+u}  \non  \\
& \quad+ \frac{1}{K^N}\sum_{k=1}^{K^N}\frac{K^N}{N}  \sum_{j=A^N_k((t-\mfa)^+)+1}^{A^N_k(t+u)}    \big( F^c(t-\tau^N_{j,k} ) -  F^c(t+u-\tau^N_{j,k} )\big) \non \\
& \quad+ \frac{1}{K^N}   \sum_{k=1}^{K^N}  \big( \bar{A}^N_k(t+u) - \bar{A}^N_k(t) \big)+ \frac{1}{K^N}   \sum_{k=1}^{K^N}  \big( \bar{A}^N_k((t+u-\mfa)^+) - \bar{A}^N_k( (t-\mfa)^+) \big) \,. 
\end{align} 
For the first term on the right, we have
\begin{align*}
  & \E\Bigg[  \Bigg( \sup_{u \in [0,\delta]}\sup_{\mfa \in [0,\infty']} \frac{1}{K^N}\sum_{k=1}^{K^N} \frac{K^N}{N}   \sum_{j=A^N_k((t-\mfa)^+)+1}^{A^N_k(t+u)}  {\bf1}_{t <\tau^N_{j,k} + \eta_{j,k} \le t+u} \Bigg)^2 \Bigg] \\
  & \le  \E\Bigg[  \Bigg(  \frac{1}{K^N}\sum_{k=1}^{K^N} \frac{K^N}{N}  \sum_{j=A^N_k((t-\infty')^+)+1}^{A^N_k(t+\delta)}  {\bf1}_{t <\tau^N_{j,k} + \eta_{j,k} \le t+\delta} \Bigg)^2 \Bigg] \\
  & \le  \E\Bigg[  \frac{1}{K^N}\sum_{k=1}^{K^N}  \Bigg(  \frac{K^N}{N}  \sum_{j=A^N_k((t-\infty')^+)+1}^{A^N_k(t+\delta)}  {\bf1}_{t <\tau^N_{j,k} + \eta_{j,k} \le t+\delta} \Bigg)^2 \Bigg] \\
  & \le 2  \E\Bigg[  \frac{1}{K^N}\sum_{k=1}^{K^N}  \Bigg( \frac{K^N}{N} \int_{(t-\infty')^+}^{t+\delta} \int_0^\infty \int_{t-s}^{t+\delta-s} {\bf1}_{r\le \Upsilon^N(s^-)} \overline{Q}_{k,\ell}(ds,dr,dz)   \Bigg)^2 \Bigg] \\
  & \quad + 2  \E\Bigg[  \frac{1}{K^N}\sum_{k=1}^{K^N}  \Bigg(   \frac{K^N}{N}  \int_{(t-\infty')^+}^{t+\delta} 
  (F(t+\delta-s)-F(t-s)) \Upsilon^N_k(s) ds  \Bigg)^2 \Bigg] \\
  & = 2  \frac{K^N}{N}  \E\Bigg[  \frac{1}{K^N}\sum_{k=1}^{K^N}    \int_{(t-\infty')^+}^{t+\delta} 
  (F(t+\delta-s)-F(t-s)) \bar{\Upsilon}^N_k(s) ds  \Bigg] \\
  & \quad + 2  \E\Bigg[  \frac{1}{K^N}\sum_{k=1}^{K^N}  \Bigg(   \int_{(t-\infty')^+}^{t+\delta} 
  (F(t+\delta-s)-F(t-s))  \bar{\Upsilon}^N_k(s) ds  \Bigg)^2 \Bigg] \\
  & \le 2 \lambda^*C_B C_\beta \frac{K^N}{N}     \int_{(t-\infty')^+}^{t+\delta} 
  (F(t+\delta-s)-F(t-s)) ds \\
  & \quad + 2   (\lambda^*C_BC_\beta)^2  \Bigg(   \int_{(t-\infty')^+}^{t+\delta} 
  (F(t+\delta-s)-F(t-s)) ds  \Bigg)^2\,,
\end{align*} 
where  $Q_{k,\ell}(ds,dr,dz)$ is the PRM on $\RR_+^3$ with mean measure $dsdrF(dz)$  already introduced in the proof of Lemma \ref{lem-mfN1-conv}, and  $\overline{Q}_{k,\ell}(ds,dr,dz)$ is the corresponding compensated PRM, and we have used the bound $\bar\Upsilon^N_k(t)\le \lambda^*C_BC_\beta$. 
The first term on the right goes to zero as $N\to\infty$, and the integral in the second is bounded from above by
\[ \int_{0}^{t}   \big( F(s+\delta) - F(s) \big) ds\le\delta\,,\]
as in \eqref{delta} above.
Thus we obtain that for any $\ep>0$, as $\delta\to0$,
\begin{align*}
  &  \limsup_N  \sup_{t\in [0,T]} \frac{1}{\delta}   \P\Bigg( \sup_{u \in [0,\delta]}\sup_{\mfa \in [0,\infty']} \frac{1}{K^N}\sum_{k=1}^{K^N}  \frac{K^N}{N}  \sum_{j=A^N_k((t-\mfa)^+)+1}^{A^N_k(t+u)}  {\bf1}_{t <\tau^N_{j,k} + \eta_{j,k} \le t+u} >\ep \Bigg) \to 0\,. 
  \end{align*}
For the second term on the right side of \eqref{eqn-sI12-conv-u-p1}, we have
\begin{align*}
  & \E\Bigg[  \Bigg( \sup_{u \in [0,\delta]}\sup_{\mfa \in [0,\infty']} \frac{1}{K^N}\sum_{k=1}^{K^N} \frac{1}{B^N_k}  \sum_{j=A^N_k((t-\mfa)^+)+1}^{A^N_k(t+u)}  \big( F^c(t-\tau^N_{j,k} ) -  F^c(t+u-\tau^N_{j,k} )\big)  \Bigg)^2 \Bigg] \\
  & \le  \E\Bigg[    \frac{1}{K^N}\sum_{k=1}^{K^N}  \Bigg(\frac{K^N}{N}   \sum_{j=A^N_k((t-\infty')^+)+1}^{A^N_k(t+\delta)}  \big( F^c(t-\tau^N_{j,k} ) -  F^c(t+\delta-\tau^N_{j,k} )\big)  \Bigg)^2 \Bigg]  \\
  & \le  2 \E\Bigg[   \frac{1}{K^N}\sum_{k=1}^{K^N} \Big( \frac{K^N}{N}\Big)^2 \bigg( \int_{(t-\infty')^+}^{t+\delta}   \big( F^c(t-s) -  F^c(t+\delta-s)\big) d M^N_{A,k}(s)  \bigg)^2 \Bigg] \\
  & \quad +2   \E\Bigg[   \frac{1}{K^N}\sum_{k=1}^{K^N} \bigg( \int_{(t-\infty')^+}^{t+\delta}   \big( F^c(t-s) -  F^c(t+\delta-s)\big)  \bar{\Upsilon}^N_k(s) ds  \bigg)^2 \Bigg]  \\
  & = 2 \lambda^*C_BC_\beta \frac{K^N}{N}    \int_{(t-\infty')^+}^{t+\delta}   \big( F^c(t-s) -  F^c(t+\delta-s)\big)^2  ds   \\
  & \quad +2 (\lambda^*C_BC_\beta)^2 \bigg( \int_{(t-\infty')^+}^{t+\delta}   \big( F^c(t-s) -  F^c(t+\delta-s)\big)  ds  \bigg)^2\,. 
  \end{align*} 
It is clear that the first term converge to zero locally uniformly in $t$, and  the second term can be treated in the same way above. 
The third and fourth terms on the right side of \eqref{eqn-sI12-conv-u-p1} can be also treated similarly as the last two terms in \eqref{eqn-sI11-conv-u-p1}. Thus, we have shown that \eqref{eqn-sI12-conv-u} holds.

To prove \eqref{eqn-sI12-conv-v}, we have
\begin{align*}
 \bar{\sI}^{N}_{1,2}(t,\mfa+v,x) - \bar{\sI}^{N}_{1,2}(t,\mfa,x) = \sum_{k=1}^{K^N}\frac{K^N}{N}   \sum_{j=A^N_k((t-\mfa-v)^+)+1}^{A^N_k((t-\mfa)^+)} \Big( {\bf1}_{\tau^N_{j,k} + \eta_{j,k} >t} - F^c(t-\tau^N_{j,k} )\Big){\bf 1}_{\mathtt{I}_k}(x) \,,
\end{align*}
and
\begin{align*}
\big\|\bar{\sI}^{N}_{1,2}(t,\mfa+v,\cdot) - \bar{\sI}^{N}_{1,2}(t,\mfa,\cdot)\|_1
& \le \frac{1}{K^N} \sum_{k=1}^{K^N} \bigg| \frac{K^N}{N}  \sum_{j=A^N_k((t-\mfa-v)^+)+1}^{A^N_k((t-\mfa)^+)} \Big( {\bf1}_{\tau^N_{j,k} + \eta_{j,k} >t} - F^c(t-\tau^N_{j,k} )\Big)\bigg|\\
& \le \int_0^1 \Big(\bar{A}^N((t-\mfa)^+,x) - \bar{A}^N((t-\mfa-v)^+,x) \Big)dx\,. 
\end{align*}
Then, we obtain
\begin{align*}
 & \sup_{v \in [0,\delta]}\sup_{t \in [0,T]}  \big\|\bar{\sI}^{N}_{1,2}(t,\mfa+v,\cdot) - \bar{\sI}^{N}_{1,2}(t,\mfa,\cdot)\|_1 \\
 & \qquad  \le \sup_{t \in [0,T]} \int_0^1 \Big(\bar{A}^N((t-\mfa)^+,x) - \bar{A}^N((t-\mfa-\delta)^+,x) \Big)dx\,. 
\end{align*}
Here the upper bound converges in probability to 
\[
\sup_{t \in [0,T]}\int_0^1 \Big(\bar{A}((t-\mfa)^+,x) - \bar{A}((t-\mfa-\delta)^+,x) \Big)dx
\]
which converges to zero as $\delta \to 0$, uniformly in $\mfa$.  
Indeed, the convergence of the $\sup_t$ follows from the fact that the convergence in probability
$\int_0^1\bar{A}^N(t,x)dx\to\int_0^1\bar{A}(t,x)dx$ is locally uniform in $t$, thanks to  Corollary  \ref{coro-conv-A}.
Thus we have proved \eqref{eqn-sI12-conv-v} holds, and hence, the convergence of  $\bar{\sI}^{N}_{1,2}(t,\mfa,x)$ in \eqref{eqn-bar-sI-12-conv}. 
 This completes the proof of the lemma. 
\end{proof}

By the two lemmas above, we can conclude the convergence of $\bar{\sI}^N(t,\mfa,x)$ to $\bar{\sI}(t,\mfa,x)$. 

\begin{prop}  \label{prop-sIn-conv}
Under Assumptions \ref{AS-LLN-1}, \ref{AS-LLN-2} and \ref{AS-lambda},  
\begin{equation} 
\|\bar{\sI}^N(t,\mfa,\cdot) - \bar{\sI}(t,\mfa,\cdot)\|_1 \to 0
\end{equation}
in probability, locally uniformly in $t$ and $\mfa$, as $N \to \infty$, where $ \bar{\sI}(t,\mfa,x) =  \bar{\sI}_0(t,\mfa,x) + \bar{\sI}_1(t,\mfa,x) $, 
  $\bar{\sI}_0$ and $\bar{\sI}_1$ being given respectively by \eqref{eqn-bar-sI-0} and \eqref{eqn-bar-sI-1}. 
\end{prop}

\medskip 

{\bf Completing the proof of Theorem \ref{thm-FLLN}}. 
Given the results in Propositions \ref{prop-conv-S-mfF} and \ref{prop-sIn-conv} and Corollary \ref{coro-conv-A}, the convergence of $\bar{R}^N(t,x)$ and $\bar{I}^N(t,x)$ can be easily established and their limits $\bar{R}(t,x)$ and $\bar{I}(t,x)$ follows directly. 
The second expression of $\bar\Upsilon(t,x)$ in \eqref{eqn-barUpsilon-tx} is obtained from $\bar{\sI}(t,\mfa,x)$ in \eqref{eqn-barsI-tax}, by noting that $\bar\sI_a(t,0,x) = \lim_{\mfa\to0} \frac{ \bar\sI(t,\mfa,x) - \bar\sI(t,0,x)}{\mfa}$.

\section*{Acknowledgement}
We thank the reviewers for the helpful comments that have led to improvements of the presentation of the paper, and the correction of one error.  G. Pang is partly funded by the NSF Grants DMS 2216765 and CMMI 2452829. 

\bigskip
\bibliographystyle{abbrv}
\bibliography{Epidemic-Age-Spatial}

\end{document}